\documentclass[11pt]{article}

\RequirePackage[OT1]{fontenc}
\RequirePackage{amsthm,amsmath,amssymb,mathrsfs}
\RequirePackage[numbers]{natbib}
\RequirePackage[colorlinks,citecolor=blue,urlcolor=blue]{hyperref}
\usepackage{amstext} 
\usepackage{array}   
\usepackage{enumitem} 
\setlist[itemize]{noitemsep} 
\allowdisplaybreaks
\usepackage{multirow}
\usepackage[flushleft]{threeparttable}
\usepackage{verbatim}
\usepackage{hyperref}
\usepackage{bm,bbm}
\usepackage{longtable}
\usepackage{rotating}
\usepackage{caption}
\usepackage{indentfirst}
\usepackage{subcaption}
\usepackage{adjustbox}

\usepackage{anysize}

\marginsize{1.1in}{1.1in}{0.55in}{0.8in} %

\numberwithin{equation}{section}
\theoremstyle{plain}
\newtheorem{theorem}{Theorem}[section]
\newtheorem{lemma}[theorem]{Lemma}
\newtheorem{proposition}{Proposition}[section]

\theoremstyle{remark}
\newtheorem{remark}{Remark}

\theoremstyle{remark}\newtheorem{assumption}{Assumption}
\DeclareMathOperator*{\argmin}{arg\,min}

\renewcommand{\d}{\mathrm{d}}
\newcommand{\RR}{\mathbb{R}}
\newcommand{\N}{\mathrm{N}}
\renewcommand{\P}{\mathrm{P}}
\newcommand{\E}{\mathrm{E}}
\newcommand{\NN}{\mathbb{N}}
\newcommand{\ZZ}{\mathbb{Z}}
\newcommand{\LL}{\mathbb{L}}
\newcommand{\HH}{\mathbb{H}}
\newcommand{\bj}{\bm{j}}
\newcommand{\bJ}{\bm{J}}

\newcommand{\by}{\bm{y}}
\newcommand{\bX}{\bm{X}}
\newcommand{\bx}{\bm{x}}
\newcommand{\bu}{\bm{u}}
\newcommand{\bv}{\bm{v}}
\newcommand{\cF}{\mathcal{F}}
\newcommand{\cJ}{\mathcal{J}}
\newcommand{\cL}{\mathcal{L}}

\newcommand{\Ind}{\mathbbm{1}}
\newcommand{\ceil}[1]{\left\lceil #1 \right\rceil}
\newcommand{\abs}[1]{| #1 |}
\newcommand{\pnorm}[2]{\lVert#1\rVert_{#2}}
\newcommand{\trans}{^{\tiny{\mathrm{T}}}}

\begin{document} 

\title{Coverage of Credible Intervals in Bayesian Multivariate Isotonic Regression
\thanks{
The research was supported in part by NSF Grant number DMS-1916419.
}}

\author{Kang Wang \thanks{Department of Statistics, North Carolina State Universtiy. Email: kwang22@ncsu.edu}
\and
Subhashis Ghosal\thanks{Department of Statistics, North Carolina State Universtiy. Email: sghosal@stat.ncsu.edu}}

\date{}

\maketitle
\begin{abstract}
\noindent
    We consider the nonparametric multivariate isotonic regression problem, where the regression function is assumed to be nondecreasing with respect to each predictor. Our goal is to construct a Bayesian credible interval for the function value at a given interior point with assured limiting frequentist coverage. A natural prior on the regression function is given by a random step function with a suitable prior on increasing step-heights, but the resulting posterior distribution is hard to analyze theoretically due to the complicated order restriction on the coefficients. We instead put a prior on unrestricted step-functions, but make inference using the induced posterior measure by an ``immersion map'' from the space of unrestricted functions to that of multivariate monotone functions. This allows maintaining the natural conjugacy for posterior sampling. A natural immersion map to use is a projection with respect to a distance function, but in the present context, a block isotonization map is found to be more useful. The approach of using the induced ``immersion posterior'' measure instead of the original posterior to make inference provides a useful extension of  the Bayesian paradigm, particularly helpful when the model space is restricted by some complex relations. We establish a key weak convergence result for the posterior distribution of the function at a point in terms of some functional of a multi-indexed Gaussian process that leads to an expression for the limiting coverage of the Bayesian credible interval. Analogous to a recent result for univariate monotone functions, we find that the limiting coverage is slightly higher than the credibility, the opposite of a phenomenon observed in smoothing problems. Interestingly, the relation between credibility and limiting coverage does not involve any unknown parameter. Hence by a recalibration procedure, we can get a predetermined asymptotic coverage by choosing a suitable credibility level smaller than the targeted coverage, and thus also shorten the credible intervals.

\vskip 0.3cm


\end{abstract}

\section{Introduction}
\label{intro}

Nonparametric inference often involves a regression function or a density function in modeling. 
Commonly, a smoothness assumption on a function of interest is imposed, but in some applications, qualitative information, such as monotonicity, unimodality, and convexity, on the shape of the function may be available. This leads to a control on the complexity of the space of functions analogous to what a smoothness assumption does, allowing convergence without requiring the latter.
Monotonicity is the simplest and the most extensively studied shape restriction, especially in the univariate case.
In regression analysis, this problem is commonly referred to as isotonic regression when the mean function of the response variable is assumed to be nondecreasing.
Starting from the early works on monotone shape restricted problems, such as \cite{ayer1955, brunk1955}, research on non-Bayesian approaches, mainly on the least squares estimator (LSE) and the nonparametric maximum likelihood estimator (MLE), has been fruitful, see \cite{Grenander1956, barlow1972statistical, Groeneboom1985, robertson1988order}. 
Assuming a non-zero derivative, the pointwise asymptotic distribution of the MLE or the LSE 
turned out to be the rescaled Chernoff distribution, that is, the  minimizer of a quadratically drifted standard two-sided Brownian motion \cite{PrakasaRao1969, Brunk1970, wright1981asymptotic, Groeneboom1985}.
The same limiting distribution can also be found in other problems where monotonicity is implied, such as the monotone hazard rate estimation with randomly right-censored observations in survival analysis \cite{huang1994estimating, huang1995hazard}, and many inverse problems, including the current status model and deconvolution problems \cite{groeneboom2014nonparametric}. Global properties of  shape-restricted estimators were also studied extensively \cite{Groeneboom1985, Kulikov2005}. The convergence rates and limiting distributional behaviors of $\LL_p$- and $\LL_{\infty}$-distance between monotone shape-restricted estimators and the true function
was investigated by \cite{durot2007, Durot2012}. Nonasymptotic risk bounds for the LSE under a monotone shape restriction were derived by \cite{zhang2002risk, chatterjee2015risk, bellec2018sharp}. Testing for monotonicity was addressed in \cite{hall2000testing, gijbels2000tests, ghosal2000testing, dumbgen2001multiscale}. 

The Bayesian approach to shape restricted problems was also explored, albeit to a lesser extent.
Neelon and Dunson \cite{neelon2004bayesian} used piecewise linear structures for the regression function, and put monotone restrictions on the priors for the slope values. Cai and Dunson \cite{cai2007bayesian} proposed a linear spline model and added an initial Markov random field prior to the coefficients, and then the monotone constraint was incorporated by considering the relation between the slopes and coefficients. Wang \cite{wang2008bayesian} adopted the free-knot cubic regression spline model, converted the shape restriction to the coefficients, and then projected the unconstrained coefficients with conventional priors to the target set, inducing the constrained priors. Shively {\em et al.} \cite{shively2009bayesian} also used Bayesian splines with constrained normal priors on the coefficients to comply with the monotone shape restriction. 
Lin and Dunson \cite{lin2014bayesian} addressed this problem by using a Gaussian process prior and projected unconstrained posterior samples to the monotone function class by a min-max formula. Such an approach can also be applied in the multivariate case. Chakraborty and Ghosal \cite{chakraborty2020coverage, chakraborty2021convergence,ChakrabortyDensity} also used the idea of projection-posterior, making the investigation of frequentist limiting coverage of credible sets possible. Salomond \cite{Salomond2014monotonicity} used a scale mixture of uniform representation of a nonincreasing density on $[0,\infty)$ and obtained the nearly minimax posterior contraction rate  for both a Dirichlet Process and a finite mixture prior on the mixing distribution. Bayesian tests for monotonicity were developed by \cite{salomond2014adaptive, chakraborty2021convergence,ChakrabortyDensity}. A Bayesian credible interval with assured frequentist coverage for a monotone regression quantile, and an accelerated rate of contraction for it using a two-stage sampling, were obtained by \cite{ChakrabortyRquantile}. 

Multivariate monotone function estimation was also studied in the literature. Non-Bayesian works focused on the construction of the LSE with respect to various partial orderings on the domain; see \cite{barlow1972statistical, robertson1988order}. Only the consistency of the isotonic estimator was known until a recent rise in interest in multivariate shape-restricted problems. In a multivariate isotonic regression model, 
the $\LL_2$-risk of the LSE, respectively for $d=2$ and for a general dimension $d$ was studied by \cite{chatterjee2018matrix,han2019isotonic}. They found that the LSE achieved the optimal minimax rate up to logarithmic factors, and adapted to the parametric rate for a piecewise constant true regression function only when $d\le 2$. 
Han \cite{han2021set} confirmed that the global empirical risk minimizer is indeed rate-optimal in some set structured models even with rapidly diverging entropy integral and thus gave a simpler proof for the optimal convergence rate of the LSE in the multivariate isotonic regression.
Deng and Zhang \cite{deng2020isotonic} investigated a block-estimator proposed by \cite{Fokianos2020} and obtained an $\LL_q$-risk bound. This is minimax rate optimal, and adapts to the parametric rate up to a logarithmic factor when the true regression function is piecewise constant. 
Pointwise  distributional limits for the block-estimator were obtained by \cite{han2020}, which lays the foundation for subsequent inference.

The Bayesian approach to multivariate isotonic regression is much less developed. Saarele and Arjas \cite{saarela2011method} proposed a Bayesian approach to this problem based on marked point processes and resulted in piecewise constant realizations of the regression function. Lin and Dunson \cite{lin2014bayesian} mentioned that the method of projecting Gaussian process posterior samples can also be applied in regression surface case. However, these authors did not study any convergence properties of the Bayesian procedures in higher dimensions. 

The problem of construction of confidence regions in function estimation problems was studied by many authors, mostly in smoothness regimes. For shape restricted problems, confidence regions using limit theory were constructed by 
 \cite{dumbgen2003optimal, dumbgen2004confidence, cai2013adaptive, schmidt2013multiscale}. The natural bootstrap method does not lead to a valid confidence interval for the function value at a point, but a modified bootstrap method works \cite{kosorok2008bootstrapping,sen2010}. Confidence intervals by inverting the acceptance region of a likelihood ratio test for the value of a monotone function at a point were obtained by \cite{Banerjee2001, banerjee2007likelihood, groeneboom2015nonparametric}. This approach has the advantage that no additional nuisance parameters need to be estimated and plugged in the limit distribution. Deng {\em et al.} \cite{deng2020confidence} constructed a confidence interval for multivariate monotone regression from a pivotal limit result for the block-estimator of \cite{deng2020isotonic} relying on the limiting distributional theory by \cite{han2020}.


In this paper, we consider a Bayesian approach to the multivariate monotone regression problem. Our objective is to construct a Bayesian credible interval for the function value at an interior point and obtain its frequentist coverage. As in the univariate problem studied by \cite{chakraborty2020coverage}, we ignore the shape restriction at the prior stage, put a prior on random step-functions without an order restriction, retain posterior conjugacy, and make correction on posterior samples through a map to induce a posterior distribution concentrated on multivariate monotone functions to obtain credible intervals. 
However, unlike in the univariate case, the projection-posterior does not have a limiting distribution for the projection map obtained by minimizing the empirical $\LL_2$-metric since the partial sum process involved the characterization of the empirical  $\LL_2$-projection is not tight in the limit, see \cite{han2020}. 
We also note that the non-Bayesian confidence interval constructed in \cite{deng2020confidence} was also not based on isotonization by distance minimization, but a block max-min procedure. We instead use a map related to the block max-min operation to enforce multivariate monotonicity on posterior samples. As the map immerses a general function into the space of monotone functions, such a map will be referred to as an immersion map, and the induced posterior will be termed an immersion posterior.


The rest of the paper is organized as follows. In the next section, we introduce the notion of an immersion posterior distribution, which will be used to address the problem under study, and is very helpful for similar Bayesian problems with complicated restrictions on the parameter space. In Section \ref{sec:priors}, we introduce the model and assumptions, construct a prior distribution, and describe the posterior distribution used to make inference. Our main results are presented in Section \ref{sec:coverage}. We obtain the immersion posterior procedure, and derive the weak limit of the scaled and centered pointwise immersion posterior distribution. Based on the limit theory, we compute the asymptotic coverage of credible intervals. Numerical results are present in Section \ref{sec:simulation}. We include all the proofs of the main theorems in Section \ref{sec:proof}. Proofs of all auxiliary lemmas and propositions are provided in the supplementary material.

\section{Immersion Posterior}
\label{immersion}

Consider a general statistical model with observation $X\sim P_\theta$, where $\theta\in \Theta_0$. Suppose that the parameter space $\Theta_0$ is a complicated subset of a larger, but simpler to represent, set $\Theta$. This is often the case for shape restricted inference, where structural constraints like monotonicity, convexity, log-concavity, etc. are imposed on a regression function or a density function. In differential equation models, the parameter space is implicitly described as the set of solutions of a system of ordinary or partial differential equations, involving some unknown parameters. In a vector autoregressive process, the set of autoregression coefficients leading to stationary processes may be the parameter space of interest, but it is described by many complicated constraints. Because of the complicated restrictions on $\Theta_0$, a prior for $\theta$ with support on $\Theta_0$ may be hard to construct, and the corresponding posterior may be difficult to compute. More importantly, the corresponding posterior may be hard to analyze from a frequentist orientation. This may be particularly important for studying delicate properties such as the limiting coverage of a Bayesian credible region.

Often, the distribution $P_\theta$ makes sense for any $\theta\in \Theta$, so that $\Theta_0$ can be embedded in $\Theta$ keeping the statistical problem meaningful. For shape restricted models, this becomes the standard nonparametric regression or the density estimation problem. A differential equation model also embeds in a nonparametric regression model. A prior distribution $\Pi$ may be specified on the entire $\Theta$, initially disregarding the restriction of $\theta$ to $\Theta_0$. This is typically a standard problem, and often a conjugate prior distribution can be identified. The resulting posterior distribution $\Pi(\cdot|X)$ thus resides in the whole of $\Theta$, and hence is not appropriate to make inference about $\theta$, which is known to live in $\Theta_0$. The requirement can be met by considering the random measure induced by a mapping $\iota$ from $\Theta$ to $\Theta_0$, in that, we consider the random measure $\Pi^*(B|X)=\Pi (\iota(\theta)\in B|X)$ to make inference on $\theta$. The map $\iota$ immerses $\theta$ into the desirable space $\Theta_0$, and hence will be referred to as the {\it immersion map}. The induced posterior $\Pi^*$ will be referred to as the {\it immersion posterior}. This provides an extension of the Bayesian paradigm since the identity map as the immersion map for the situation $\Theta_0=\Theta$ reduces the immersion posterior to the classical Bayesian posterior.

The approach has been successfully used by several works including \cite{lin2014bayesian, chakraborty2020coverage,chakraborty2021convergence,ChakrabortyDensity,ChakrabortyRquantile} for shape-restricted problems, and by \cite{bhaumik2015bayesian,bhaumik2017efficient,bhaumik2017bayesian, BhaumikPDE} for differential equation models. These authors used a projection map $\mathfrak{p}$ obtained by minimizing a certain distance from the posterior sample to the restricted space, and the resulting induced random measure is called the projection posterior  distribution. The projection map $\mathfrak{p}$ satisfies the appealing property $\mathfrak{p}(\theta)=\theta$ for all $\theta\in\Theta_0$. 

While a projection map with respect to an appropriate distance is a natural choice for an immersion map, the restriction to a projection map is unnecessary for the concept to be used. Depending on the aspect to be studied, there may not be a natural distance associated with it.  This happens, for instance, if we are interested in studying the posterior distribution of the function value at a given point. It is also not necessary for the immersion map $\iota$ to satisfy  $\iota(\theta)=\theta$ for all $\theta\in\Theta_0$. Neither, $\Theta_0$ needs to be a subset of $\Theta$, nor the immersion map needs to be defined all over $\Theta$. All that is needed is that an alternative parameter space $\Theta$ exists where the model distribution $P_\theta$ makes sense, a prior $\Pi$ can be put on $\Theta$ such that the posterior distribution can be computed relatively easily, and the random distribution induced by a map $\iota$ from the support of the posterior distribution $\Pi(\cdot|X)$ to $\Theta_0$ can be analyzed theoretically to establish some desirable properties. In most situations, the family of measures $\{P_\theta: \theta\in \Theta\}$ is dominated, so the support of the posterior distribution $\Pi(\cdot|X)$ is contained in the support of the prior distribution $\Pi$. The immersion map may be allowed to depend on the sample size like a prior distribution may be allowed to depend on the sample size. Even dependence of $\iota$ on the data $X$ may be allowed. Although there is no uniqueness in the choice of the immersion map, the main purpose is to increase flexibility in the posterior measure to achieve a targeted asymptotic frequentist property, such as coverage of a credible region. A choice of an immersion map is therefore guided by a desirable frequentist property. Even if $\Theta_0$ and $\Theta$ coincide, the flexibility of the immersion posterior may be helpful to satisfy a desirable convergence property of the immersion posterior that the classical Bayesian posterior may lack. 

In many applications, the prior distribution may be actually a sequence of prior distributions specified through a sieve indexed by a discrete variable $J=J_n$ depending on the sample size $n$. Let $\Theta_J$ stand for the sieve (typically a finite-dimensional subset of $\Theta$) and $\Pi_J$ stand for the prior at that stage concentrated on $\Theta_J$. Then the computation of the posterior in the unrestricted space reduces to a finite-dimensional computation, often also aided by posterior conjugacy. It is then typical that the immersion map $\iota$ on $\Theta_J$ has the range in $\Theta_0\cap \Theta_J$, so that the computation of the immersion posterior involves finite-dimensional computations only. Most examples from the existing literature, as well as the method used in this paper, fall in this setting.

\section{Notations, Model, Prior, and Posterior Distribution}
\label{sec:priors}

\subsection{Notations} 
We summarize the notations we shall use in this paper. The notations $\RR$, $\NN$, and $\ZZ$ will stand for the real line, the set of natural numbers, and the set of all integers respectively. The positive half-line with and without $0$, and the set of nonnegative integers are respectively denoted by $\RR_{\geq 0}$,$\RR_{>0}$ and, $\ZZ_{\geq 0}$.
Bold Latin or Greek letters will be used to indicate column vectors, and the non-bold font of a letter with a subscript will denote a  coordinate of the corresponding vector. 
For example, $a_i$ is the $i$th coordinate of $\bm{a}\in \RR^d$.
Let $\mathbf{1}$ denote the $d$-dimensional all-one column vector and $\bm{0}$ the all-zero column vector. Let $\bm{A}\trans$ denote the transpose of a matrix or a vector $\bm{A}$. For an arbitrary set $A$, the indicator function will be denoted by $\Ind_A(\cdot)$ and $\#{A}$ will denote the cardinality of a finite set $A$. Let $\lceil a\rceil$ stand for the smallest integer greater than or equal to a real number $a$. The symbol $\lesssim$ will stand for an inequality up to an unimportant constant multiple. 
For two positive real sequences $a_n$ and $b_n$, we also use $a_n \ll b_n$ if $a_n = o(b_n)$.
For $a,b\in\RR$, let $a\wedge b=\min\{a,b\}$ and $a\vee b=\max\{a,b\}$. For $\bm{a}$, $\bm{b}\in\RR^d$, let $\bm{a}\wedge\bm{b}=(a_1\wedge b_1,\ldots, a_d\wedge b_d)\trans$,  $\bm{a}\vee\bm{b}=(a_1\vee b_1,\ldots, a_d\vee b_d)\trans$, and the point-wise product 
$\bm{a}\circ\bm{b}=(a_1b_1,\ldots,a_d b_d)\trans$. For a vector $\bm{a}\in \RR^d$, the Euclidean and the maximum norms are respectively denoted by $\|\bm{a}\|$ and $\pnorm{\bm{a}}{\infty} = \max\{|a_k|: 1\leq k \leq d\}$. 
Let $[\bj_1:\bj_2]=\{ \bj\in\ZZ^d: j_{1,k}\le j_k \le j_{2,k}, \text{ for all }1\le k\le d\}$ stand for the lattice with boundaries $\bj_1,\bj_2\in \ZZ^d$. 

For a multivariate function $f:\RR^d \to \RR$, let $\partial^{l}_k f(\bx) = \partial^l f(\bx)/\partial x_k^l$ for $k\in\{1,\ldots,d\}$ and $l\in\ZZ_{\geq 0}$ at a suitable point $\bx\in\RR^d$. For a multiple index $\bm{l}=(l_1,\dots,l_d)\trans\in\ZZ_{\geq 0}^d$, $\partial^{\bm{l}}=\partial^{l_1}_1\cdots\partial^{l_d}_d$, $
\bm{l}!= l_1!\cdots l_d!$ and $\bx^{\bm{l}}= x_1^{l_1}\cdots x^{l_d}_d$.
We adopt the  coordinate-wise partial ordering on $\RR^d$, that is, $\bx\preceq\by$ if $x_k \leq y_k$ for all $1\leq k \leq d$. We say that a function $f$ on $\RR^d$ is multivariate monotone if $ f(\bx)\leq f(\by)$ for all $\bx \preceq \by$. The class of all multivariate monotone functions on $[0,1]^d$ will be denoted by $\mathcal{M}$. Let $\LL_p[\bm{a},\bm{b}]$, $1\le p\le \infty$, stand for the  
Lebesgue $\LL_p$-space on a multivariate interval $[\bm{a},\bm{b}]$. Convergence in probability under a measure $\P$ is denoted by $\to_{\P}$. Distributional equality will be denoted by $=_d$ and weak convergence by $\rightsquigarrow$.

\subsection{Model}
We observe from  the nonparametric multiple regression model $n$ independent and identically distributed random samples  $\mathbb{D}_n=((\bX_1,Y_1),\ldots,(\bX_n,Y_n))$ 
\begin{align}
\label{model}
Y= f(\bX) + \varepsilon, 
\end{align}
where $Y$ is the response variable, $\bX$ is a $d$-dimensional predictor, and $\varepsilon$ is a random error with mean $0$ and finite variance $\sigma^2$, independent of $\bX$. Instead of assuming any global smoothness condition on $f$, we assume that $f$ is a multivariate monotone function. 
To construct the likelihood function, we assume that $\varepsilon$ is normally distributed, but the actual data generating process need not be so. 

The first assumption is about the local regularity of the true regression function $f_0$ near a point of interest $\bx_0$. This assumption, as in \cite{han2020}, is an essential ingredient to establish the limiting distribution.
\begin{assumption}
\label{assumption_approximation}
Let $f_0\in \mathcal{M}$. For $\bx_0\in (0,1)^d$ and $1\leq k\leq d$, let $\beta_k$ be the order of the first non-zero derivative of $f$ at $\bx_0$ along the $k$-th coordinate, that is, $\beta_k = \min_{l\geq 1}\{l : \partial_k^l f(\bx_0)\neq 0\}$ and  $\beta_k=\infty$ if $\partial^l_k f_0(\bx_0)=0$ for all $l\geq 1$. Without loss of generality, we may assume that $f_0$ depends on its first $s$ arguments locally at $\bx_0$, that is, $1\leq \beta_1,\ldots,\beta_s <\infty$, and that $\beta_{s+1}=\ldots=\beta_d=\infty$ for some $0\leq s\leq d$.
Define an index set $L = \{\bm{l}: 0<\sum_{k=1}^s l_k/\beta_k \leq 1 \text{ and } l_k=0,\text{ for } k=s+1,\ldots, d\}$. 
For a positive sequence $\omega_n\downarrow 0$, set $\bm{r}_n=(\omega_n^{1/\beta_1}, \ldots, \omega_n^{1/\beta_s}, 1,\ldots,1)\trans$.
For any $t> 0$,
\begin{align}
\label{eqn:approx_rate}
    \lim_{\omega_n\downarrow 0} \omega_n^{-1} \sup_{\substack{ \bx\in [0,1]^d, \\ \abs{x_k-x_{0,k}}\leq t r_{n,k}, \\ 1\leq k \leq d }} \big|{f_0(\bx)-f_0(\bx_0)-\sum_{\bm{l}\in L} \frac{\partial^{\bm{l}}f_0(\bx_0)}{\bm{l}!} (\bx-\bx_0)^{\bm{l}} }\big|=0.
\end{align}
\end{assumption}
Assumption \ref{assumption_approximation} considers different convergence rates along different coordinates according to the smoothness levels. All terms in the expansion contribute to approximation rates larger than or equal to $\omega_n$. 
Let 
\begin{align}
\label{L0}
L_0 &=\{ \bm{l}: 0<\sum_{k=1}^s l_k/\beta_k < 1 \text{ and } l_k=0\text{ for } k=s+1,\ldots, d\},\\
\label{L*}
 L^{\ast} &=\{ \bm{l}: \sum_{k=1}^s l_k/\beta_k = 1 \text{ and } l_k=0\text{ for } k=s+1,\ldots, d\}. 
\end{align}
Under Assumption \ref{assumption_approximation}, a unique feature for functions in  $\mathcal{M}$ is that the derivatives of order $\bm{l} \in L_0$ are zero (see Lemma 1 of \cite{han2020}). Only those terms corresponding to the index set $L^{\ast}$ can be nonzero. Thus, the nonzero terms in the expansion of \eqref{eqn:approx_rate} contribute the same approximation rate $\omega_n$. 
However, Assumption \ref{assumption_approximation} cannot eliminate the nonzero mixed derivatives. Additional assumptions will be needed when we want to exclude the mixed derivative terms.

Next, we make the following assumption on the distribution of the covariate $\bX$ and that of the error $\varepsilon$ from the data generating process \eqref{model}. 

\begin{assumption}
\label{assumption_data}
The covariate $\bX$ has a density $g$ such that $a_1\leq g(\bx)\leq a_2$ for all $\bx \in [0,1]^d$ and some $0<a_1\le a_2<\infty$. 
Suppose $g$ is continuous in a neighborhood of the set $\{(x_{0,1},\ldots, x_{0,s}, x_{s+1}, \ldots, x_d): x_k\in [0,1] \text{ for }s+1\le k \le d \}$.
The random error $\varepsilon$, with mean $0$ and variance $\sigma_0^2$, has a finite $2(\sum_{k=1}^s\beta_k^{-1}+1)$-th moment.
\end{assumption}

\subsection{Prior}

We put a prior distribution on $f$ through a sieve of piecewise constant functions with gradually refining intervals of constancy, forming a partition of $[0,1]^d$. For $\bJ \in \ZZ_{>0}^d$, let $I_{\bj}=\prod_{k=1}^d ((j_k-1)/J_k, j_k/J_k]$ be a hyperrectangle in $[0,1]^d$, indexed by a $d$-dimensional vector $\bj$, for $\bj\in[\bm{1}:\bJ] \backslash \{\bm{1}\}$ and $I_{\bm{1}}=\prod_{k=1}^d [0, 1/J_k]$. Then $\{ I_{\bj} \}_{\bj\in [\bm{1}:\bJ]}$ forms a partition of $[0,1]^d$. 
We define a class of piecewise constant functions $\mathcal{K}_{\bJ} := \{f=\sum_{\bj\in [\mathbf{1}:\bJ]} \theta_{\bj}\Ind_{I_{\bj}}: \theta_{\bj}\in \RR\}$. As we follow the immersion posterior approach, we do not initially impose the order restriction.  A prior is imposed on $f=\sum_{\bm{j}\in [\bm{1}:\bm{J}]} \theta_{\bm{j}} \Ind_{I_{\bm{j}}}$ in $\mathcal{K}_{\bJ}$ by giving independent Gaussian priors to $\theta_{\bj}$, namely, 
\begin{equation}
\label{eq:prior}
\theta_{\bj}\sim \N(\zeta_{\bj}, \sigma^2 \lambda_{\bj}^2), \quad \mbox{ independently for all }\bj\in [\bm{1}:\bJ], 
\end{equation}
where $\max_{\bj} |{\zeta_{\bj}}| < \infty$ and $ \min_{\bj}\lambda^2_{\bj} \ge b > 0$.

The values of the prior parameters, $\zeta_{\bj}$ and $\lambda_{\bj}$, will not affect our asymptotic results. However, in practice, when very little prior information is available, it is sensible to choose $\zeta_{\bj}=0$ and $\lambda_{\bj}$ large for all $\bj$. 

\subsection{Posterior distribution}

We use the Gaussian distribution 
\begin{equation}
	\label{working model}
Y_i\sim \mathrm{N}\big(\sum_{\bj\in [\bm{1}:\bm{J}]} \theta_{\bj} \Ind\{\bX_i\in I_{\bm{j}}\}, \sigma^2\big),
	\end{equation}
which leads to, in the unrestricted parameter space, a Gaussian joint likelihood for $(\theta_{\bj}: \bj\in [\mathbf{1}:\bJ])$ without any cross-product terms in the exponent. This gives independent Gaussian posterior distribution for each $\theta_{\bj}$, given $\sigma$, such that by conjugacy, 
 \begin{equation}
 \label{eq:unrestricted posterior}
\theta_{\bj}| \mathbb{D}_n, \sigma \sim \N((N_{\bj}\bar{Y}|_{I_{\bj}}  + \zeta_{\bj}\lambda_{\bj}^{-2})/(N_{\bj}+ \lambda^{-2}_{\bj}), \sigma^2 /(N_{\bj}+\lambda_{\bj}^{-2})), 
 \end{equation}
where $N_{\bj}=\#{\{i:\bX_i\in I_{\bj}\}}$ and $\bar{Y}|_{I_{\bj}} = \sum_{i=1}^n Y_i \Ind\{\bX_{i}\in I_{\bj}\}/N_{\bj}$. 

The parameter $\sigma^2$ can be estimated by maximizing the marginal likelihood function given by
\begin{multline*}
    (2\pi\sigma^2)^{-n/2} \prod_{\bj\in[\bm{1}:\bJ]}(1+\lambda^2_{\bj}N_{\bj})^{-1/2}\\ 
    \times \exp \big[-\frac{1}{2\sigma^2}\big\{ \sum_{i=1}^n \big(Y_i - \sum_{\bj:\bX_i\in I_{\bj}}\zeta_{\bj}\big)^2 - \sum_{\bj\in[\bm{1}:\bJ]}\frac{N_{\bj}^2 (\bar{Y}|_{I_{\bj}} - \zeta_{\bj})^2}{N_{\bj}+\lambda_{\bj}^{-2}} \big\}\big],
\end{multline*}
and the resulting estimator 
\begin{align}\label{sighat mmle}
	\hat{\sigma}_n^2=
	\frac{1}{n}\big[\sum_{i=1}^n \big(Y_i - \sum_{\bj:\bX_i\in I_{\bj}}\zeta_{\bj}\big)^2 - \sum_{\bj\in[\bm{1}:\bJ]}\frac{N_{\bj}^2 (\bar{Y}|_{I_{\bj}} - \zeta_{\bj})^2}{N_{\bj}+\lambda_{\bj}^{-2}} \big],
\end{align}
may be 
plugged in the expression \eqref{eq:unrestricted posterior}. 
Alternatively, in a fully Bayesian framework, we can give $\sigma^2$ an Inverse-Gamma prior $\mathrm{IG}(b_1,b_2)$ with parameters $b_1>0$, $b_2>0$, and obtain that the posterior distribution of $\sigma^2$ is given by $\mathrm{IG}(b_1+n/2,b_2+n\hat{\sigma}^2_n/2)$. It will be shown in Lemma B5 of \cite{kang_supp} that the marginal maximum likelihood estimator of $\sigma^2$ as well as the posterior for $\sigma^2$ concentrate in the neighborhood of its true value $\sigma_0^2$. Then it easily follows that the asymptotic behavior of the posterior distribution of $f$ is identical with that when $\sigma$ is known to be $\sigma_0$. Hence it suffices to study the asymptotic behavior of the posterior distribution given $\sigma$. 

The unrestricted posterior distribution of $f$ given $\sigma$ is induced from \eqref{eq:unrestricted posterior} by the representation $f=\sum_{\bj\in [\mathbf{1}:\bJ]} \theta_{\bj}\Ind_{I_{\bj}}$. To obtain the immersion posterior distribution to make inference, we consider three possible immersion maps.

Define 
\begin{align}
\label{sieve}
\mathcal{M}_{\bJ} = \big\{ f=\sum_{\bj\in [\mathbf{1}:\bJ]} \theta_{\bj}\Ind_{I_{\bj}}: \theta_{\bj}\in \RR \text{ and } \theta_{\bj_1}\leq \theta_{\bj_2}  \text{ if } \bj_1\preceq \bj_2 \big\},
\end{align}
consisting of the coordinatewise nondecreasing functions taking constant values on every $I_{\bj}$. 

Based on the isotonization procedure introduced in \cite{Fokianos2020}, consider transformations $\underline{\iota}$ and  $\overline{\iota}$ acting on $f=\sum_{\bj\in [\mathbf{1}:\bJ]} \theta_{\bj}\Ind_{I_{\bj}} \in \mathcal{K}_{\bJ}$ mapping to an element of $\mathcal{M}_{\bJ}$ defined by 
\begin{align}
\label{maximin}
\underline{\iota}(f)(\bx) = \max_{\bj_1\preceq \bj_0(\bx)} \min_{ \substack{\bj_0(\bx) \preceq \bj_{2} \\ N_{[\bj_1:\bj_2]}>0}} \frac{\sum_{\bj\in [\bj_{1}:\bj_{2}]} N_{\bj} \theta_{\bj}}{N_{[\bj_{1}:\bj_{2}]}},\\
\label{minimax}
\overline{\iota}(f)(\bx) = \min_{\bj_0(\bx)\preceq \bj_2} \max_{ \substack{\bj_1 \preceq \bj_0(\bx) \\ N_{[\bj_1:\bj_2]}>0}} \frac{\sum_{\bj\in [\bj_{1}:\bj_{2}]} N_{\bj} \theta_{\bj}}{N_{[\bj_{1}:\bj_{2}]}},
\end{align}
where $\bj_0(\bx)=\ceil{\bx\circ \bJ}$, $N_{[\bj_1:\bj_2]}=\sum_{\bj\in [\bj_1 : \bj_2]} N_{\bj}$,  and $\bx\in [0,1]^d$, for $\bj_1, \bj_2$ in $\ZZ^d$. 
The immersion posterior for inference may be induced by $\iota$ taken to be equal to $\underline{\iota}$ or $\overline{\iota}$ by looking at the induced distribution of 
\begin{align}
    f_* & =  \underline{\iota}(f), \label{immepost1}\\
    f^* & = \overline{\iota}(f). \label{immepost2}
\end{align}
It is obvious that $\iota(f)\in \mathcal{M}_{\bJ}$ and $\iota(f)=f$ if $f\in\mathcal{M}_{\bJ}$ and $N_{\bj}>0$ for all $\bj\in[\mathbf{1}:\bJ]$. Generally, $\underline{\iota}(f)(\bx)\leq \overline{\iota}(f)(\bx)$ for any $\bx \in [0,1]^d$, but this may fail to hold if $N_{\bj}=0$ for some $\bj$, see \cite{deng2020confidence}. To neutralize the effect caused by the order of minimization and maximization, we consider using the average of $\underline{\iota}$ and $\overline{\iota}$. This leads to  another immersion map $\iota=(\underline{\iota}+\overline{\iota})/2$, where $f$ is mapped to 
\begin{align}
    \tilde{f}= (f_* + f^*)/2. \label{immepost3}
\end{align}

The projection map for the univariate case is typically computed by the pool adjacent violator algorithm (cf. Section 2.3 of \citep{barlow1972statistical}), which requires $O(J)$ computations for a function with $J$ steps. The computation of $f_*$ or $f^*$ requires no more than $(\prod_{k=1}^d J_k)^3$ operations by the brute-force search.

\subsection{Effect of the immersion map} 
To see the effect of the immersion map on the posterior distribution of the function value at a point $\bx_0=(0.5, 0.5)\in [0,1]^2$, we conduct a small simulation study and compare the unrestricted and immersion posterior density for a randomly generated sample of three different sizes $n= 100, 200, 500$, and three different regression functions: (i)  $f_0(x_1, x_2) = x_1 + x_2$; (ii) $f_0(x_1, x_2) = \sqrt{x_1 + x_2}$; (iii) $f_0(x_1, x_2) = \Ind\{x_1 < 1/3\} + 2\Ind\{1/3 \leq x_1 < 2/3\} + 3\Ind\{x_1 \geq 2/3\}$. 
We let $X_1$ and $X_2$ be distributed independently and uniformly on $[0,1]$ and error $\varepsilon\sim \N(0, \sigma^2)$ with true value of $\sigma$ to be $0.1$. We choose the number of grid points $J_1=J_2=J=\lceil n^{1/4}\log_{10}n\rceil$. The random heights, $\{\theta_{(j_1, j_2)}: j_1, j_2 \le J\}$, are endowed with the independent Gaussian prior $\N(0, 1000\sigma^2)$. $\sigma^2$ is estimated using the maximum marginal likelihood method. We plot the unrestricted posterior density and the estimated immersion posterior density based on 2,000 posterior samples transformed by the immersion map $(\bar{\iota}+\underline{\iota})/2$ on the same graph.

\begin{figure}
	\centering
	\begin{subfigure}{.32\textwidth}
		\centering
		\includegraphics[width=\textwidth]{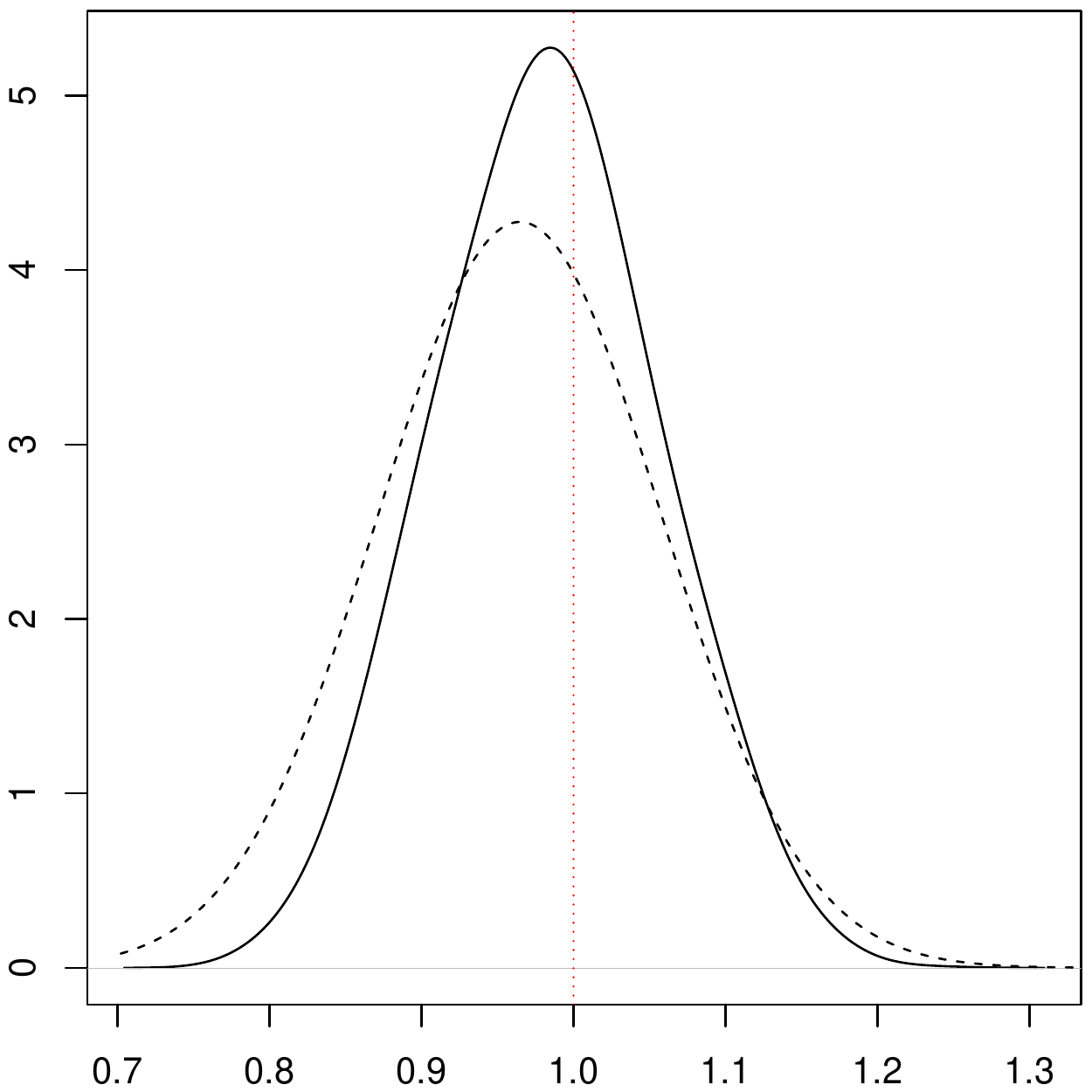}
	\end{subfigure}
	\hfill
	\begin{subfigure}{.32\textwidth}
		\centering
		\includegraphics[width=\textwidth]{./pics/n200fty1.pdf}
	\end{subfigure}
	\hfill
	\begin{subfigure}{.32\textwidth}
		\centering
		\includegraphics[width=\textwidth]{./pics/n200fty1.pdf}
	\end{subfigure}
	
	\begin{subfigure}{.32\textwidth}
		\centering
		\includegraphics[width=\textwidth]{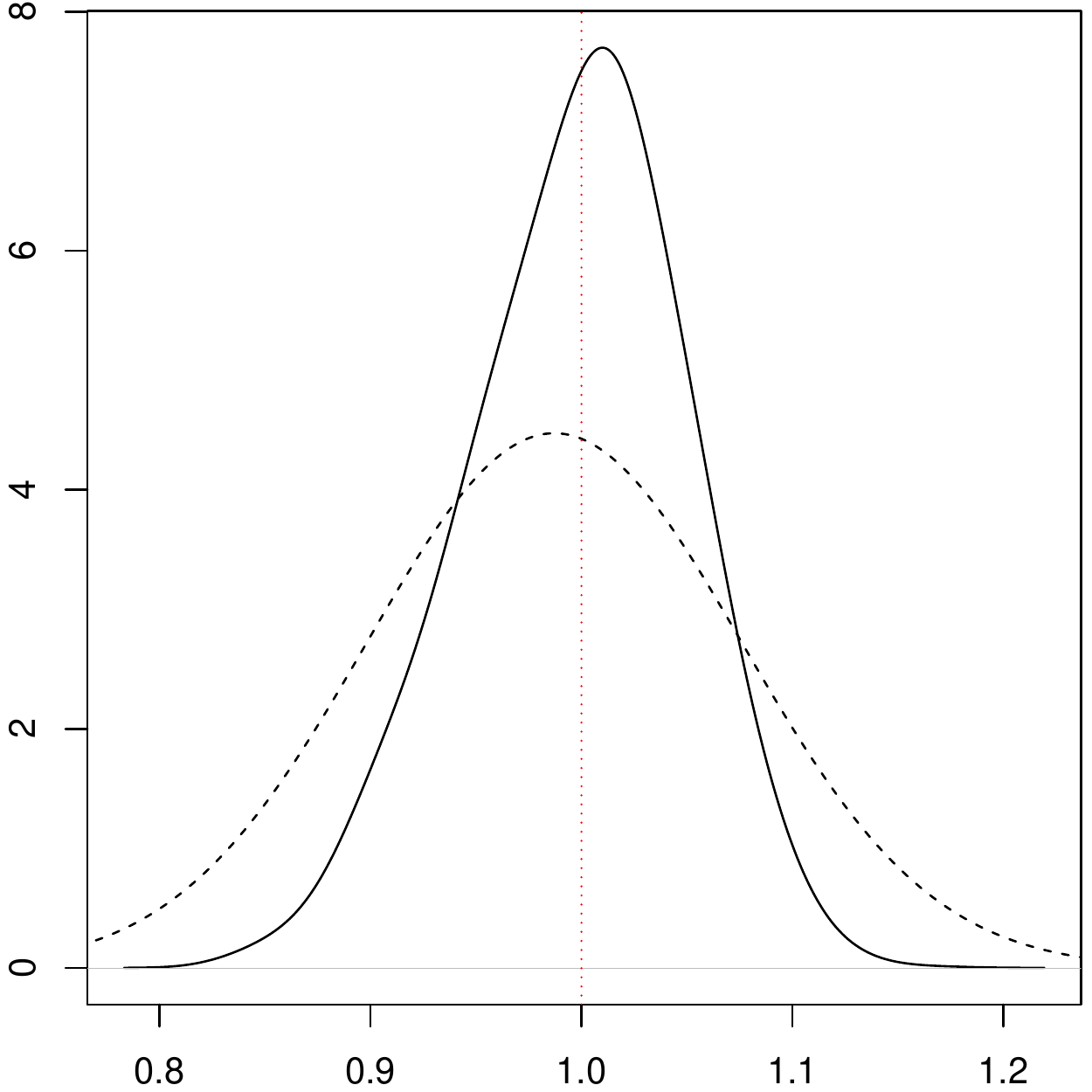}
	\end{subfigure}
	\hfill
	\begin{subfigure}{.32\textwidth}
		\centering
		\includegraphics[width=\textwidth]{./pics/n200fty7.pdf}
	\end{subfigure}
	\hfill
	\begin{subfigure}{.32\textwidth}
		\centering
		\includegraphics[width=\textwidth]{./pics/n200fty7.pdf}
	\end{subfigure}


	\begin{subfigure}{.32\textwidth}
		\centering
		\includegraphics[width=\textwidth]{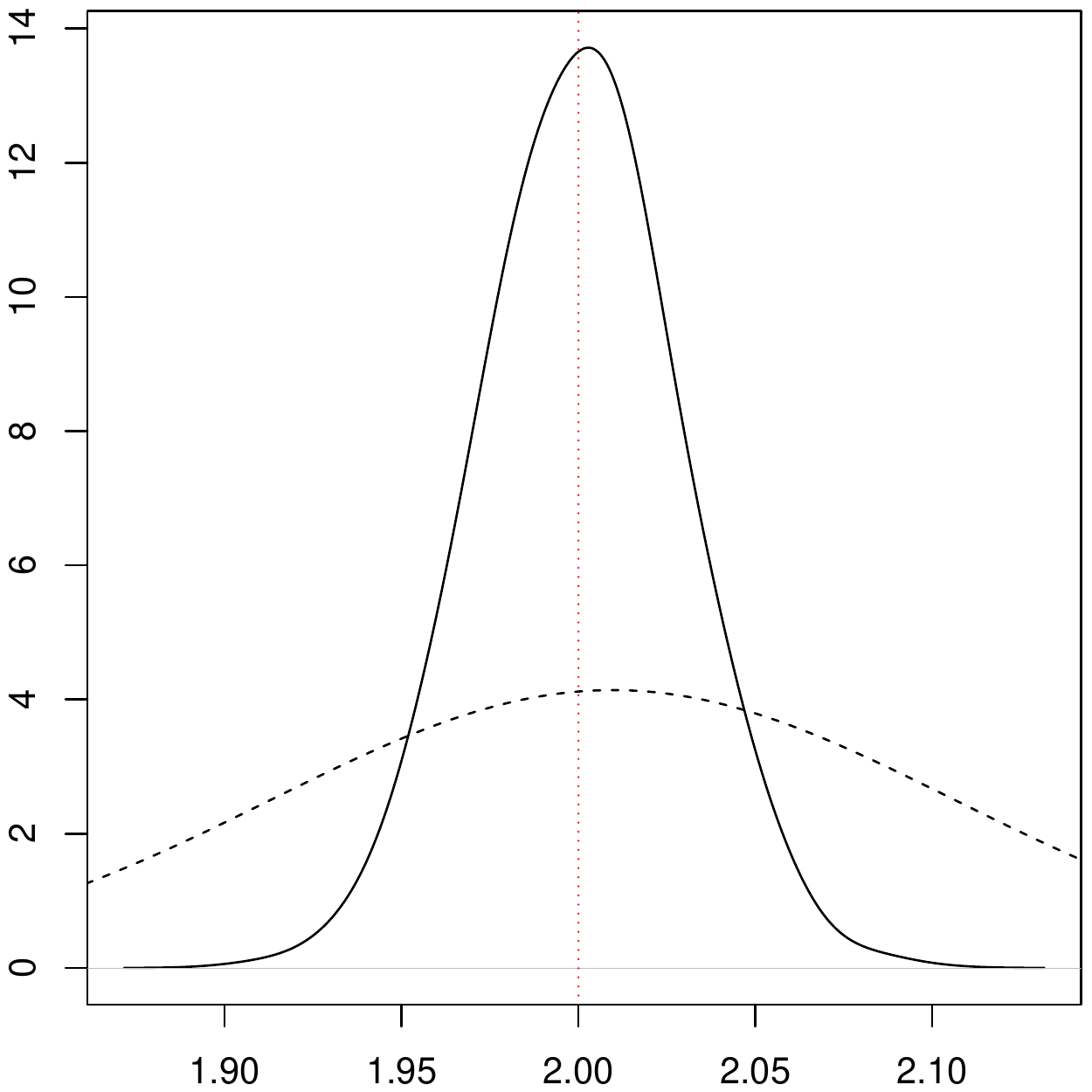}
		\caption*{$n=100$}
	\end{subfigure}
	\hfill
	\begin{subfigure}{.32\textwidth}
		\centering
		\includegraphics[width=\textwidth]{./pics/n200fty9.pdf}
		\caption*{$n=200$}
	\end{subfigure}
	\hfill
	\begin{subfigure}{.32\textwidth}
		\centering
		\includegraphics[width=\textwidth]{./pics/n200fty9.pdf}
		\caption*{$n=500$}
	\end{subfigure}
	\caption{Unrestricted and immersion posterior density functions of $f(\bm{x}_0)$.} The solid black line stands for the immersion posterior density. The black dash line stands for the unrestricted posterior density. We mark the true function value by the red dotted vertical line. The rows correspond to functions (i), (ii), and (iii) respectively, and the columns correspond to sample sizes $n=100, 200$, and $500$.
	\label{fig:postdens}
\end{figure}

We can see from Figure \ref{fig:postdens} that the immersion posterior density functions possess smaller variance across all the instances, although to different extents for different true regression functions and sample sizes, and  the immersion posterior modes are closer to the true value. The effects of the immersion maps $\bar{\iota}$ and $\underline{\iota}$ on the posterior were also found to be similar, and are not reported here.

\section{Coverage of Credible Intervals}
\label{sec:coverage}

Let $\bx_0\in (0,1)^d$ be fixed, and suppose that we want to make inference on $f(\bx_0)$. For a given $0<\gamma<1$, consider a $(1-\gamma)$-credible interval with endpoints the $\gamma/2$ and $(1-\gamma/2)$ quantiles of $f_*(\bx_0)$, $f^*(\bx_0)$, or $\tilde{f}(\bx_0)$ defined in \eqref{immepost1}--\eqref{immepost3}. 
To obtain the limiting frequentist coverage of these credible intervals, we obtain the weak limit of the immersion posterior distributions of $f$ for all three immersion maps $\underline\iota$, $\overline\iota$ and $(\underline\iota+\overline\iota)/2$ at $\bx=\bx_0$. 

Let $H_1$ and $H_2$ be two independent centered Gaussian processes indexed by $(\bm{u},\bm{v})\in\RR_{\geq 0}^d\times \RR_{\geq 0}^d$ with the covariance kernel 
\begin{align}
\label{H covariance}
\prod_{k=1}^s  (u_k\wedge u'_k + v_k\wedge v'_k)D_s(\bm{u}\wedge \bm{u}', \bm{v}\wedge \bm{v}'),
\end{align}
where $D_d(\bm{u},\bm{v})=g(\bx_0)$, where $g$ is the probability density function of $\bX$, and for $s=0,\ldots,d-1$, and 
$D_s(\bm{u},\bm{v})$ is given by 
\begin{align}
\int_{\substack{ x_k\in[(\bx_0-\bm{u})_k,(\bx_0+\bm{v})_k]\cap[0,1] \\ s+1\leq  k \leq d}}
g(x_{0,1},\ldots,x_{0,s}, x_{s+1},\ldots,x_d)\mathrm{d} x_{s+1}\cdots\mathrm{d}x_{d}.
\end{align}
Further, we define a Gaussian process 
\begin{align}
\label{U process}
U (\bm{u},\bm{v}) & =\frac{\sigma_0 H_1(\bm{u},\bm{v})}{\prod_{k=1}^s(u_k+v_k)D_s(\bm{u},\bm{v})}+\frac{ \sigma_0 H_2(\bm{u},\bm{v})}{\prod_{k=1}^s(u_k+v_k)D_s(\bm{u},\bm{v})}\\
& \qquad 
+\sum_{\bm{l}\in L^*} \frac{\partial^{\bm{l}}f_0(x_0)}{(\bm{l}+\bm{1})!} 
\prod_{k=1}^s \frac{v_k^{l_k+1}-(-u_k)^{l_k+1}}{u_k+v_k}
\nonumber
\end{align}
indexed by $(\bm{u},\bm{v})\in\RR_{\geq 0}^d\times \RR_{\geq 0}^d$,
and its functionals  
\begin{align}
Z_* =\sup_{\substack{\bm{u}\succeq \bm{0} \\ u_k \leq x_{0,k} \\ s+1\leq k \leq d}}\inf_{\substack{\bm{v} \succeq\bm{0}\\  v_k \leq 1- x_{0,k} \\ s+1 \leq k \leq d}} U(\bm{u},\bm{v}),\qquad 
Z^* =\inf_{\substack{\bm{v} \succeq\bm{0}\\  v_k \leq 1- x_{0,k} \\ s+1 \leq k \leq d}}\sup_{\substack{\bm{u}\succeq \bm{0} \\ u_k \leq x_{0,k} \\ s+1\leq k \leq d}} U(\bm{u},\bm{v}).
\end{align}

The following result describes the asymptotic behavior of the normalized immersion posterior distributions of $f(\bx_0)$. Recall that $\mathbb{D}_n$ represents the data and $r_{n,k}$ in Assumption \ref{assumption_approximation} is the convergence rate along the $k$-th direction through adjusting the overall rate $\omega_n$ according to the local smoothness levels. 
The weak limit of the normalized immersion posterior distribution function plays a central role in the study of the limiting coverage of the credible intervals based on the immersion posterior quantiles. 

\begin{theorem}
\label{thm: posterior_distribution} 
Let $\omega_n=n^{-{1}/{(2+\sum_{k=1}^{s}\beta_k^{-1})}}$ and let  $\bm{r}_n=(\omega_n^{1/\beta_1},\ldots,\omega_s^{1/\beta_s},1,\ldots,1)\trans$. Suppose that $\bJ$ satisfies $J_k\gg r_{n,k}^{-1}$, for each $k=1,\ldots,d$, and  $\prod_{k=1}^d J_k \ll n\omega_n$. Under Assumptions \ref{assumption_approximation} and \ref{assumption_data}, for any $z\in\RR$, we have
\begin{align}
        \Pi(\omega_n^{-1}(f_*(\bx_0)-f_0(\bx_0))\leq z|\mathbb{D}_n)
        &\rightsquigarrow \mathrm{P}(Z_*
         \leq z |H_1);\\
        \Pi(\omega_n^{-1}(f^*(\bx_0)-f_0(\bx_0))\leq z|\mathbb{D}_n)
        &\rightsquigarrow \mathrm{P}(Z^* 
        \leq z |H_1);\\
        \Pi(\omega_n^{-1}(\tilde{f}(\bx_0)-f_0(\bx_0))\leq z|\mathbb{D}_n) & \rightsquigarrow 
        \mathrm{P}((Z_*+Z^*)/2 \leq z |H_1).
\end{align}
Furthermore, for any $(z_1, z_2)\in \RR^2$, 
\begin{align}
& \Pi\big(\omega_n^{-1}(f_*(\bx_0)-f_0(\bx_0))\leq z_1,\omega_n^{-1}(f^*(\bx_0)-f_0(\bx_0))\leq z_2|\mathbb{D}_n \big)\nonumber\\
& \qquad \qquad \qquad 
 \rightsquigarrow  \mathrm{P}(Z_*\leq z_1, Z^*\leq z_2|H_1).
\end{align}
\end{theorem}

\begin{remark} We make some remarks on Theorem \ref{thm: posterior_distribution}:
\begin{enumerate}
    \item The weak limit is understood in the usual sense for random variables since we consider the limiting behavior of the random probability measure over a fixed set $(-\infty, z]$. We refer to the proof technique of \cite{han2020}, which provides the distributional theory for the block estimator in general multivariate isotonic regression, especially the small and large deviation arguments therein. 
    \item For the choice of $J_k$, the fineness of the partition $\{I_{\bj}\}$, the lower bound $r_{n,k}^{-1}$ for $J_k$ is an essential requirement for Theorem \ref{thm: posterior_distribution}. 
    That eliminates the effect of the roughness of piecewise constant functions in view of the target local contraction rate.
    But the upper bound, $n\omega_n$ in Theorem \ref{thm: posterior_distribution}, is not that necessary for the validity of the weak limit 
    if we set the hyperparameters $\lambda_{\bj}$ large enough. Specifically, when
    $\min \lambda_{\bj}^2 \gg \omega_n^{-1} \sqrt{n}$ and only assume that the second moment of $\varepsilon$ is finite, the conclusion of Lemma \ref{lemma:Anprimeconv} still holds without any upper bound of $\prod_{k=1}^d J_k$. The rest proof of Theorem \ref{thm: posterior_distribution} is affected much without the upper bound for $J_k$ except for the treatment of $\sigma^2$.
    From \eqref{sighat mmle}, we see that $\sigma^2$ would be likely underestimated empirically if $J_k$ is too large, say $\prod_{k=1}^d J_k \geq n$.
    To overcome this, we may obtain the marginal MLE of $\sigma^2$ by using a smaller $J_k$.
    For any $\beta_k \ge 1$ and any $0\le s \le d$,
    we observe that $r_k \le n^{-1/3}$ for all $1\le k \le d$. 
    Without the local smoothness information, we can choose the $J_k \gg n^{1/3}$.
    On the other side, if we admit that $\beta_k = 1$, $1\le k \le d$, is the leading case for the multivariate regression function, we can then choose $J_k \gg n^{-1/(2 + d)}$. 
    In addition, it should be pointed out that $\bJ$ here is not a tuning parameter, and immersion posterior is regulated by the shape restrictions instead of the tuning procedure, like the bandwidth in kernel smoothing. A notable difference from some usual tuning parameters is that the choice of $\bJ$ does not affect the local contraction rate and the distributional theory. 
    
    \item As we contend in the last item, the moment condition for the random error $\varepsilon$ can be relaxed to the second by choosing a large enough $\lambda_{\bj}^2$, which satisfies the condition in the frequentist method \cite{han2020}. It is worth noting that we use a working normal model for the likelihood in the construction of the posterior. The validity of this approach is still ensured even when the model is misspecified.
    
\end{enumerate}
\end{remark}

The covariance kernels of the processes $H_1$ and $H_2$ depend on $g$, and the distributional limits depend on $g$ and the values of $f_0$ and some of its derivatives at $\bx_0$. A considerable simplification happens in some special cases where the parameters appear through a scale parameter in the kernel. It will be seen shortly that this fact has a far-reaching implication in that the limiting coverage of a credible interval constructed from the immersion posterior is free of the unknown parameters of the model. If $L^*$ defined by \eqref{L*} only contains $\beta_k \bm{e}_k$ for $k=1,\ldots,s$, where $\bm{e}_k$ denotes the standard unit vector in $\RR^d$ with one in the $k$th component and zero elsewhere, then the limiting processes in Theorem \ref{thm: posterior_distribution} can be further simplified by the  self-similarity property of the underlying Gaussian processes. A factor depending on $f_0$ and some of its derivatives at $\bx_0$ comes out as a multiplicative constant and the remaining factor is only a known functional of $H_1$ and $H_2$. 
The case $s=d$ stands for the regular case that all directional derivatives of $f_0$ at $\bx_0$ are positive. Then the covariance kernel further simplifies as a completely known function and a factor involving derivatives of the regression function and predictor density $g$. The result is precisely formulated in the result below.

\begin{proposition}
	\label{lemma:separation}
    If $L^*=\{\beta_k\bm{e}_k:{1\leq k \leq s}\}$, then  
\begin{align*}
    &\sup_{\bm{u}\succeq \bm{0}}\inf_{\bm{v}\succeq\bm{0}}\Big\{ 
    \frac{\sigma_0 H_1(\bm{u},\bm{v})}{\prod_{k=1}^s(u_k+v_k)D_s(\bm{u},\bm{v})}+\frac{ \sigma_0 H_2(\bm{u},\bm{v})}{\prod_{k=1}^s(u_k+v_k)D_s(\bm{u}, \bm{v})} \\
    & \qquad +\sum_{k=1}^s \big[\frac{\partial^{\beta_k}_k f_0(\bx_0)}{(\beta_k+1)!} 
     \cdot\frac{v_k^{\beta_k+1}-(-u_k)^{\beta_k+1}}{u_k+v_k}\big] \Big\}\\
    & =_d A_{\bm{\beta}} \cdot \sup_{\bm{u}\succeq \bm{0}}\inf_{\bm{v}\succeq\bm{0}}\Big\{ 
    \frac{H_1(\bm{u},\bm{v})}{\prod_{k=1}^s(u_k+v_k)D_s(\bm{u},\bm{v})} +\frac{  H_2(\bm{u},\bm{v})}{\prod_{k=1}^s(u_k+v_k)D_s(\bm{u},\bm{v})}\\ 
  &\qquad  +\sum_{k=1}^s 
    \frac{v_k^{\beta_k+1}-(-u_k)^{\beta_k+1}}{u_k+v_k} \Big\},
\end{align*}
where 
  $  A_{\bm{\beta}}=\big(\sigma_0^2\prod_{k=1}^s \big(\frac{\partial^{\beta_k}_k f_0(\bx_0)}{(\beta_k+1)!}\big)^{1/\beta_k}\big)^{1/(2+\sum_{k=1}^s \beta_k^{-1})}$.  

Furthermore, if $s=d$, then the above expression further simplifies to 
\begin{align*}
    \tilde{A}_{\bm{\beta}} \sup_{\bm{u}\succeq \bm{0}}\inf_{\bm{v}\succeq\bm{0}}\Big\{ 
    \frac{\tilde{H}_1(\bm{u},\bm{v})}{\prod_{k=1}^d(u_k+v_k)} +\frac{ \tilde{H}_2(\bm{u},\bm{v})}{\prod_{k=1}^d(u_k+v_k)}  
    +\sum_{k=1}^d 
    \frac{v_k^{\beta_k+1}-(-u_k)^{\beta_k+1}}{u_k+v_k} \Big\},
\end{align*}
where 
$    \tilde{A}_{\bm{\beta}}=\big(\frac{\sigma^2_0}{g(\bx_0)}\prod_{k=1}^d \big(\frac{\partial^{\beta_k}_k f_0(\bx_0)}{(\beta_k+1)!}\big)^{1/\beta_k}\big)^{1/({2+\sum_{k=1}^d \beta_k^{-1}})}$, 
and 
$\tilde{H}_1$ and $\tilde{H}_2$ are two independent centered Gaussian processes with covariance kernel given by $\prod_{k=1}^d (u_k\wedge u'_k + v_k\wedge v'_k)$, 
 $(\bm{u}, \bm{v}),(\bm{u}', \bm{v}')\in \RR^d_{\geq 0}\times\RR^d_{\geq 0}$. 

The same conclusion also applies to the $\inf\,\sup$-functional obtained by switching the positions of the supremum and the infimum. 
\end{proposition}

\begin{remark}[Univariate case]\rm 
We specialize to the univariate  case $s=d=1$, with a general  $\beta$, expanding from the case $\beta=1$ studied by \cite{chakraborty2020coverage}. Then
\begin{align}
\tilde{H}_i(u,v)=_d W_i(v) + W_i(-u)=_d W_i(v) - W_i(-u), \quad (u,v)\in \RR_{\geq 0}^2, 
\end{align}
where $W_1,W_2$ are two independent standard two-sided Brownian motions starting from $0$. Observe that the sup-inf functional 
\begin{align*}
&\sup_{u>0}\inf_{v>0}\Big\{ 
    \frac{\tilde{H}_1(u,v)}{u+v}+\frac{ \tilde{H}_2(u,v)}{u+v} +  
    \frac{v^{\beta_1+1}-(-u)^{\beta_1+1}}{u+v} \Big\}\\
&\quad =_d  \sup_{u>0}\inf_{v>0}\Big\{ \frac{(W_1(v)+W_2(v)+v^{\beta_1+1})-(W_1(-u) + W_2(-u) +u^{\beta_1+1})}{v-(-u)} \Big\},
\end{align*}
coincides with the slope of the greatest convex minorant of the process  $W_1(t)+W_2(t)+t^{\beta+1}$. 
By the switching relation [cf. \cite{groeneboom2014nonparametric}, page 56], for any $z\in\RR$,
\begin{align*}
& \mathrm{P}\Big(\tilde{A}_{\beta}\sup_{u>0}\inf_{v>0}\big\{ 
    \frac{\tilde{H}_1(u,v)}{u+v}+\frac{ \tilde{H}_2(u,v)}{u+v} + 
    \frac{v^{\beta+1}-(-u)^{\beta+1}}{u+v} \big\}\leq z\Big)\\
&\quad =\mathrm{P}\Big(
\argmin \{W_1(t)+W_2(t)+t^{\beta+1}-\tilde{A}_{\beta}^{-1}zt: {t\in\RR}\}\geq 0\Big).
\end{align*}
If $\beta=1$, the last display can be further simplified by applying the change of variable, $t = s + z/(2\tilde{A}_1)$, and noting that $W_i(s - a) =_d W_i(s) - W_i(a)$ for some constant $a\in \RR$ and $i=1,2$,
and is equal to 
\[
    \mathrm{P}(2\tilde{A}_{1}
\argmin\{W_1(s)+W_2(s)+s^2: {s\in\RR}\}\leq z),
\]
with $\tilde A_1=(\sigma_0^2 f'(x_0)/(2 g(x_0)))^{1/3}$. This reproduces the result of \cite{chakraborty2020coverage} in view of the fact that the sup-inf (or inf-sup) functional acting on $f$ gives the slope of the greatest convex minorant of $f$ (i.e., the isotonization of $f$) in the univariate case.
\end{remark}

Now we are ready for the evaluation of the limiting coverage of an immersion posterior credible interval for $f(\bx_0)$. Let 
\begin{align}
Q^{(1)}_{n,\gamma}=\inf \{z: \Pi(f_*(\bx_0)\leq z|\mathbb{D}_n)\geq 1-\gamma\}
\end{align} 
stand for the $(1-\gamma)$-quantile of $f_*(\bx_0)$. Similarly, let $Q^{(2)}_{n,\gamma}$ and $Q^{(3)}_{n,\gamma}$ stand for that of $f^*(\bx_0)$ and $\tilde{f}(\bx_0)$ respectively. Let $\tilde{U}(\bm{u},\bm{v})$ stand for the Gaussian process  
\begin{align}
\frac{\tilde{H}_1(\bm{u},\bm{v})}{\prod_{k=1}^d(u_k+v_k)} +\frac{ \tilde{H}_2(\bm{u},\bm{v})}{\prod_{k=1}^d(u_k+v_k)} 
+\sum_{k=1}^d 
\frac{v_k^{\beta_k+1}-(-u_k)^{\beta_k+1}}{u_k+v_k}
\end{align}
indexed by $(\bm{u}, \bm{v})\in \RR^d_{\geq 0}\times\RR^d_{\geq 0}$.

The following result gives the ultimate conclusion of the paper about asymptotic coverage of credible intervals for the function value at a point. 

\begin{theorem}
	\label{coverage}
Under the assumed setup, Assumptions \ref{assumption_approximation} and \ref{assumption_data}, and the condition that $L^*=\{\beta_k\bm{e}_k:{1\leq k \leq s}\}$, the asymptotic coverage of the quantile-based one-sided credible interval $(-\infty, Q_{n,\gamma}^{(1)}]$ is given by 
\begin{align*}
& \P\bigg( \P\bigg(\sup_{\bm{u}\succeq \bm{0}}\inf_{\bm{v}\succeq\bm{0}}\Big\{ 
\frac{H_1(\bm{u},\bm{v})}{\prod_{k=1}^s(u_k+v_k)D_s(\bm{u},\bm{v})}+\frac{ H_2(\bm{u},\bm{v})}{\prod_{k=1}^s(u_k+v_k)D_s(\bm{u},\bm{v})}\\
&
 \qquad\qquad+\sum_{k=1}^s 
\frac{v_k^{\beta_k+1}-(-u_k)^{\beta_k+1}}{u_k+v_k} \Big\} \leq 0 \big| H_1 \Big)\leq 1-\gamma \Big).
\end{align*}
If $Q^{(1)}_{n,\gamma}$ is replaced by $Q^{(2)}_{n,\gamma}$, the above limit is changed by swapping the order of the supremum and infimum operations. If 
$Q^{(1)}_{n,\gamma}$ is replaced by $Q^{(3)}_{n,\gamma}$, the above limit is changed by replacing the expression on the right by the average of the $\sup\,\inf$ and $\inf\,\sup$ operations. 

Moreover, if $s=d$, 
\begin{enumerate}
	\item [{\rm (i)}] $\P_0(f_0(\bx_0)\leq Q^{(1)}_{n,\gamma})  \to \P(Z_B^{(1)} \leq 1-\gamma )$;
	\item [{\rm (ii)}] $\P_0(f_0(\bx_0)\leq Q^{(2)}_{n,\gamma})  \to \P(Z_B^{(2)} \leq 1-\gamma )$;
	\item [{\rm (iii)}] $\P_0(f_0(\bx_0)\leq Q^{(3)}_{n,\gamma})  \to\P(Z_B^{(3)} \leq 1-\gamma )$,  
\end{enumerate}
where $Z_B^{(1)} = \P(\displaystyle\sup_{\bm{u}\succeq \bm{0}}\inf_{\bm{v}\succeq\bm{0}}\tilde{U}(\bm{u},\bm{v})  \leq 0 | \tilde{H}_1)$, 
$Z_B^{(2)} = \P(\displaystyle\inf_{\bm{v}\succeq\bm{0}}\sup_{\bm{u}\succeq \bm{0}}\tilde{U}(\bm{u},\bm{v})  \leq 0 | \tilde{H}_1)$, and 
$Z_B^{(3)} = \P(\frac12 \displaystyle\{ \sup_{\bm{u}\succeq \bm{0}}\inf_{\bm{v}\succeq\bm{0}}\tilde{U}(\bm{u},\bm{v})+\inf_{\bm{v}\succeq\bm{0}}\sup_{\bm{u}\succeq \bm{0}}\tilde{U}(\bm{u},\bm{v}) \} \leq 0 | \tilde{H}_1)$.
\end{theorem}

\begin{proof}
We observe that $f_0(\bx_0)\leq Q^{(1)}_{n,\gamma}$ if and only if 
\begin{align*}
\Pi(f_{\ast}(\bx_0)\leq f_0(\bx_0)|\mathbb{D}_n)=\Pi(\omega_n^{-1}(f_{\ast}(\bx_0)-f_0(\bx_0))\leq 0 |\mathbb{D}_n)\leq 1-\gamma.
\end{align*}
Hence by Theorem~\ref{thm: posterior_distribution} and Proposition~\ref{lemma:separation}, as the multiplicative constant in the limiting process can be dropped because the interval $(-\infty,0]$ remains invariant under a scale-change, the first conclusion follows immediately. The special cases follow from the second part of Proposition~\ref {lemma:separation}.
\end{proof}

\begin{remark}
For $d=1$, $Z_B^{(1)}$, $Z_B^{(2)}$ and $Z_B^{(3)}$ all coincide, and may be simply denoted by $Z_B$ as in \cite{chakraborty2020coverage}. 
\end{remark}

The distributions of $Z_B^{(1)}$ and $Z_B^{(2)}$ are related, as shown next. 

\begin{proposition}
\label{dual tail}
For any $z\in [0,1]$, we have $\P(Z_B^{(1)}\leq z)=\P(Z_B^{(2)}\geq 1-z)$, and $Z_B^{(3)}$ is symmetrically distributed about $1/2$.
\end{proposition}

From Theorem~\ref{coverage} and Proposition~\ref{dual tail}, it follows that the limiting coverage of a one-sided Bayesian credible interval for $f(\bx_0)$ using one of the three proposed immersion posteriors can be evaluated, is free of the true regression function (and also is free of the density $g$ of the predictor if $s=d$, and hence depends only on the credibility level), but in general, need not be equal to the credibility. Nevertheless, a targeted limiting coverage can be obtained by starting with a certain credibility level that can be explicitly computed by back-calculation. As in the univariate monotone problems studied by \cite{chakraborty2020coverage,chakraborty2021convergence}, numerical calculations show that the required credibility to obtain a specific limiting coverage is less than the targeted coverage, the opposite of the phenomenon \cite{cox1993analysis} observed for smoothing problems. However, unlike in the univariate case where the limiting Bayes-Chernoff distribution determining the asymptotic coverage of the credible interval is symmetric, the corresponding random variables $Z_B^{(1)}$ and $Z_B^{(2)}$ for the posterior based on the immersion maps $\underline{\iota}$ and $\overline{\iota}$ appearing in the multivariate case are not symmetric. This has implications for the limiting coverage of a two-sided credible interval, which is more commonly used in practice. For instance, for $0<\gamma<1/2$, a two-sided $(1-\gamma)$-credible interval $[Q_{n,1-\gamma/2},Q_{n,\gamma/2}]$ based on the immersion posterior using the map $\underline{\iota}$, the limiting coverage is given by $\P(Z_B^{(1)}\leq 1-\gamma/2)-\P(Z_B^{(1)}\leq \gamma/2)$. 
The corresponding limit for the immersion posterior using the map $\overline{\iota}$ is $\P(Z_B^{(2)}\leq 1-\gamma/2)-\P(Z_B^{(2)}\leq \gamma/2)$. Interestingly, a separate table for the distribution function of $Z_B^{(2)}$ is not needed, as it can be obtained from that of $Z_B^{(1)}$ in view of Proposition~\ref{dual tail}. The symmetry of $Z_B^{(3)}$, however, implies that the credibility level $1-\gamma$ needed to make the asymptotic coverage of an equal-tailed $(1-\gamma)$-credible interval $1-\alpha$ is obtained by choosing $1-\gamma=1-2 F_{Z_B^{(3)}}^{-1} (\alpha/2)$, which is readily obtained once the cumulative distribution function $F_{Z_B^{(3)}}$ of $Z_B^{(3)}$ is tabulated. 

\section{Numerical Results}
\label{sec:simulation}

\subsection{Distribution of $Z_B$}
In this section, we give tables for the distribution and quantiles for $Z_B$ for the case $d=1$ when $\beta=1,3,5$, and for those of $Z_B^{(1)}, Z_B^{(2)}, Z_B^{(3)}$ for the case $d=2$ when $\bm{\beta}=(1,1),(1,3),(3,3)$. The distributions of these variables are simulated by the Monte Carlo method. The Gaussian processes involved will be generated by a discrete approximation. The quantile table can serve as a recalibration reference to achieve the correct asymptotic coverage.

\subsubsection{Case $d=1$}
First, we generate the approximation to Gaussian processes $\tilde{H}_1$ and $\tilde{H}_2$. Let $\tilde{H}$ denote either $\tilde{H}_1$ or $\tilde{H}_2$. To approximate $\tilde{H}$, we generate $14 m$ independent standard Gaussian random variables $\{\zeta_j:  j=1,\ldots, 7 m\}$ and $\{\zeta'_j:  j =1,\ldots, 7 m\}$ for $m=50$. Then $\tilde{H}$ can be approximately represented as 
\begin{align}
    \tilde{H}(u,v)\approx \frac{1}{\sqrt{m}} \big[
    \sum_{j=1}^{\ceil{mu}}\zeta_j+\sum_{j=1}^{\ceil{mv}}\zeta_j'
    \big],
\end{align}
for $u,v\in [0,7]$. Given each $\tilde{H}_1$, we generate $500$ realizations of $\tilde{H}_2$. For each realization, we calculate the sup-inf functional. Then the proportion of non-positive outcomes is a sample value of $Z_B$. We repeat the generation process $50,000$ times to obtain the approximate distribution function of $Z_B$.

\begin{figure}
\centering
\begin{subfigure}{.32\textwidth}
  \centering
  \includegraphics[width=\textwidth]{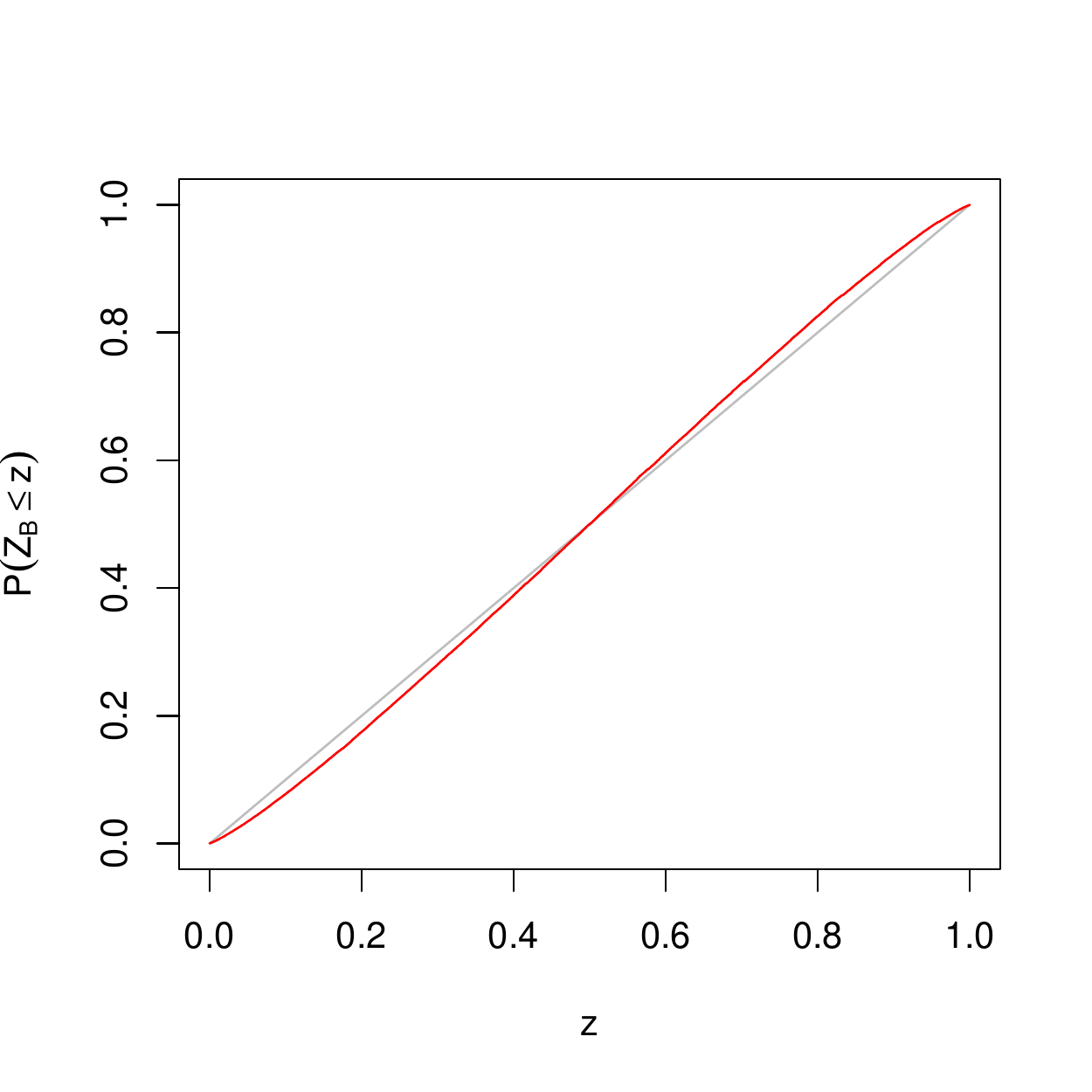}
  \caption{$\beta=1$}
\end{subfigure}
\hfill
\begin{subfigure}{.32\textwidth}
  \centering
  \includegraphics[width=\textwidth]{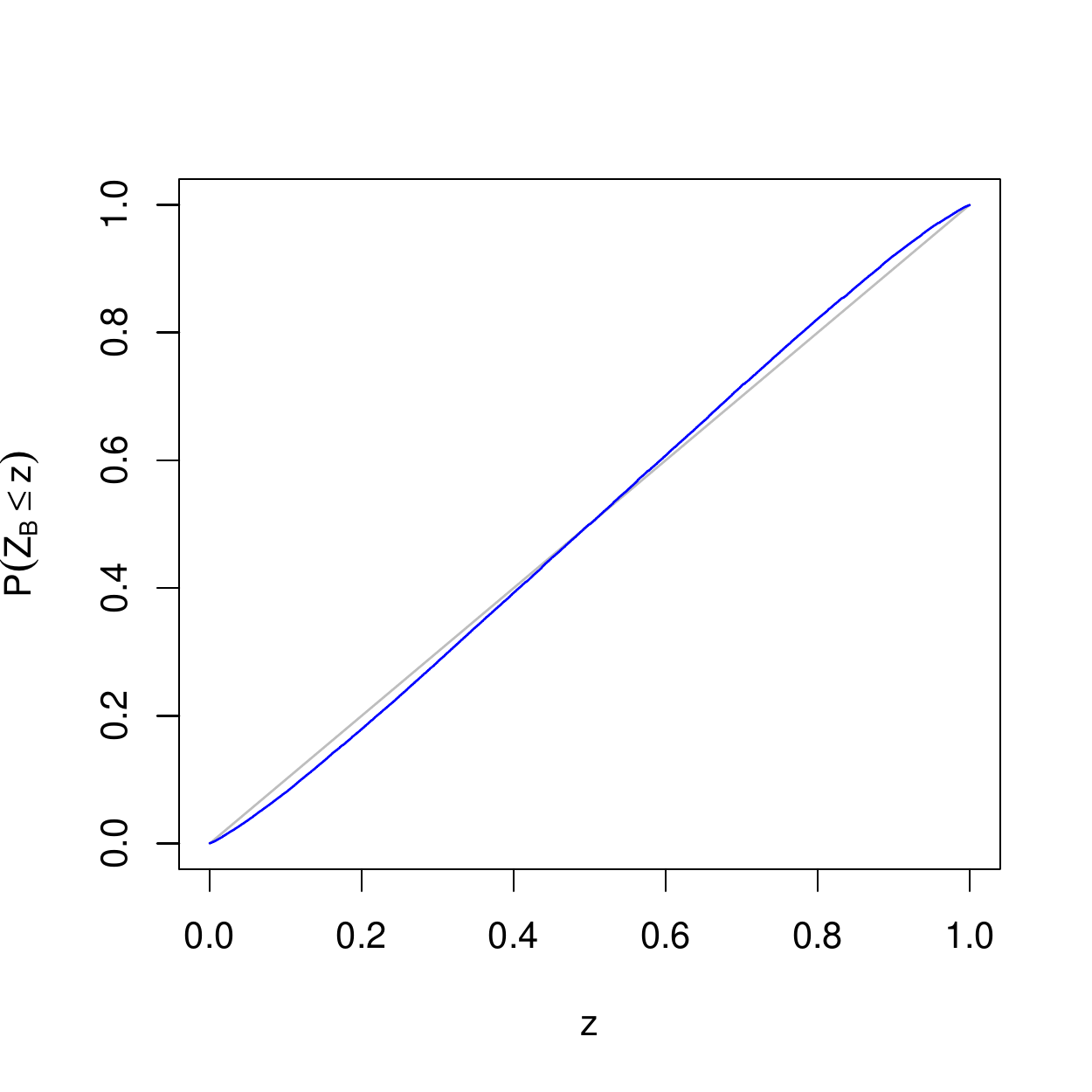}
  \caption{$\beta=3$}
\end{subfigure}
\hfill
\begin{subfigure}{.32\textwidth}
  \centering
  \includegraphics[width=\textwidth]{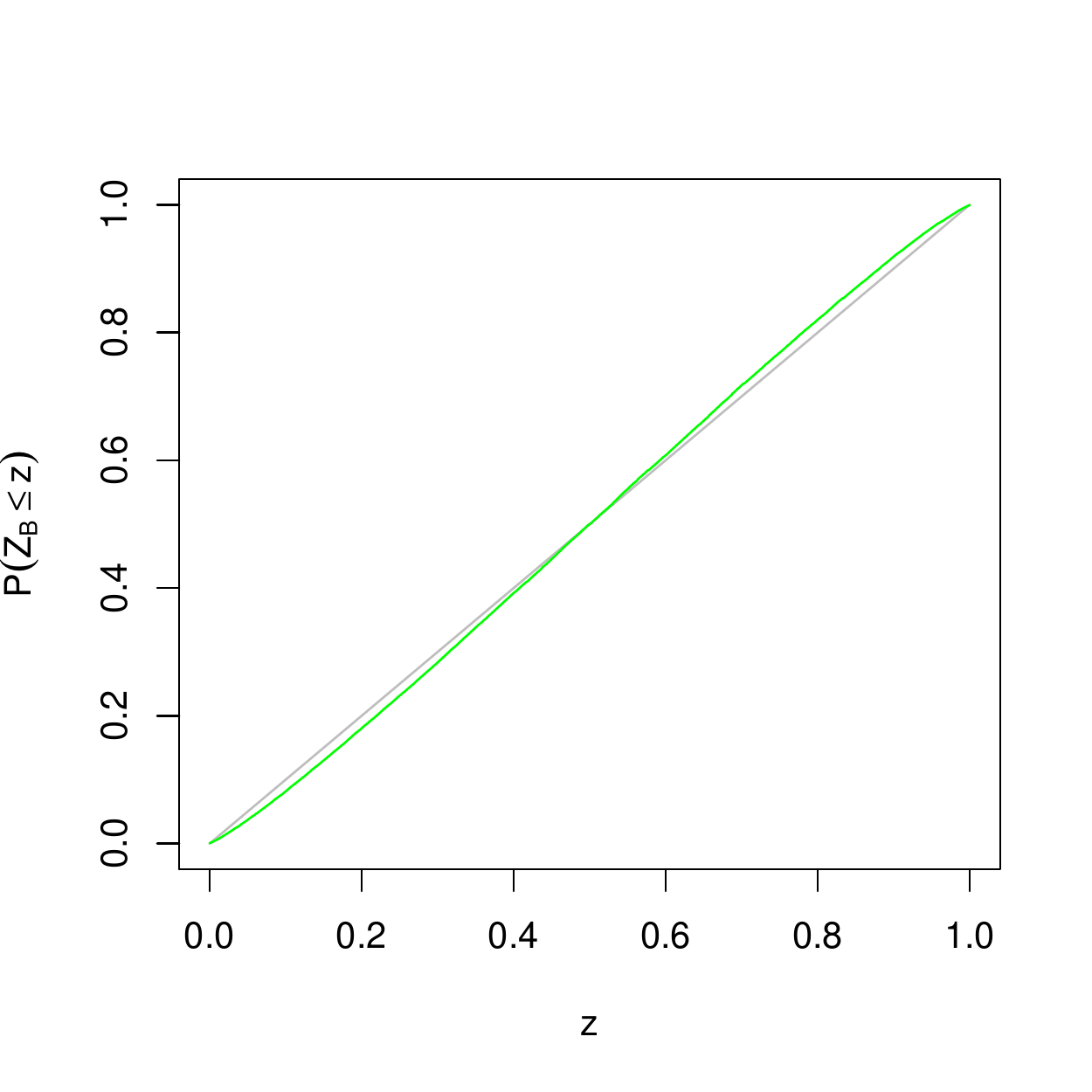}
  \caption{$\beta=5$}
\end{subfigure}
\caption{The distribution function of $Z_B$.}
\label{fig:zb1}
\end{figure}

In Figure \ref{fig:zb1}, we draw the simulated distribution function of $Z_B$ with $\beta = 1, 3$ and $5$.
We give the values of $\P(Z_B \leq z)$ with different smoothness levels for some selected $q$ in Table \ref{tb:d1} and the values of the quantiles of the distribution function of $Z_B$ in Table \ref{tb:d1quan}.

\begin{table}[ht]
	\caption{Values of $\P(Z_B \leq z)$ }
	\label{tb:d1}
	\centering
	\begin{tabular}{|c|ccccccccc|}
		\hline
		$z$ &  0.700 & 0.750 & 0.800 & 0.850 & 0.900 & 0.950 & 0.975 & 0.990 &  0.995\\
		\hline
		$\beta =1$ & 0.719 & 0.772 &  0.826 & 0.875 & 0.923 & 0.965 & 0.985 & 0.994 & 0.997 \\
		$\beta = 3$ & 0.715 & 0.768 & 0.821 & 0.870 & 0.921 & 0.963 & 0.983 & 0.994 & 0.997 \\
		$\beta =5$ & 0.716 & 0.768 & 0.820 & 0.869 & 0.919 & 0.962 & 0.983 & 0.993 & 0.997 \\
		\hline
	\end{tabular}
\end{table}

\begin{table}[ht]
	\caption{Values of $q=\inf\{z: \P(Z_B\leq z)\geq p\}$ }
	\label{tb:d1quan}
	\centering
	\begin{tabular}{|c|ccccccccc|}
		\hline
		$p$ &  0.700 & 0.750 & 0.800 & 0.850 & 0.900 & 0.950 & 0.975 & 0.990 &  0.995\\
		\hline
		$\beta =1$ & 0.683  & 0.730 &  0.777 & 0.825 & 0.878 & 0.932  & 0.964 & 0.994 & 0.997 \\
		$\beta = 3$ & 0.687 & 0.734 & 0.781 & 0.829 & 0.882 & 0.935 & 0.966 & 0.986 & 0.992 \\
		$\beta =5$ & 0.686 & 0.734 & 0.782 & 0.831 & 0.882 & 0.936  & 0.966 & 0.986 & 0.994 \\
	 \hline
\end{tabular}
\end{table}

\subsubsection{Case $d=2$}
 To approximate $\tilde{H}(\bm{u},\bm{v})$, for $\bm{u},\bm{v}\in\RR^2$, we generate 4 random matrices $\zeta^{(1)}$, $\zeta^{(2)}$, $\zeta^{(3)}$ and $\zeta^{(4)}$ with independent standard Gaussian random variables. The dimensions of these 4 matrices are $\ceil{m t_1} \times \ceil{m t_2}$, $\ceil{m s_1} \times \ceil{m t_2}$, $\ceil{m s_1} \times \ceil{m s_2}$ and $\ceil{m t_1} \times \ceil{m s_2}$ for $m=5$ and $t_1=t_2=s_1=s_2=5$. Then $\tilde{H}(\bm{u},\bm{v})$ is approximated by 
\begin{align*}
    \frac{1}{m}\big(
    \sum_{i=1}^{\ceil{mv_1}}\sum_{j=1}^{\ceil{mv_2}}\zeta^{(1)}_{ij}
    + \sum_{i=1}^{\ceil{mu_1}}\sum_{j=1}^{\ceil{mv_2}}\zeta^{(2)}_{ij}
    + \sum_{i=1}^{\ceil{mu_1}}\sum_{j=1}^{\ceil{mu_2}}\zeta^{(3)}_{ij}
    + \sum_{i=1}^{\ceil{mv_1}}\sum_{j=1}^{\ceil{mu_2}}\zeta^{(4)}_{ij}
    \big),
\end{align*}
for $u_1,u_2,v_1,v_2\in [0,5]$.

To get a sample of any one of $Z_B^{(1)}$, $Z_B^{(2)}$ and $Z_B^{(3)}$, 
we generate a sample of $\tilde{H}_1$. Given this sample, we generate 500 realizations of $\tilde{H}_2$. Then we calculate the three functionals defining $Z_B^{(1)}$, $Z_B^{(2)}$ and $Z_B^{(3)}$. The conditional probabilities are approximated by the frequency of negative functional values. This process is repeated 50,000 times for $\bm{\beta}=(1,1), (3,1)$ and $(3,3)$, to estimate the distribution of  $Z_B^{(1)}$, $Z_B^{(2)}$ or $Z_B^{(3)}$.

Since in higher-dimensional cases, $Z_B^{(1)}$ and $Z_B^{(2)}$ are not equal in distribution and their distribution functions are not symmetric about $0.5$, 
we give both the values of $\P(Z_B^{(1)} \leq z)$ and  $\P(Z_B^{(2)} \leq z)$ for some selected $z$ in Table~\ref{tb:d2a11}. The corresponding cumulative distribution functions are plotted in Figures \ref{fig:zbd2}. We present the quantiles of $Z_B^{(3)}$ with different smoothness levels in Table~\ref{tb:d2quan}.

\begin{table}[ht]
	\caption{Values of $\P(Z_B \leq z)$ for various $z$ and $\bm{\beta}$, and $Z_B=Z_B^{(1)},Z_B^{(2)},Z_B^{(3)}$. }
	\label{tb:d2a11}
	\centering
	\begin{tabular}{|c|ccc|ccc|ccc|}
		\hline
		$z$ & \multicolumn{3}{c|}{ $\bm{\beta}=(1,1)$} &  \multicolumn{3}{c|}{ $\bm{\beta}=(3,1)$} & \multicolumn{3}{c|}{ $\bm{\beta}=(3,3)$} \\
		\cline{2-10}
	&	$Z_B^{(1)}$ & $Z_B^{(2)}$ & $Z_B^{(3)}$ 	&	$Z_B^{(1)}$ & $Z_B^{(2)}$ & $Z_B^{(3)}$ 	&	$Z_B^{(1)}$ & $Z_B^{(2)}$ & $Z_B^{(3)}$ \\
		\hline
		0.700 & 0.705 & 0.752 & 0.725 & 0.704 & 0.741 & 0.721 & 0.708 & 0.735 & 0.718 \\
		0.750 & 0.762 & 0.803 & 0.778 & 0.760 & 0.791 & 0.773 & 0.762 & 0.787 & 0.771 \\
		0.800 & 0.817 & 0.851 & 0.832 & 0.814 & 0.842 & 0.827 & 0.817 & 0.838 & 0.825\\
		0.850 & 0.871 & 0.898 & 0.880 & 0.868 & 0.889 & 0.877 & 0.868 & 0.885 & 0.874\\
		0.900 & 0.921 & 0.939 & 0.927  & 0.917 & 0.932 & 0.924 & 0.918 & 0.930 & 0.922\\
		0.950 & 0.966 & 0.975 & 0.968 & 0.964 & 0.971 & 0.966  & 0.964 & 0.970 & 0.965 \\
		0.975 & 0.985 & 0.989 & 0.987  & 0.983 & 0.987 & 0.986  & 0.984 & 0.986 & 0.985 \\
		0.990 & 0.995 & 0.997 & 0.995 & 0.995 & 0.997 & 0.995 & 0.995 & 0.996 & 0.994 \\
		0.995 & 0.997 & 0.998 & 0.998 & 0.998 & 0.998 & 0.998 & 0.997 & 0.998 & 0.997 \\
		\hline
	\end{tabular}
\end{table}

\begin{figure}[ht]
\centering
\includegraphics[scale=0.71]{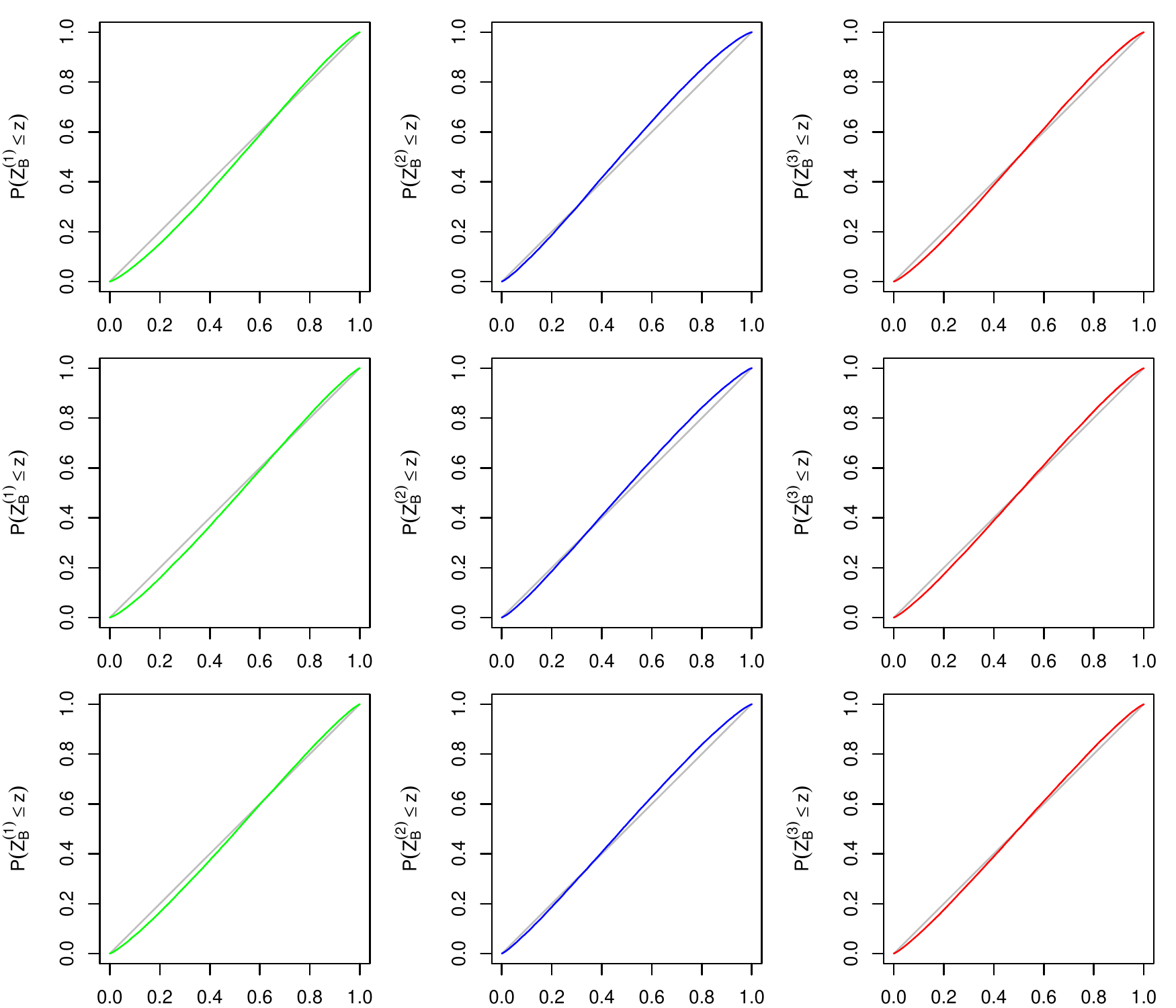}
\centering
\caption{Distribution functions of $Z_B^{(1)}$, $Z_B^{(2)}$ or $Z_B^{(3)}$.}{The three plots in the first row are for $\bm{\beta}=(1,1)$; the second row is for $\bm{\beta}=(3,1)$; the last row is for
$\bm{\beta}=(3,3)$.}
\label{fig:zbd2}
\end{figure}

\begin{table}[ht]
  \caption{Values of $q=\inf\{z: \P(Z_B\leq z)\geq p\}$ for various $p$, and $Z_B=Z_B^{(1)},Z_B^{(2)},Z_B^{(3)}$. }
  \label{tb:d2quan}
  \centering
  \begin{tabular}{|c|ccc|ccc|ccc|}
		\hline
		$p$ & \multicolumn{3}{c|}{ $\bm{\beta}=(1,1)$} &  \multicolumn{3}{c|}{ $\bm{\beta}=(3,1)$} & \multicolumn{3}{c|}{ $\bm{\beta}=(3,3)$} \\
		\cline{2-10}
	&	$Z_B^{(1)}$ & $Z_B^{(2)}$ & $Z_B^{(3)}$ 	&	$Z_B^{(1)}$ & $Z_B^{(2)}$ & $Z_B^{(3)}$ 	&	$Z_B^{(1)}$ & $Z_B^{(2)}$ & $Z_B^{(3)}$ \\
		\hline
		0.700 & 0.697 & 0.653 & 0.677 & 0.699 & 0.665 & 0.681 & 0.695 & 0.669 & 0.684 \\
		0.750 & 0.741 & 0.699 & 0.724 & 0.743 & 0.711 & 0.728 & 0.741 & 0.715 & 0.732 \\
		0.800 & 0.787 & 0.749 & 0.771 & 0.789 & 0.759 & 0.776 & 0.787 & 0.763 & 0.778\\
		0.850 & 0.833 & 0.801 & 0.819 & 0.835 & 0.811 & 0.823 & 0.833 & 0.815 & 0.825\\
		0.900 & 0.881 & 0.855 & 0.872 & 0.883 & 0.865 & 0.876 & 0.883 & 0.869 & 0.878\\
		0.950 & 0.933 & 0.917 & 0.928 & 0.937 & 0.925 & 0.931 & 0.935 & 0.927 & 0.933 \\
		0.975 & 0.963 & 0.951 & 0.959 & 0.965 & 0.957 & 0.962 & 0.965 & 0.959 & 0.964 \\
		0.990 & 0.983 & 0.977 & 0.982 & 0.985 & 0.981 & 0.984 & 0.985 & 0.983 & 0.984 \\
		0.995 & 0.991 & 0.987 & 0.990 & 0.991 & 0.989 & 0.992 & 0.993 & 0.991 & 0.992 \\
		\hline
	\end{tabular}

\end{table}

\subsection{Comparison with Deng, Han and Zhang's method}
For pointwise inference in multivariate isotonic regression, Deng {\em et al.} \cite{deng2020confidence} constructed the confidence interval by the asymptotic distribution of a pivotal statistic. Their method is referred to as DHZ in the following.
Let $\hat{\bu}(\bx_0)$ and $\hat{\bv}(\bx_0)$ be such that
\begin{align*}
    & \hat{f}^{-}(\bx_0) = \max_{\bu\preceq\bx_0} \min_{\substack{\bv \succeq \bx_0\\ \#\{i:\bX_i\in[\bu:\bv]\}>0}} \bar{Y}|_{[\bu:\bv]}=  \min_{\substack{\bv \succeq \bx_0\\ \#\{i:\bX_i\in[\hat{\bu}(\bx_0):\bv]\}>0}} \bar{Y}|_{[\hat{\bu}(\bx_0):\bv]},\\
    & \hat{f}^{+}(\bx_0) = \min_{\bv\succeq\bx_0} \max_{\substack{\bu \preceq \bx_0\\ \#\{i:\bX_i\in[\bu:\bv]\}>0}} \bar{Y}|_{[\bu:\bv]}=  \max_{\substack{\bu \preceq \bx_0\\ \#\{i:\bX_i\in[\bu:\hat{\bv}(\bx_0)]\}>0}} \bar{Y}|_{[\bu:\hat{\bv}(\bx_0)]},
\end{align*}
and $\hat{f}(\bx_0) = (\hat{f}^-(\bx_0)+\hat{f}^+(\bx_0))/2$.
Under the same data generating conditions as in Theorem \ref{coverage} and additionally assuming $\bX$ is uniform distributed, Deng {\em et al.} \cite{deng2020confidence} showed that
\begin{align*}
    \frac{\sigma}{\sqrt{\#\{i:\bX_i\in[\hat{\bu}(\bx_0):\hat{\bv}(\bx_0)]\}}}(\hat{f}(\bx_0) - f(\bx_0)) \rightsquigarrow K_{\bm{\beta}},
\end{align*}
where $K_{\bm{\beta}}$ is a universal distribution depending only on the local regularity $\bm{\beta}$.
Let $1-\gamma\in(0.5,1)$ be the confidence level. They proposed the following confidence interval, referred to as DHZ in the following, for $f_0(\bx_0)$:
\begin{equation}
    [\hat{f}(\bx_0) - \frac{c_{\gamma}\hat{\sigma}}{\sqrt{\#\{i:\bX_i\in [\hat{\bu}(\bx_0): \hat{\bv}(\bx_0)]\}}}, \hat{f}(\bx_0) + \frac{c_{\gamma}\hat{\sigma}}{\sqrt{\#\{i:\bX_i\in [\hat{\bu}(\bx_0): \hat{\bv}(\bx_0)]\}}}],
\end{equation}
where $c_{\gamma}$ is the critical value obtained by simulating the limiting distribution $K_{\bm{\beta}}$ and $\hat{\sigma}$ is a consistent estimator of $\sigma$.

We propose five regression functions: 
\begin{enumerate}
    \item[(1)] $f_1(x_1,x_2)=(x_1+x_2)^2$; \item[(2)]$f_2(x_1,x_2)=\sqrt{x_1+x_2}$; \item[(3)]$f_3(x_1,x_2)= x_1x_2$; \item[(4)]$f_4(x_1,x_2)=e^{x_1+x_2}$; \item[(5)]$f_5(x_1,x_2)=e^{x_1x_2}$.
\end{enumerate}
Set $\varepsilon_i\sim \N(0,1)$ and $X_1,X_2\sim \text{Unif}(0,1)$, mutually independent, for $i=1,\ldots, n$. We consider sample sizes $n= 200, 500, 1000,$ and $2000$. To construct credible intervals we proposed, we choose $J=\ceil{n^{1/3}\log(\log n)}$. We compare our immersion credible interval (IB) and the recalibrated credible interval (IB(adj)) and the DHZ's confidence intervals with two confidence levels $0.95$ and $0.90$. The coverage percentage and the average length are recorded for $2000$ replications. The result is summarized in Table \ref{simu}.

The unadjusted credible intervals tend to overcover the true function value for large sample sizes and the coverage of the recalibrated credible intervals is more accurate to a different extent for different functions.
DHZ's method gives more accurate coverage at the given confidence level when the sample sizes are relatively smaller. However, our credible intervals are generally shorter and have less variation compared to DHZ's confidence interval. The variation of our method across different regression functions may be due to the roughness of the partition we used. In practice, we can set a slightly larger $J$ if the credible intervals can be computed in a reasonable time.

\begin{table}[ht!]
\caption{Coverage percentage (C) and length (L) comparison,} \label{simu}
\centering
\adjustbox{scale = 0.9}{
\begin{tabular}{|c|c|r|cc|cc|cc|}
\hline
\multirow{2}{*}{$f$} & \multirow{2}{*}{level} & \multirow{2}{*}{$n$} & \multicolumn{2}{c|}{IB} & \multicolumn{2}{c|}{IB(adj)} & \multicolumn{2}{c|}{DHZ}\\
\cline{4-9}
&&& C & L & C & L & C & L \\
\hline
\multirow{8}{*}{$f_1$} & \multirow{4}{*}{$0.05$} 
 & 200 &93.6 &0.903(0.145) & 90.0 & 0.805(0.132) & 92.4 & 1.138(0.600)\\
&& 500 & 98.8 & 0.777(0.111) & 97.7& 0.692(0.101) & 95.0 & 0.959(0.490)\\
&&1000 & 97.0 & 0.630(0.086) & 94.4& 0.562(0.078) & 94.8 & 0.827(0.435)\\
&&2000 &97.4 &0.535(0.072) &95.4 &0.476(0.066) & 94.9 & 0.686(0.324)\\
\cline{2-9}
&\multirow{4}{*}{0.10}
 & 200 &88.1 &0.761(0.126) &81.5 &0.668(0.112)& 86.1 & 0.898(0.473)\\
&& 500 &96.9 &0.656(0.097) &94.4 &0.576(0.087) & 90.0 & 0.757(0.387)\\
&&1000 & 92.8 & 0.532(0.075) &88.8 &0.467(0.068) & 88.8 & 0.652(0.343)\\
&&2000 & 94.3 & 0.451(0.063) &90.4 &0.397(0.057) & 89.7 & 0.541(0.256)\\
\hline

\multirow{8}{*}{$f_2$} & \multirow{4}{*}{$0.05$} 
 & 200 & 91.0 & 0.503(0.089) & 87.0 & 0.447(0.081) & 95.2 & 0.722(0.339)\\
&& 500 & 96.6 & 0.380(0.061) &93.8 & 0.338(0.055) & 95.4 & 0.546(0.303)\\
&&1000 & 94.9 & 0.308(0.047) & 91.7 & 0.274(0.043) & 94.8 & 0.439(0.252)\\
&&2000 &95.9 & 0.253(0.039) & 93.3 & 0.225(0.035) & 95.3 & 0.357(0.175)\\
\cline{2-9}
&\multirow{4}{*}{0.10}
 & 200 & 85.0 & 0.423(0.077) & 79.0 & 0.371(0.068) & 89.8 & 0.570(0.268)\\
&& 500 & 92.7 & 0.320(0.052) & 88.0 & 0.280(0.047) & 90.3 & 0.431(0.239)\\
&&1000 & 90.0 & 0.259(0.040) & 84.7 & 0.227(0.036) & 88.9 & 0.346(0.199)\\
&&2000 & 91.8 & 0.213(0.033) & 87.0&0.186 (0.030) & 90.6 & 0.281(0.138)\\
\hline

\multirow{8}{*}{$f_3$} & \multirow{4}{*}{$0.05$} 
 & 200 & 91.8 & 0.476(0.084) & 87.2 &0.423(0.076) & 94.7 & 0.740(0.410)\\
&& 500 & 96.0 & 0.371(0.061) & 93.4 & 0.329(0.055) & 95.0 & 0.532(0.246)\\
&&1000 & 95.0 & 0.293(0.046) & 91.6& 0.260(0.042) & 95.4 & 0.433(0.207)\\
&&2000 & 95.6 & 0.242(0.037)& 93.0 & 0.215(0.034) & 94.8 & 0.353(0.165)\\
\cline{2-9}
&\multirow{4}{*}{0.10}
 & 200 & 84.9& 0.400(0.072)& 79.6& 0.350(0.064) & 89.4 & 0.584(0.323)\\
&& 500 & 92.2 &0.311(0.053)& 87.8& 0.273(0.047) & 89.7 & 0.419(0.194)\\
&&1000 & 90.0 & 0.246(0.040) & 84.7 & 0.216(0.036) & 90.3 & 0.341(0.163)\\
&&2000 & 91.6 & 0.204(0.033)& 86.9 & 0.178(0.029) & 89.8 & 0.279(0.131)\\
\hline

\multirow{8}{*}{$f_4$} & \multirow{4}{*}{$0.05$} 
 & 200 & 97.0 & 1.260(0.188) &94.4 &1.122(0.170) & 89.2 & 1.234(0.621)\\
&& 500 & 99.8 &1.086(0.133)& 99.5 &0.968(0.120) & 92.8 & 1.077(0.538)\\
&&1000 & 99.0& 0.869(0.100)& 97.9& 0.774(0.092) & 94.4 & 0.927(0.437)\\
&&2000 & 99.7 & 0.728(0.083) & 98.2 & 0.649(0.075) & 94.8 & 0.800(0.405)\\
\cline{2-9}
&\multirow{4}{*}{0.10}
 & 200 & 92.8 &1.063(0.162) & 87.9 &0.932 (0.144) & 83.4 & 0.974(0.490)\\
&& 500 & 99.3& 0.917(0.115)& 98.4 &0.805(0.103) & 87.0 & 0.850(0.424)\\
&&1000 & 97.0& 0.733(0.088)& 93.8 & 0.644(0.079) & 89.1 & 0.731(0.345)\\
&&2000 & 97.2 & 0.615(0.072) & 94.9 &0.540(0.064) & 89.3 & 0.631(0.320)\\
\hline

\multirow{8}{*}{$f_5$} & \multirow{4}{*}{$0.05$} 
 & 200 & 94.0 &0.540(0.093)& 90.2 &0.480(0.083) & 95.4 & 0.798(0.398)\\
&& 500 & 97.4 &0.432(0.068) & 95.6& 0.384(0.062) & 94.2 & 0.597(0.316)\\
&&1000 & 96.5 & 0.338(0.051) & 93.4& 0.301(0.047) & 95.1 & 0.491(0.265)\\
&&2000 & 96.9 & 0.280(0.042) & 94.4 &0.249(0.038) & 95.5 & 0.401(0.195)\\
\cline{2-9}
&\multirow{4}{*}{0.10}
 & 200 & 88.2 & 0.454(0.079) & 82.6 &0.398(0.070) & 90.3 & 0.629(0.314)\\
&& 500 & 94.4& 0.364(0.059)& 90.6 &0.319(0.053) & 89.1 & 0.471(0.249)\\
&&1000 & 92.2 & 0.285(0.045)& 87.8 &0.249(0.040) & 91.1 & 0.387(0.209)\\
&&2000 & 93.4 & 0.236(0.036)& 89.2 & 0.207(0.033) & 90.1 & 0.317(0.154)\\
\hline
\end{tabular}
}
\end{table}

\section{Proofs}\label{sec:proof}
\subsection{Proof of Theorem \ref{thm: posterior_distribution}}

As in the univariate case, the immersion posterior is the induced distribution of a functional of a certain stochastic process. 
The proof uses a truncation of the domain to establish the convergence of the underlying stochastic processes. Let 
\begin{align}
	\label{f*value}
 f_{*,c}(\bx_0) = \max_{\substack{c^{-\gamma}\bm{1}\preceq\bm{u}\preceq c\bm{1}, \\ u_k \leq x_{0,k}, \\ s+1\leq k \leq d}}\min_{\substack{c^{-\gamma}\bm{1}\preceq\bm{v}\preceq c\bm{1}, \\ v_k \leq 1-x_{0,k}, \\ s+1\leq k \leq d}} \frac{\sum_{\bj\in[j(-\bm{u}):j(\bm{v})]}N_{\bj}\theta_{\bj}}{N_{[j(-\bm{u}):j(\bm{v})]}}, 
\end{align}
where $\gamma$ is a positive constant to be determined later. We also introduce the notations $W^{\ast}_n=\omega_n^{-1}(f_*(\bx_0)-f_0(\bx_0))$, $W^{\ast}_{n,c}=\omega_n^{-1}(f_{*,c}(\bx_0)-f_0(\bx_0))$, 
\begin{align*}
    W_c &=
    \sup_{\substack{c^{-\gamma}\bm{1}\preceq\bm{u}\preceq c\bm{1}, \\ u_k \leq x_{0,k}, \\ s+1\leq k \leq d}}\inf_{\substack{c^{-\gamma}\bm{1}\preceq\bm{v}\preceq c\bm{1}, \\  v_k \leq 1- x_{0,k}, \\ s+1 \leq k \leq d}}\Bigg\{
        \frac{\sigma_0 H_1(\bm{u},\bm{v})}{\prod_{k=1}^s(u_k+v_k)D_s(\bm{u},\bm{v})}\\
       & \hspace{2em} +\frac{ \sigma_0 H_2(\bm{u},\bm{v})}{\prod_{k=1}^s(u_k+v_k)D_s(\bm{u},\bm{v})}+\sum_{\bm{l}\in L^*} \frac{\partial^{\bm{l}}f_0(x_0)}{(\bm{l}+\bm{1})!} 
        \prod_{k=1}^s \frac{v_k^{l_k+1}-(-u_k)^{l_k+1}}{u_k+v_k}
        \Bigg\},\\
    W &= \sup_{\substack{\bm{u}\succeq \bm{0}, \\ u_k \leq x_{0,k}, \\ s+1\leq k \leq d}}\inf_{\substack{\bm{v} \succeq\bm{0},\\  v_k \leq 1- x_{0,k}, \\ s+1 \leq k \leq d}}\Bigg\{
        \frac{\sigma_0 H_1(\bm{u},\bm{v})}{\prod_{k=1}^s(u_k+v_k)D_s(\bm{u},\bm{v})}\\
     &   \hspace{2em} +\frac{ \sigma_0 H_2(\bm{u},\bm{v})}{\prod_{k=1}^s(u_k+v_k)D_s(\bm{u},\bm{v})} +\sum_{\bm{l}\in L^*} \frac{\partial^{\bm{l}}f_0(x_0)}{(\bm{l}+\bm{1})!} 
        \prod_{k=1}^s \frac{v_k^{l_k+1}-(-u_k)^{l_k+1}}{u_k+v_k}
        \Bigg\}.
\end{align*}

The proof of the theorem is carried out in several steps using Proposition B.1 of \cite{kang_supp}, presented as lemmas below.

\begin{lemma}\label{lemma:weakconv}
Under the conditions of Theorem \ref{thm: posterior_distribution}, for every $c>0$ and $\gamma>0$, $\cL(W^{\ast}_{n,c}|\mathbb{D}_n)$ converges weakly to $\cL(W_c|H_1)$ as random probability measures.
\end{lemma}

\begin{proof}
 For $\bm{t}\in \RR^d$, let $j(\bm{t})=\ceil{(\bx_0 + \bm{t}\circ \bm{r}_n)\circ \bJ}$. For every $\bu,\bv\succeq \bm{0}$, we can write 
\begin{equation}
	\frac{\sum_{\bj\in [j(-\bm{u}):j(\bm{v})]} N_{\bj} \theta_{\bj}}{\sum_{\bj\in [j(-\bm{u}):j(\bm{v})]} N_{\bj}}-f_0(\bx_0)=A_n(\bm{u},\bm{v};\bm{\theta})+A_n'(\bm{u},\bm{v})+B_n(\bm{u},\bm{v}),
\end{equation} 
and $
	W^{\ast}_{n,c}= \displaystyle \max_{c^{-\gamma}\bm{1}\preceq\bu\preceq c\bm{1}}\min_{c^{-\gamma}\bm{1}\preceq\bu\preceq c\bm{1}} \{A_n(\bm{u},\bm{v};\bm{\theta}) + A_n'(\bm{u},\bm{v}) + B_n(\bm{u},\bm{v})\}$, 
where 
\begin{align}
	\label{An}
    A_n(\bm{u},\bm{v};\bm{\theta})&=\omega_n^{-1}\frac{\sum_{\bj\in [j(-\bm{u}):j(\bm{v})]} N_{\bj} (\theta_{\bj}- \E[\theta_{\bj}|\mathbb{D}_n])}{\sum_{\bj\in [j(-\bm{u}):j(\bm{v})]} N_{\bj}},\\
    \label{An'}
    A_n'(\bm{u},\bm{v})&=\omega_n^{-1}\frac{\sum_{\bj\in [j(-\bm{u}):j(\bm{v})]} N_{\bj} ( \E[\theta_{\bj}|\mathbb{D}_n]-\bar{Y}|_{I_{\bj}} ) }{\sum_{\bj\in [j(-\bm{u}):j(\bm{v})]} N_{\bj}},\\
    \label{Bn}
    B_n(\bm{u},\bm{v}) &=\omega_n^{-1}(\overline{Y}|_{I_{[j(-\bm{u}):j(\bm{v})]}}-f_0(\bx_0)).
\end{align}

Since the max-min functional is continuous on the space $\LL_{\infty}([c^{-\gamma}\bm{1}, c \bm{1}]\times[c^{-\gamma}\bm{1}, c \bm{1}])$, 
it suffices to show that $A_n+A'_n+B_n$ converges weakly in $\LL_{\infty}([c^{-\gamma}\bm{1}, c \bm{1}]\times[c^{-\gamma}\bm{1}, c \bm{1}])$, conditional on the data $\mathbb{D}_n$. By Lemma B.2 of \cite{kang_supp} and Lemma \ref{lemma:Anweakconv}, we prove the weak convergence of $A_n$. We show that $A'_n$ converges to zero uniformly in Lemma \ref{lemma:Anprimeconv}. The convergence of $B_n$ is completed by combining Lemma B.2 of \cite{kang_supp}, Lemma \ref{lemma:h1nconv} and Lemma \ref{lemma:b2nconv}.
\end{proof}

\begin{lemma}\label{lemma:Anweakconv}
   Under the conditions of Theorem \ref{thm: posterior_distribution}, for every $c>0$, let $\HH_{2,n}(\bm{u},\bm{v};\bm{\theta})= \omega_n \sum_{\bj\in [j(-\bm{u}):j(\bm{v})]} N_{\bj} (\theta_{\bj}- \E[\theta_{\bj}|\mathbb{D}_n])$. Then $\HH_{2,n}$ converges weakly to a centered Gaussian process $H_2$ in $\LL_{\infty}([\bm{0},c\bm{1}]\times[\bm{0},c\bm{1}])$ for every $c>0$ in $\P_0$-probability. 
\end{lemma}

\begin{proof}
By \eqref{eq:unrestricted posterior}, Lemmas B.2, B.4, and B.5 of \cite{kang_supp}, the covariance kernel of $\HH_{2,n}$ given $(\mathbb{D}_n, \sigma^2_n)$, is given by 
$
 \omega_n^2\sigma^2_n\sum_{\bj\in [j(-\bm{u}\wedge \bm{u}'):j(\bm{v}\wedge \bm{v}')]}{N_{\bj}^2}/{(N_{\bj}+\lambda_{\bj}^{-2})}$, which converges in $\P_0$ probability to $\sigma^2_0 \prod_{k=1}^{s}\left(u_k\wedge u'_k+v_k\wedge v'_k\right)D_s(\bm{u}\wedge \bm{u}', \bm{v}\wedge \bm{v}')$. 
Thus finite-dimensional distributions of $\HH_{2,n}$ converge weakly to those of a centered Gaussian process $\sigma_0 H_2$ in  $\P_0$-probability.

Next we shall show that $\mathcal{L}(\HH_{2,n}(\bm{u},\bm{v};\bm{\theta}):(\bm{u},\bm{v})\in [\bm{0},c\bm{1}]\times[\bm{0},c\bm{1}])$ is tight on $\LL_{\infty}([\bm{0},c\bm{1}]\times[\bm{0},c\bm{1}])$ for any $c>0$ in $\P_0$-probability. In view of Theorem 18.14 of \cite{vandervaart2000asymptotic}, we need to verify that, for every $\epsilon > 0$ and $\eta > 0$, there exists a finite partition $\{ T_p : p\leq K\}$ of $[\bm{0},c\bm{1}]\times [\bm{0},c\bm{1}]$ with $K$ depending only on $\epsilon$ and $\eta$ such that
\begin{align*}
    \P\big(\sup_{(\bm{u}_1,\bm{v}_1),(\bm{u}_2,\bm{v}_2)\in T_p} \{ |\HH_{2,n}(\bm{u}_1,\bm{v}_1)-\HH_{2,n}(\bm{u}_2,\bm{v}_2)|: 
    1\leq p \leq K\}>\epsilon|\mathbb{D}_n\big)< \eta
\end{align*}
with $\P_0$-probability tending to $1$. Let $\delta>0$, to be determined later depending only on $\epsilon$ and $\eta$. Let $0=s_0 < s_1 <\ldots < s_l = c$ with $(s_{t-1},s_t]$ having equal lengths at least $\delta$ and $l\leq 2c/\delta$. 
We choose a partition $\{T_p: p\leq K\}$ of $[\bm{0},c\bm{1}]\times [\bm{0},c\bm{1}]$ to be 
\begin{equation}
\label{partition}
\mathcal{P}(\delta)=\{\prod_{k=1}^{d}(s_{t_k-1},s_{t_k}] \times \prod_{k=1}^{d}(s_{r_k-1},s_{r_k}]: t_k, r_k\in \{1,\ldots, l\}\}, 
\end{equation}
with cardinality $K=\#\mathcal{P}(\delta)=l^{2d}$. It suffices to verify that, for any $p\leq K$,
\begin{align*}
   \P\big(\sup_{(\bm{u}_1,\bm{v}_1),(\bm{u}_2,\bm{v}_2)\in T_p} \{ |\HH_{2,n}(\bm{u}_1,\bm{v}_1)-\HH_{2,n}(\bm{u}_2,\bm{v}_2)| \}>\epsilon|\mathbb{D}_n\big)< \eta \big(\frac{\delta}{2c}\big)^{2d}.
\end{align*}
Let $\mathcal{J}(\bm{u},\bm{v})=[j(-\bm{u}):j(\bm{v})]$.
For $(\bm{u}_1,\bm{v}_1),(\bm{u}_2,\bm{v}_2)$, we write $\HH_{2,n}(\bm{u}_1,\bm{v}_1)-\HH_{2,n}(\bm{u}_2,\bm{v}_2)$ as the difference of the sums of $ \omega_n N_{\bj} (\theta_{\bj}- \E[\theta_{\bj}|\mathbb{D}_n])$ over the sets $\mathcal{J}(\bu_1,\bv_1)\setminus\mathcal{J}(\bu_1\wedge\bu_2,\bv_1\wedge\bv_2)$ and $\mathcal{J}(\bu_2,\bv_2)\setminus\mathcal{J}(\bu_1\wedge\bu_2,\bv_1\wedge\bv_2)$, after canceling out the common terms. Thus its absolute value can be bounded by the sum of the corresponding absolute values over these two index sets. To verify tightness, it then suffices to show that 
\begin{align*}
	\P\big(\max \big\{ \omega_n \big|{\sum_{ \mathcal{J}(\bm{u},\bm{v})\setminus\mathcal{J}(\bm{s}_{\bm{t}-\bm{1}},\bm{s}_{\bm{r}-\bm{1}})} N_{\bj} (\theta_{\bj}- \E[\theta_{\bj}|\mathbb{D}_n])}\big|: (\bm{u},\bm{v})\in T_p\big\}>\frac{\epsilon}{4}\big|\mathbb{D}_n\big)
\end{align*}
is bounded by ${\eta}({\delta}/{(2c)})^{2d}/4$, 
with $T_p=\prod_{k=1}^d (s_{t_k-1}, s_{t_k}] \times \prod_{k=1}^d (s_{r_k-1}, s_{r_k}]$, for any $\bm{s}_{\bm{t}}=(s_{t_1},\ldots,s_{t_d})$ and $\bm{s}_{\bm{r}}=(s_{r_1},\ldots,s_{r_d})$. 

Let $S_{(-j(-\bu),j(\bv))} = \displaystyle\sum_{ \mathcal{J}(\bm{u},\bm{v})\setminus \mathcal{J}(\bm{s}_{\bm{t}-\bm{1}},\bm{s}_{\bm{r}-\bm{1}})} N_{\bj} (\theta_{\bj}- \E[\theta_{\bj}|\mathbb{D}_n])$, 
a collection of random variables indexed by a $2d$-dimensional vector in a finite index set. The negative sign in front of $j(-\bu)$ in the subscript of $S$ is to make the $\sigma$-fields 
\begin{align*}
	\mathscr{F}_j^{(k)}=
	\begin{cases} 
		\sigma\langle N_{\bj} (\theta_{\bj}- \E[\theta_{\bj}|\mathbb{D}_n]) : -(j(\bm{s}_{\bm{t}-\bm{1}}))_{k} < -(j(-\bu))_{k} \leq j \rangle, &\text{ if } k\leq d,\\
		\sigma\langle N_{\bj} (\theta_{\bj}- \E[\theta_{\bj}|\mathbb{D}_n]) : (j(\bm{s}_{\bm{r}-\bm{1}}))_{k-d}<(j(\bv))_{k-d}\leq j\rangle, &\text{ if } k > d.
	\end{cases}
\end{align*}
 increase with respect to each of the first $d$ components in the subscript. 
In the sum above, all $\bj$ are in $ \mathcal{J}(s_{\bm{t}},\bm{s}_{\bm{r}})\setminus\mathcal{J}(\bm{s}_{\bm{t}-\bm{1}},s_{\bm{r}-\bm{1}})$. We note that for every $k\leq 2d$, the random sequence $\{S_{(j_1, \ldots, j_{k-1}, j, j_{k+1},\ldots,j_{2d})}, \mathscr{F}^{(k)}_j\}$ is a martingale. Applying Lemma B.6 of \cite{kang_supp} with $p=4d+2$, we can get an upper bound of the probability of the maximal deviation needed to verify tightness to be a constant multiple of  
\begin{align}
\label{upperbound1}
(\omega_n/\epsilon)^{(4d+2)} \E\big(\big|{\sum_{\mathcal{J}(s_{\bm{t}},s_{\bm{r}})\setminus\mathcal{J}(s_{\bm{t}}-\bm{1},s_{\bm{r}}-\bm{1})}N_{\bj}(\theta_{\bj}-\E[\theta_{\bj}|\mathbb{D}_n])}\big|^{4d+2}\big|\mathbb{D}_n\big).
\end{align}
Observe that $\#\mathcal{J}(\bm{s}_{\bm{t}},\bm{s}_{\bm{r}})\leq \prod_k (r_{n,k} J_k(s_{t_k}+s_{r_k})+2)$, $\#\mathcal{J}(\bm{s}_{\bm{t}-\bm{1}},\bm{s}_{\bm{r}-\bm{1}})\geq \prod_k r_{n,k} J_k(s_{t_k}+s_{r_k}-2\delta)$. As $\delta \leq s_{t_k},s_{r_k}\leq c$ and $J_k\gg r_{n,k}^{-1}$, it follows that the cardinality of the index set $\mathcal{J}(\bm{s}_{\bm{t}},\bm{s}_{\bm{r}})\setminus \mathcal{J}(\bm{s}_{\bm{t}-\bm{1}},\bm{s}_{\bm{r}-\bm{1}})$ is bounded by a multiple of 
\begin{align*}
   \prod_{k=1}^d r_{n,k} J_k \big(\prod_{k=1}^d(s_{t_k}+s_{r_k})- \prod_{k=1}^d(s_{t_k}+s_{r_k}-2\delta)\big)
    \leq (2d\delta)(2c)^{d-1}\prod_{k=1}^d r_{n,k} J_k, 
\end{align*}
where the last inequality follows from Lemma B.7 of \cite{kang_supp}.

The variance $\sigma^2_n N_{\bj}^2/(N_{\bj}+\lambda_{\bj}^{-2})\lesssim n(\prod_{k=1}^d J_k)^{-1}$ with $\P_0$-probability tending to $1$ by Lemma B.4 of \cite{kang_supp}. Hence \eqref{upperbound1} is bounded by a constant multiple of 
\begin{align*}
    &\epsilon^{-(4d+2)}\omega_n^{4d+2}\bigg(  \frac{n\# (\mathcal{J}(s_{\bm{t}},s_{\bm{r}})\setminus\mathcal{J}(s_{\bm{t}-\bm{1}},s_{\bm{r}-\bm{1}}))}{\prod_{k=1}^d J_k}\bigg)^{2d+1}\\
    & \qquad \lesssim \epsilon^{-(4d+2)}\omega_n^{4d+2} \big(\prod_{k=1}^d r_{n,k}\big)^{2d+1} n^{2d+1} \delta^{2d+1}, 
\end{align*}
which simplifies to $\epsilon^{-(4d+2)}\delta^{2d+1}$.
With $\delta$ chosen a sufficiently small constant multiple of $\eta\epsilon^{4d+2}$, the tightness condition is verified.
\end{proof}

\begin{lemma}\label{lemma:Anprimeconv}
    Under the conditions of Theorem \ref{thm: posterior_distribution}, $A'_n(\bm{u},\bm{v})$ converges to $0$ in $\P_0$-probability uniformly in $(\bu,\bv)$. 
\end{lemma}

\begin{proof}
Let $E_n=\{a_1 n/(2\prod_{k=1}^d J_k) \le  N_{\bj} \leq 2a_2 n/(\prod_{k=1}^d J_k)\}$ for some $a_1, a_2 >0$ and $\bar{\varepsilon}|_{I_{\bj}}=\sum_{i\in I_{\bj}} \varepsilon_i/N_{\bj}$. By Lemma B.4 of \cite{kang_supp}, we have, for every $T>0$,
\begin{align}
    \P_0(\max_{\bj} |\bar{\varepsilon}|_{I_{\bj}}|>T) \leq \sum_{\bj} \P_0(|\bar{\varepsilon}|_{I_{\bj}}|>T|E_n)+\P_0(E_n^c).\label{eps_bound}
\end{align}  
By Assumption \ref{assumption_data} and Marcinkiewicz–Zygmund inequality,
\begin{align*}
    \E(|\bar{\varepsilon}|_{I_{\bj}}|^{2(\sum_{k=1}^s \beta_k^{-1} + 1)}|E_n) \lesssim  (a_1 n/(2\prod_{k=1}^d J_k)) ^{-(\sum_{k=1}^s \beta_k^{-1} + 1)}.
\end{align*}
Then \eqref{eps_bound} is bounded by a constant multiple of
$
(\prod_{k=1}^d J_k)^{\sum_{k=1}^s \beta_k^{-1} + 2} n^{-(\sum_{k=1}^s \beta_k^{-1} + 1)} +o(1),
$
which tends to zero because $\prod_{k=1}^d J_k \ll n\omega_n = n^{(\sum_{k=1}^s \beta_k^{-1} + 1)/(\sum_{k=1}^s \beta_k^{-1} + 2)}$.

On the other hand, $\max_{\bj}|\overline{f_0(\bX_i)}|_{I_{\bj}}|\leq f_0(\bm{1})$. Thus $\max_{\bj}|\bar{Y}|_{I_{\bj}}|=O_{\P_0}(1)$.
Because $\E[\theta_{\bj}|\mathbb{D}_n]=( N_{\bj}\bar{Y}|_{I_{\bj}} + \zeta_{\bj}\lambda_{\bj}^{-2})/(N_{\bj}+\lambda_{\bj}^{-2})$, on the event $E_n$, 
\begin{align}
|A'_n(\bm{u},\bm{v})|
&=\omega_n^{-1}\abs{\frac{\sum_{\bj\in [j(-\bm{u}):j(\bm{v})]} \lambda^{-2}_{\bj} N_{\bj} (N_{\bj}+\lambda_{\bj}^{-2} )^{-1}(\zeta_{\bj}-\bar{Y}|_{I_{\bj}}) }{\sum_{\bj\in [j(-\bm{u}):j(\bm{v})]} N_{\bj}} }\nonumber \\ 
&\lesssim \omega_n^{-1} (\max_{\bj}\abs{\bar{Y}|_{I_{\bj}}}+\zeta_{\bj}) (\min_{\bj}N_{\bj})^{-1},\label{formula:boundofAprime}
\end{align}
which is of the order of $ (n\omega_n)^{-1}\prod_{k=1}^d J_k$ in $\P_0$-probability. 
As $\prod_{k=1}^d J_k \ll n\omega_n$ and $\P_0(E_n)\to 1$, we can conclude $A'(\bu,\bv)\to_{\P_0} 0$ uniformly for any $\bu\succeq \bm{0}$ and $\bv\succeq \bm{0}$ provided that $\bx_0 - \bu\circ\bm{r}_n$ and $\bx_0 + \bv\circ\bm{r}_n$ in $[0,1]^d$.
  
\end{proof}

To establish the weak convergence of $B_n$ in $\LL_{\infty}([\bm{0},c\bm{1}]\times [\bm{0},c\bm{1}])$, write 
\begin{equation}\label{equation:bndecomposition}
    B_n(\bm{u},\bm{v})=\omega_n^{-1}\big(\bar{\varepsilon}|_{I_{[j(-\bm{u}):j(\bm{v})]}} + \overline{f_0(\bX)}|_{I_{[j(-\bm{u}):j(\bm{v})]}}-f_0(\bx_0)\big).
\end{equation}

\begin{lemma}\label{lemma:h1nconv}
Let 
$Z_{ni}(\bm{u},\bm{v})= \omega_n \varepsilon_i\Ind_{ \{ \bX_i\in I_{[j(-\bm{u}):j(\bm{v})]} \} }$ and $\HH_{1,n}(\bm{u},\bm{v})=\sum_{i=1}^n Z_{ni}(\bm{u},\bm{v})$. 
Under the conditions of Theorem \ref{thm: posterior_distribution}, 
$
    \HH_{1,n}(\bm{u},\bm{v}) \rightsquigarrow \sigma_0 H_1(\bm{u},\bm{v})\text{ in } \LL_{\infty}([\bm{0},c\bm{1}]\times [\bm{0},c\bm{1}]).
$
\end{lemma}

\begin{lemma}\label{lemma:b2nconv}
Under the conditions of Theorem \ref{thm: posterior_distribution}, for any $c>0$, uniformly in  $(\bm{u},\bm{v})\in [\bm{0},c\bm{1}]^2$,  we have
    \begin{align*} 
        \omega_n^{-1}\big( \overline{f_0(\bX)}|_{I_{[j(-\bm{u}):j(\bm{v})]}}-f_0(\bx_0)\big) 
    \to_{\P_0}\sum_{\bm{l}\in L^{\ast}} \frac{\partial^{\bm{l}}f_0(\bx_0)}{\bm{l}+\bm{1}!} 
    \prod_{k=1}^s \frac{v_k^{l_k+1}-(-u_k)^{l_k+1}}{u_k+v_k}.
    \end{align*}
\end{lemma}

\begin{lemma}\label{lemma:contraction}
  Under the conditions of Theorem \ref{thm: posterior_distribution},
  for any $M_n \uparrow \infty$, $\Pi(\abs{f_*(\bx_0)-f_0(\bx_0)}> M_n\omega_n|\mathbb{D}_n)\to 0$ in $\P_0$-probability.
\end{lemma}

With the aid of Lemma \ref{lemma:contraction}, the second condition of Proposition B.1 of \cite{kang_supp} is verified by Lemma \ref{lemma:large} in the following.
\begin{lemma}
	\label{lemma:large}
	Let $\bm{u}^*$ and $\bm{v}^*$ be any pair indexes such that  
\begin{align}
   f_*(\bx_0)=\max_{\bm{u}\succeq \bm{0}}\min_{\bm{v}\succeq \bm{0}}    \frac{\sum_{[j(-\bm{u}):j(\bm{v})]}N_{\bj}\theta_{\bj}}{\sum_{[j(-\bm{u}):j(\bm{v})]}N_{\bj}}=\frac{\sum_{[j(-\bm{u}^*):j(\bm{v}^*)]}N_{\bj}\theta_{\bj}}{\sum_{[j(-\bm{u}^*):j(\bm{v}^*)]}N_{\bj}}.
\end{align}
  Let $\omega_n=n^{-{1}/{(2+\sum_{k=1}^{s}\beta_k^{-1})}}$ and let  $\bm{r}_n=(\omega_n^{1/\beta_1},\ldots,\omega_s^{1/\beta_s},1,\ldots,1)\trans$. Suppose that $\bJ$ satisfies $J_k\gg r_{n,k}^{-1}$, for each $k=1,\ldots,d$, and  $\prod_{k=1}^d J_k \ll n\omega_n$. 
  Under Assumptions  \ref{assumption_approximation} and \ref{assumption_data}, there exists  $\gamma>0$ such that
    $$\lim_{c\to \infty} \limsup_{n\to\infty}\Pi(c^{-\gamma}\leq \min_{1\leq k \leq d}\{v_k^{\ast}\} \leq \max_{1\leq k \leq d}\{v_k^{\ast}\}\leq c |\mathbb{D}_n)=1,$$ in $\P_0$-probability.
\end{lemma}
The proofs of Lemmas~\ref{lemma:h1nconv}--\ref{lemma:large} are provided in \cite{kang_supp}.


The proof of Theorem~\ref{thm: posterior_distribution} can now be completed. 
Using arguments similar to Proposition~7 of \cite{han2020}, it can be verified that $\P(W_c \neq W) \to 0$ as $c\to\infty$. Hence the proof follows by an application of Lemma B.1 of \cite{kang_supp}. 

\subsection{Proof of Proposition~\ref{lemma:separation}}

This can be shown by the self-similarity property of Gaussian processes $H_1$ and $H_2$: for $\bm{t}\in\RR_{>0}^d$ such that $t_{s+1}=\cdots=t_d=1$, we have that $H_i(\bm{t}\circ \bm{u}, \bm{t}\circ \bm{v})=_d (\prod_{j=1}^s t_j )^{1/2}H_i(\bm{u},\bm{v})$, $i=1,2$. By the choice of $\bm{t}$, multiplying a vector coordinatewise by $\bm{t}$ does not change the last $d-s$ coordinates and thus $D_s(\bm{t}\circ \bm{u},\bm{t}\circ \bm{v})=D_s(\bm{u}, \bm{v})$. Then, since a scaling of the domain does not alter suprema and infima, the expression in the limiting distribution is equal to 
\begin{eqnarray*}
    \lefteqn{ \sup_{\bm{u}\succeq \bm{0}}\inf_{\bm{v}\succeq\bm{0}}\Big\{ 
    \frac{\sigma_0 H_1(\bm{t}\circ \bm{u},\bm{t}\circ \bm{v}) + \sigma_0 H_2(\bm{t}\circ \bm{u},\bm{t}\circ \bm{v})}{\prod_{k=1}^s(t_k u_k+t_k v_k)D_s(\bm{u},\bm{v})} }  \\
    && \qquad +\sum_{k=1}^s \big[\frac{\partial^{\beta_k}_k f_0(\bx_0)}{(\beta_k+1)!} 
     \cdot\frac{(t_k v_k)^{\beta_k+1}-(-t_k u_k)^{\beta_k+1}}{t_k u_k+t_k v_k}\big] \Big\}\\
     && =_d \sup_{\bm{u}\succeq \bm{0}}\inf_{\bm{v}\succeq\bm{0}}\Big\{ \big(\sigma_0^{-2} \prod_{j=1}^s t_j \big)^{-1/2}
    \frac{H_1(\bm{u},\bm{v}) + H_2(\bm{u},\bm{v})}{\prod_{k=1}^s(u_k+v_k)D_s(\bm{u},\bm{v})}\\
    && \qquad  +\sum_{k=1}^s \big[\frac{t_k^{\beta_k}\partial^{\beta_k}_k f_0(\bx_0)}{(\beta_k+1)!} 
     \cdot\frac{v_k^{\beta_k+1}-(-u_k)^{\beta_k+1}}{u_k+v_k}\big] \Big\}.
\end{eqnarray*}
By equating $(\sigma_0^{-2}\prod_{j=1}^s t_j )^{-1/2}$ to $t_k^{\beta_k}\partial^{\beta_k}_k f_0(\bx_0)/(\beta_k + 1 )!$ for each $k=1,\dots,s$, we can find the solution $t_k$ to the system of equations, and also the common factor $A_{\bm{\beta}}$ as stated in the proposition.

If $s=d$, then $D_d(\bm{u}, \bm{v})=g(\bx_0)$ and $H_i(\bm{u},\bm{v})=_d \sqrt{g(\bx_0)}\tilde{H}_i(\bm{u},\bm{v})$. For $\bm{t}\in\RR_{>0}^d$, $\tilde{H}_i(\bm{t}\circ \bm{u}, \bm{t}\circ \bm{v})=_d (\prod_{j=1}^d t_j )^{1/2}\tilde{H}_i(\bm{u},\bm{v})$, $i=1,2$. Hence by self-similarity, the last expression reduces to 
\begin{eqnarray*}
    \lefteqn{ \sup_{\bm{u}\succeq \bm{0}}\inf_{\bm{v}\succeq\bm{0}}\Big\{ 
    \frac{\sigma_0}{\sqrt{g(\bx_0)}}\big(\frac{\tilde{H}_1(\bm{t}\circ\bm{u},\bm{t}\circ\bm{v})}{\prod_{k=1}^d(t_k u_k+t_k v_k)}+\frac{ \tilde{H}_2(\bm{t}\circ\bm{u},\bm{t}\circ\bm{v})}{\prod_{k=1}^d(t_k u_k+t_k v_k)}\big) } \\ 
    && \qquad +\sum_{k=1}^d\left[ \frac{\partial^{\beta_k}_k f_0(\bx_0)}{(\beta_k+1)!} 
     \cdot\frac{(t_k v_k)^{\beta_k+1}-(-t_k u_k)^{\beta_k+1}}{t_k u_k+t_k v_k} \right]\Big\}\\
     &&=_d \sup_{\bm{u}\succeq \bm{0}}\inf_{\bm{v}\succeq\bm{0}}\Big\{ 
    \sqrt{\frac{\sigma^2_0}{g(\bx_0)\prod_{j=1}^d t_j}}\big(\frac{\tilde{H}_1(\bm{u},\bm{v})}{\prod_{k=1}^d(u_k+v_k)}+\frac{ \tilde{H}_2(\bm{u}, \bm{v})}{\prod_{k=1}^d(u_k+v_k)}\big)\\
    && \qquad  +\sum_{k=1}^d\big[ \frac{t_k^{\beta_k}\partial^{\beta_k}_k f_0(\bx_0)}{(\beta_k+1)!} 
     \cdot\frac{v_k^{\beta_k+1}-(-u_k)^{\beta_k+1}}{ u_k+v_k} \big]\Big\}.
\end{eqnarray*}
By exploring the equation system for $t_k$ as follows,
\begin{align*}
    \sqrt{\frac{\sigma^2_0}{g(\bx_0)\prod_{j=1}^d t_j}}=\frac{t_k^{\beta_k}\partial^{\beta_k}_k f_0(\bx_0)}{(\beta_k+1)!}, \text{ for }k=1,\dots,d,
\end{align*}
we can find the common factor $\tilde{A}_{\bm{\beta}}$ in a similar way of solving a set of equations.

\begin{proof}[Proof of Proposition~\ref{dual tail}]
	For $0\le z\le 1$, 
	\begin{align*}
	\P(Z_B^{(1)}\leq z)&=\P(1-Z_B^{(1)}\geq 1-z)\\
	&=\P( \P(-\sup_{\bm{u}\succeq \bm{0}}\inf_{\bm{v}\succeq\bm{0}}\tilde{U}(\bm{u},\bm{v})  \leq 0 | \tilde{H}_1 )\geq 1-z  )\\
	&=\P( \P(\inf_{\bm{u}\succeq \bm{0}}\sup_{\bm{v}\succeq\bm{0}} [-\tilde{U}(\bm{u},\bm{v})]  \leq 0 | \tilde{H}_1)\geq 1-z).
	\end{align*}
	Note that $\tilde{H}_i(\bm{u},\bm{v})=_d \tilde{H}_i(\bm{v},\bm{u})$ and $\tilde{H}_i =_d-\tilde{H}_i$ for $i=1,2$. Denote $\tilde{H}_1^* = - \tilde{H}_1$. Then we have
	\begin{align*}
	&\P\big(\inf_{\bm{u}\succeq \bm{0}}\sup_{\bm{v}\succeq\bm{0}}[-\tilde{U}(\bm{u},\bm{v})]  \leq 0 \big| \tilde{H}_1 \big)\\
	&\quad =_d\P\big(\inf_{\bm{u}\succeq \bm{0}}\sup_{\bm{v}\succeq\bm{0}}
	\big\{\frac{\tilde{H}_1^*(\bm{u},\bm{v})}{\prod_{k=1}^d(u_k+v_k)} +\frac{ -\tilde{H}_2(\bm{u},\bm{v})}{\prod_{k=1}^d(u_k+v_k)} \\
	&\qquad\qquad 
	+\sum_{k=1}^d 
	\frac{-v_k^{\beta_k+1}+(-u_k)^{\beta_k+1}}{u_k+v_k} \big\}
	\leq 0 \big| -\tilde{H}_1^*\big)\\
	&\quad =_d \P\big(\inf_{\bm{u}\succeq \bm{0}}\sup_{\bm{v}\succeq\bm{0}} \big\{
	\frac{\tilde{H}_1^*(\bm{u},\bm{v})}{\prod_{k=1}^d(u_k+v_k)} +\frac{ \tilde{H}_2(\bm{u},\bm{v})}{\prod_{k=1}^d(u_k+v_k)} \\
	&\qquad 
	+\sum_{k=1}^d 
	\frac{u_k^{\beta_k+1}-v_k^{\beta_k+1}}{u_k+v_k} \big\}
	\leq 0 \big| \tilde{H}_1^* \big)\\
	&\quad =_d \P\big(\inf_{\bm{u}\succeq \bm{0}}\sup_{\bm{v}\succeq\bm{0}}
	\big\{
	\frac{\tilde{H}_1^*(\bm{v},\bm{u})}{\prod_{k=1}^d(u_k+v_k)} +\frac{ \tilde{H}_2(\bm{v},\bm{u})}{\prod_{k=1}^d(u_k+v_k)} \\
	&\qquad\qquad 
	+\sum_{k=1}^d 
	\frac{u_k^{\beta_k+1}-v_k^{\beta_k+1}}{u_k+v_k} \big\}
	\leq 0 \big| \tilde{H}_1^* \big)\\
	&\quad = \P\big(\inf_{\bm{u}\succeq \bm{0}}\sup_{\bm{v}\succeq\bm{0}}
	\tilde{U}(\bm{v},\bm{u})\leq 0\Big| \tilde{H}_1^* \big).
	\end{align*}
	Hence 
	$\P(Z_B^{(1)}\leq z)=\P(Z_B^{(2)}\geq 1-z)$. 
	The symmetry of the distribution of $Z_B^{(3)}$ holds by similar arguments.
\end{proof}

\bibliographystyle{plain}
\bibliography{mybib}


\newpage
\setcounter{page}{1}
\begin{center}{\bf \Large Supplement to ``Coverage of Credible Intervals in Bayesian Multivariate Isotonic Regression''}

\bigskip
	
Kang Wang and Subhashis Ghosal
\end{center}

\noindent
In this supplement, we provide the remaining proofs of some lemmas in the main body of the paper in Appendix \ref{sec:appendixA} and all the supporting lemmas and their proofs in Appendix \ref{sec:appendixB}. In what follows we use the notations 
and numbered elements (like equations, sections) from the main paper. 

\def\thesection{\Alph{section}}
\setcounter{section}{0} 
\section{Appendix}
\label{sec:appendixA}
\begin{proof}[Proof of Lemma 6.4]
We first verify the finite-dimensional convergence.  For any pair $(\bm{u},\bm{v})$ and $(\bm{u}',\bm{v}')$, by the independence of $\bX$ and $\varepsilon$, the covariance of 
$\HH_{1,n}(\bm{u},\bm{v})$ and $\HH_{1,n}(\bm{u}',\bm{v}')$ is 
$$n \omega_n^2  \E \big(\varepsilon^2 \Ind_{\{\bX\in I_{[j(-\bm{u}\wedge \bm{u}'):j(\bm{v}\wedge \bm{v}')]} \}}\big)= \sigma_0^2 n\omega_n^2 \int_{I_{[j(-\bm{u}\wedge \bm{u}'):j(\bm{v}\wedge \bm{v}')]}} g(\bx) \d\bx.$$
Write $g(\bx)=g(\bx_s) +  (g(\bx)  - g(\bx_s))$ where  $\bx_{s}=(x_{0,1},\ldots,x_{0,s},x_{s+1},\ldots,x_d)$ and use the continuity of $g$ around $\bx_0$ to reduce the expression to 
\begin{equation*}
    \sigma_0^2 n\omega_n^2  \prod_{k=1}^s \left( (u_k\wedge u_k' + v_k \wedge v_k') r_{n,k} + O(j_k^{-1})\right) D_s^J(\bm{u}\wedge \bm{u}',\bm{v}\wedge \bm{v}')(1+o(1)).
\end{equation*}
From the proof of Lemma \ref{lemma:njconv},
   $ D_s^J(\bu, \bv) \to D_s(\bu,\bv)$.
Since $\bu,\bv,\bu',\bv'$ are bounded and $J_k\gg r_{n,k}^{-1}$, it follows that the limit of the expression in the last display converges to $ \sigma^2_0 \prod_{k=1}^s \left( u_k\wedge u_k'+ v_k\wedge v'_k \right) D_s(\bm{u}\wedge \bm{u}',\bm{v}\wedge \bm{v}')$.

To establish the asymptotic tightness of $\HH_1$ in $\LL_{\infty}([\bm{0},c\bm{1}]\times [\bm{0},c\bm{1}])$, we apply Lemma~\ref{lemma:clt_partial_sum_process} with $\mathcal{F}=[\bm{0},c\bm{1}]^2$ and $\rho((\bm{u},\bm{v}),(\bm{u}',\bm{v}'))=\|\bm{u}-\bm{u}'\|+\|\bm{v}-\bm{v}'\|$, where $\|\cdot\|$ is the Euclidean norm. To verify the first conditions in Lemma \ref{lemma:clt_partial_sum_process}, note that 
$$
\|Z_{ni}\|_{\mathcal{F}}=\omega_n|\varepsilon_i|\Ind_{\{\bX_i\in I_{[j(-c\bm{1}):j(c\bm{1})]}\}}.
$$
For any $\eta >0$,
\begin{align*}
    \sum_{i=1}^{n}\E \|Z_{ni}\|^2_{\mathcal{F}}\Ind_{\{\|Z_{ni}\|_{\mathcal{F}}>\eta\}}
    &\leq  \omega_n^2\sum_{i=1}^n \E[\varepsilon_i^2\Ind_{\{X_i\in I_{[j(-c\bm{1}):j(c\bm{1})]}\}} \Ind_{\{\omega_n|\varepsilon_i|>\eta\}}]\\
     & = n\omega^2 \int_{I_{[j(-c\bm{1}):j(c\bm{1})]}} g(\bx)\mathrm{d}\bx \times \E [ \varepsilon^2 \Ind_{\{|\varepsilon|>\eta\omega_n^{-1}\}}],
\end{align*}
which is bounded by a constant multiple of $ (2c)^s D_s^J(c\bm{1},c\bm{1})\E [ \varepsilon^2 \Ind_{\{|\varepsilon|>\eta\omega_n^{-1}\}}]$, and hence goes to zero.

To check the second condition, note that 
\begin{eqnarray}
\lefteqn{\sum_{i=1}^n \E | Z_{ni}(\bm{u},\bm{v})-Z_{ni}(\bm{u}',\bm{v}')|^2 }\nonumber\\
&&\leq n \omega_n^2 \E\varepsilon^2
\E |\Ind_{\{\bX\in I_{[j(-\bm{u}):j(\bm{v})]}\}} - \Ind_{\{\bX\in I_{[j(-\bm{u}'):j(\bm{v}')]}\}}|^2. \label{quadratic variation}
\end{eqnarray}
The last factor can be bounded by a constant multiple of 
\begin{align*}
& \big(\prod_{k=1}^s r_{n,k}\big) \big[\prod_{k=1}^s (u_k+v_k+O(J_k^{-1}r_{n,k}^{-1})) D_s^J(\bm{u},\bm{v})\\
&\qquad\qquad\quad + \prod_{k=1}^s (u'_k+v'_k+O(J_k^{-1}r_{n,k}^{-1})) D_s^J(\bm{u}',\bm{v}')\\
&\qquad\qquad\quad -2 \prod_{k=1}^s (u_k\wedge u'_k+v_k\wedge v'_k+O(J_k^{-1}r_{n,k}^{-1})) D_s^J(\bm{u}\wedge \bm{u}',\bm{v}\wedge \bm{v}')\big]. 
\end{align*}
Note that $\prod_{k=1}^s r_{n,k} =(n\omega_n^2)^{-1}$. This gives a bound for \eqref{quadratic variation} a constant multiple of 
\begin{align*}
   &  \prod_{k=1}^s (u_k+v_k) D_s(\bm{u},\bm{v}) + \prod_{k=1}^s (u'_k+v'_k) D_s(\bm{u}',\bm{v}')\\
    &\qquad\qquad -2 \prod_{k=1}^s (u_k\wedge u'_k+v_k\wedge v'_k) D_s(\bm{u}\wedge \bm{u}',\bm{v}\wedge \bm{v}'). 
\end{align*}
If $\rho((\bm{u},\bm{v}),(\bm{u}',\bm{v}'))\leq \delta_n$, using Lemma~\ref{lemma:delta}, this expression is bounded by a constant multiple of 
$\delta_n$. Hence the assertion is verified for every $\delta_n \to 0$. 

It remains to verify the third condition of Lemma~\ref{lemma:clt_partial_sum_process}. For any $\epsilon>0$, we consider the partition $\mathcal{P}(\delta')$ given by (6.6) with some $\delta'>0$ depending on $\epsilon$, to be determined later. Let $0=s_0 < s_1 <\ldots < s_l = c$ with $(s_{t-1},s_t]$ of equal length at most $\delta'$ and $l\leq 2c/\delta'$. 
Then $\mathcal{F}$ is covered by $\{(\bu,\bv) \in [\bm{0},c\bm{1}]^2: s_{t_k-1}< u_k \leq s_{t_k}, s_{r_k-1}< v_k \leq s_{r_k}, 1\leq k\leq d\}$,  $\bm{t},\bm{r}\in\{1,\ldots,l\}^d$. Let $\cF_{\epsilon j}^n$ stand for the elements of the partition indexed by $j$ arranged with a certain ordering.
Then for every $j$, 
\begin{eqnarray*}
    \lefteqn{\sum_{i=1}^{n}\E \big[\sup_{f,g\in \mathcal{F}_{\epsilon j}^n} |Z_{ni}(f)-Z_{ni}(g)|^2\big]}\\
     && \le n\omega^2_n \E\varepsilon^2 \E\big|\Ind_{\{ \bX\in I_{[j(-\bm{s}_{\bm{t}}):j(\bm{s}_{\bm{r}})]}\}}-\Ind_{\{\bX\in I_{ [j(-\bm{s}_{\bm{t}-\bm{1}}):j(\bm{s}_{\bm{r}-\bm{1}})] } \}}\big|^2\\
   &&\lesssim \E\varepsilon^2 \big(\prod_{k=1}^d (s_{t_k}+s_{r_k}+O(J_k^{-1}r_{n,k}^{-1}))- \prod_{k=1}^d(s_{t_k-1}+s_{r_k-1}+O(J_k^{-1}r_{n,k}^{-1})) \big)\\
    && \lesssim \E\varepsilon^2 \big(\prod_{k=1}^d (s_{t_k}+s_{r_k})- \prod_{k=1}^d(s_{t_k-1}+s_{r_k-1}) \big),
\end{eqnarray*}
as $s_{t_k},s_{r_k}$ are bounded by $c$ and $J_k\gg r_{n,k}^{-1}$ for $k=1,\ldots, d$. 
By Lemma \ref{lemma:delta}, the above expression is bounded by a constant multiple of $\delta'$. Thus $\delta'$ can be set to a suitable multiple of $\epsilon^2$ to meet the partitioning condition, while the bracketing number $\mathcal{N}_{[\,]}(\epsilon,\mathcal{F}, \|\cdot\|_n)$  with respect to the empirical $\LL_2$-metric is bounded by $N_{\epsilon}=l^{2d}\leq (2c/\delta')^{2d}=C\epsilon^{-4d}$ for some constant $C>0$. 
Then $\int_0^{\delta_n}\sqrt{\log \mathcal{N}_{[\,]}(\epsilon, \mathcal{F},\|\cdot\|_n)} \d\epsilon \leq \int_0^{\delta_n} \sqrt{\log (C\epsilon^{-4d})}\d\epsilon \to 0$, for any $\delta_n\to 0$.
\end{proof}

\begin{proof}[Proof of Lemma 6.5]
By Assumption 1, for $(\bm{u},\bm{v})\in [\bm{0},c\bm{1}]^2$, 
\begin{equation*}
     \overline{f_0(\bX)}|_{I_{[j(-\bm{u}):j(\bm{v})]}}-f_0(\bx_0)
     =\frac{\displaystyle\sum_{i:\bX_i\in I_{[j(-\bm{u}):j(\bm{v})]}} \sum_{\bm{l}\in L^*} {\partial^{\bm{l}}f_0(\bx_0)}(\bX_i-\bx_0)^{\bm{l}}/{\bm{l}!} }
    {N_{[j(-\bm{u}),j(\bm{v})]}} + o(\omega_n),
\end{equation*}
as $u_k,v_k$ are bounded by $c$. We observe that 
\begin{align*}
    {\E\big[   \omega_n \sum_{i:\bX_i\in I_{[j(-\bm{u}):j(\bm{v})]}} (\bX_i-\bx_0)^{\bm{l}}  \big]}= n\omega_n \int_{I_{[j(-\bm{u}):j(\bm{v})]}} \prod_{k=1}^s (\bx-\bx_0)_k^{l_k} g(\bx) \mathrm{d}\bx.
\end{align*}
Again, by writing $g(\bx)=g(\bx_s) + (g(\bx)  - g(\bx_s))$ and using the continuity of $g(\bx)$ at $\bx_0$, the right-hand side of the last display reduces to
\begin{align*}
 n\omega_n  \prod_{k=1}^s \frac{r_{n,k}^{l_k+1}}{l_k+1}\big[v_k ^{l_k+1} - (-u_k)^{l_k+1} + O((J_kr_{n,k})^{-1})\big]( D^J_s(\bm{u},\bm{v}) + o(1)).
\end{align*}
As $\bm{l}\in L^{\ast}$ with $\sum_{k=1}^s l_k/\beta_k=1$, $n\omega_n  \prod_{k=1}^s {r_{n,k}^{l_k+1}} = 1$. Since $J_k\gg r_{n,k}^{-1}$ and that $u_k$ and $v_k$ are all bounded for every $k$, the expression converges to
$
    \prod_{k=1}^s (l_k+1)^{-1}[v_k ^{l_k+1} - (-u_k)^{l_k+1} ] D_s(\bm{u},\bm{v}).
$
 Further, 
\begin{align*}
	\text{Var}\big( \omega_n\sum_{i:\bX_i\in I_{[j(-\bm{u}):j(\bm{v})]}} (\bX_i-\bx_0)^{\bm{l}} \big)
	&= n\omega^2_n \text{Var}\left(  (\bX-\bx_0)^{\bm{l}} \Ind_{\{\bX\in I_{[j(-\bm{u}):j(\bm{v})]}\}} \right)\\
	& \le  n\omega^2_n\E\big(  (\bX-\bx_0)^{2\bm{l}} \Ind_{\{\bX\in I_{[j(-\bm{u}):j(\bm{v})]}\}} \big)
\end{align*}
is bounded by a constant multiple of  $n\omega_n^{2+\sum_{k=1}^s (2l_k +1)/\beta_k}=n^{-{2}/\{2+\sum_{k=1}^s\beta_k^{-1}\}}$.
Together these two imply the assertion.
\end{proof}

\begin{proof}[Proof of Lemma 6.6]

We first prove $\Pi(f_*(\bx_0)-f_0(\bx_0)\leq M_n\omega_n|\mathbb{D}_n)\to 1$ in $\P_0$-probability. For the ease of notation, we show this for the case $s=d$ only. By the max-min formula, we have
\begin{align*}
  f_*(\bx_0)-f_0(\bx_0) &\leq  \max_{\bm{u}\succeq \bm{0}}
    \frac{\sum_{[j(-\bm{u}):j(\bm{1})]}N_{\bj}\theta_{\bj}}{\sum_{[j(-\bm{u}):j(\bm{1})]}N_{\bj}} - f_0(\bx_0)\\
    & = \omega_n \max_{\bu\succeq \bm{0}}\{A_n(\bu,\bm{1};\bm{\theta}) + A'_n(\bu,\bm{1}) + B_n(\bu,\bm{1})\},
\end{align*}
where $A_n$, $A_n'$ and $B_n$ are defined in (6.3)--(6.5).

Let 
$$
E_n=\{a_1 {n}/{(2\prod_{k=1}^d J_k)}\leq \min_{\bj} N_{\bj} \leq \max_{\bj} N_{\bj}\leq 2a_2{n}/({\prod_{k=1}^d J_k})\}.
$$
Then by Lemma \ref{lemma:submartingale}, writing $\psi_{\bj} = \theta_{\bj} -\E[\theta_{\bj}|\mathbb{D}_n]$, we can bound $$\E \big[\omega_n \sup_{\bm{u}\succeq\bm{0}} |{A_n(\bu,\bm{1};\bm{\theta})}|\big| \mathbb{D}_n\big] =\E \big[ \sup_{\bm{u}\succeq \bm{1}} \big|{\frac{\sum_{[j(-\bm{u}+\bm{1}):j(\bm{1})]}N_{\bj}\psi_{\bj}}{\sum_{[j(-\bm{u}+\bm{1}):j(\bm{1})]}N_{\bj}}}\big| \big| \mathbb{D}_n\big]$$ by the sum of the supremums over subregions $\prod_{k=1}^d [2^{h_k}\leq u_k \leq 2^{h_k+1}]$ as    
\begin{align*}
   & \sum_{\substack{h_k\geq 0\\1\leq k \leq d}}\E \big[ \sup_{\substack{2^{h_k}\leq u_k \leq 2^{h_k+1} \\ 1\leq k \leq d} } \big|{\frac{\sum_{[j(-\bm{u}+\bm{1}):j(\bm{1})]}N_j\psi_{\bj}}{\sum_{[j(-\bm{u}+\bm{1}):j(\bm{1})]}N_{\bj}}}\big| \big| \mathbb{D}_n\big]\\
  & \quad \leq \sum_{\substack{h_k\geq 0\\1\leq k \leq d}} \frac{1}{\sum_{[j(-2^{\bm{h}}+\bm{1}):j(\bm{1})]}N_{\bj}}\E \big[ \sup_{2^{h_k}\leq u_k \leq 2^{h_k+1} } \big|{\sum_{[j(-\bm{u}+\bm{1}):j(\bm{1})]}N_{\bj}\psi_{\bj}}\big| \big| \mathbb{D}_n\big]\\
  & \quad \le \sum_{\substack{h_k\geq 0 \\ 1\leq k \leq d}} \frac{1}{\sum_{[j(-2^{\bm{h}}+\bm{1}):j(\bm{1})]}N_{\bj}}
  \big(\E\big[
  \big|{\sum_{[j(-2^{\bm{h}+\bm{1}}+\bm{1}):j(\bm{1})]}N_{\bj}\psi_{\bj}}\big|^2 \big| \mathbb{D}_n
  \big]\big)^{1/2}.
  \end{align*}
  Using (3.7), on the event $E_n$, this can be bounded by 
  \begin{eqnarray*}
   \lefteqn {\sigma_n\sum_{\substack{h_k\geq 0\\1\leq k \leq d}} \frac{\big(\sum_{[j(-2^{\bm{h}+\bm{1}}+\bm{1}):j(\bm{1})]}N_{\bj}\big)^{1/2}}{\sum_{[j(-2^{\bm{h}}+\bm{1}):j(\bm{1})]}N_{\bj}}}\\
   && \lesssim  \sigma_n\sum_{\substack{h_k\geq 0\\1\leq k \leq d}} \frac{\big(n(\prod_{k=1}^d J_k)^{-1}\cdot (\prod_{k} r_k\cdot 2^{h_k+1}J_k)\big)^{1/2}}{n(\prod_{k=1}^d J_k)^{-1}\cdot (\prod_{k} r_k\cdot 2^{h_k}J_k)}.
   \end{eqnarray*}
   This is clearly bounded by a constant multiple of $ \omega_n \sigma_0 \sum_{\substack{h_k\geq 0\\1\leq k \leq d}} 2^{-\sum_k{h_k}/2}$, and hence  $$
   \Pi(\max\{A_n(\bu,\bm{1};\bm{\theta}) : \bu\succeq \bm{0}\}>M_n\omega_n|\mathbb{D}_n)\to 0$$
   in $\P_0$-probability.

We have shown in (6.9) that
\begin{align*}
    \sup_{\bm{u}\succeq \bm{0}} |A'_n(\bm{u},\bm{1})| = O_{\P_0}\big(\omega_n^{1+\sum_{k=1}^s\beta_k^{-1}}\prod_{k=1}^d J_k\big)\to 0
\end{align*}
by the choice of $J_k$ in $\P_0$-probability.

To bound $B_n$, we also decompose $B_n$ into an approximation part and an error part, and bound these two parts separately. Using the similar calculation for the expectation of $\omega_n\sup_{\bu\succeq \bm{0}}|A_n(\bu,\bm{1};\bm{\theta})|$, restricted on the event $E_n$, we obtain 
$$\E\big[ \sup_{\bm{u}\succeq \bm{0}}|\bar{\varepsilon}|_{I_{[j(-\bm{u}):j(\bm{1})]}}|\Ind_{A_n} \big]=O(\omega_n).$$ 
By the monotonicity of $f_0$ and Assumption 1, 
$\overline{f_0(\bX)}|_{I_{[j(-\bm{u}):j(\bm{1})]}}-f_0(\bx_0) \leq f_0(\bx_0+ \bm{r}_n + \bJ^{-1})-f_0(\bx_0)$, which can be expanded as 
$$ \sum_{\bm{l}\in L^*} {\partial^{\bm{l}}f_0(\bx_0)}  \bm{r}_n^{\bm{l}}/{\bm{l}!} + o(\omega_n)=O(\omega_n).
$$
Combining these bounds, the claim follows. For the other side, we note that
\begin{align*}
    -(f_*(\bx_0)-f_0(\bx_0))  & \leq  - \omega_n \min_{\bv\succeq \bm{0}}\{A_n(\bm{1},\bv;\bm{\theta}) + A'_n(\bm{1},\bv) + B_n(\bm{1},\bv)\}\\
    & =  \omega_n \max_{\bv\succeq \bm{0}}\{ |A_n(\bm{1},\bv;\bm{\theta})| +  |A'_n(\bm{1},\bv)| + |B_n(\bm{1},\bv)|\}
\end{align*}
and apply the same line of arguments.

\end{proof}

\begin{proof}[Proof of Lemma 6.7]

We write
$$
f_{\ast}(\bx_0) - f_0(\bx_0) = \omega_n \max_{\bm{u}\succeq \bm{0}}\min_{\bm{v}\succeq \bm{0}} \{A_n(\bu,\bv) + A'_n(\bu,\bv) + B_n(\bu,\bv)\}.
$$
Furthermore, we write $B_n(\bu,\bv) = B_{1n}(\bu,\bv) + B_{2n}(\bu,\bv)$, where
\begin{align*}
    B_{1n}(\bu,\bv) & = \omega_n^{-1}(\bar{\varepsilon}|_{I_{[j(-\bm{u}):j(\bm{v})]}}),\\
    B_{2n}(\bu,\bv) & = \omega_n^{-1}(\overline{f_0(\bX)}|_{I_{[j(-\bm{u}):j(\bm{v})]}}-f_0(\bx_0)).
\end{align*}
For $c>\max\{(1-x_{0,k}):s+1\leq k\leq d\} $, we only need to consider the event $\{\max\{ v^*_k: 1\leq k\leq s\}>c\}$. 
By the monotonicity of $f$, we have
  $  \overline{f_0(\bX)}|_{I_{[j(-\bm{u}^*):j(\bm{v}^*)]}}-f_0(\bx_0)\geq \overline{f_0(\bX)}|_{I_{[j(-\bm{1}):j(\bm{v}^*)]}}-f_0(\bx_0)$. By Lemma \ref{lemma:c5}, on an event with $\P_0$-probability tending to $1$, up to  $o(\omega_n \max\{v^{\ast\beta_k}_k: 1\leq k \leq s\})$,  this can be expressed as 
\begin{eqnarray*}
    \lefteqn{\sum_{\bm{l}\in L^*}\frac{\partial^{\bm{l}}f_0(\bx_0)}{\bm{l}!}
    \frac{\sum_{i:\bX_i\in I_{[j(-\bm{1}):j(\bm{v}^*)]}} (\bX_i-\bx_0)^{\bm{l}}}{\sum_{[j(-\bm{1}):j(\bm{v}^*)]}N_{\bj}}}\\
     &&= \sum_{\bm{l}\in L^*}\frac{\partial^{\bm{l}}f_0(\bx_0)}{\bm{l}!}
    \frac{\int_{I_{[j(-\bm{1}):j(\bm{v}^*)]}} (\bx-\bx_0)^{\bm{l}}g(\bx)\mathrm{d} \bx}{\int_{I_{[j(-\bm{1}):j(\bm{v}^*)]}} g(\bx)\mathrm{d} \bx}\\
    &&=  \sum_{\bm{l}\in L^*}\frac{\partial^{\bm{l}}f_0(\bx_0)}{\bm{l}!}
    \prod_{k=1}^s
    \frac{\int_{\ceil{(x_{0,k}-r_{n,k}) J_k}/J_k-1/J_k}^{\ceil{(x_{0,k}+r_{n,k} v^*_k)J_k}/J_k}(x-x_0)^{l_k} g(x)\d x} {\int_{\ceil{((x_0)_k-r_{n,k}) J_k}/J_k-1/J_k}^{\ceil{((x_0)_k+r_{n,k} v^*_k)J_k}/J_k}g(x)\d x}
    \\
    && \lesssim \sum_{\bm{l}\in L^*}\frac{\partial^{\bm{l}}f_0(x_0)}{(\bm{l}+\bm{1})!}
    \prod_{k=1}^s
    \frac{(v^*_kr_{n,k})^{l_k+1}-(-r_n)_k^{l_k+1}}{(1+v^*_k)r_{n,k}},
\end{eqnarray*}
which is bounded above by a constant multiple of $\omega_n \max \{ v_k^{\ast\beta_k}: {1 \leq k \leq s}\}$. 
As $\Pi(\abs{f_*(\bx_0)-f_0(\bx_0)}>M_n\omega_n|\mathbb{D}_n)\to_{\P_0} 0$, in view of Lemma 6.6,
this gives that $\Pi(\max\{ v^*_k: 1\leq k\leq s\}\leq c|\mathbb{D}_n)\to_{\P_0} 1$ when $n, c\to\infty$. 

Define $\Lambda^{(0)}=\{v^*_d<c^{-\gamma}\}$ and $\Lambda^{(1)}=\{\max({v_k}^*: 1\leq k \leq d)\leq c\}$. By the previous proof, for every $\eta,\epsilon>0$, we have $\P_0(\Pi(\Lambda^{(1)}|\mathbb{D}_n)\geq 1 - \eta)\geq 1-\epsilon$ when $n$ and $c$ are large enough.
We consider $s<d$ for simplicity. The case $s=d$ follows with a slightly different bound by the same argument.
For some $a,b,\gamma>0$ to be determined later, define a subset $R_{a,b,\gamma}(c)\subset \RR^d_{\geq 0}\times\RR^d_{\geq 0}$ by  $$\Big\{(\bm{u},\bm{v}): 
\begin{array}{ll} 
	0\leq u_k \leq c^a\Ind_{1\leq k\leq s} + x_{0,k}\Ind_{s+1\leq k\leq d}, & 0\leq u_d \leq c^{-b}\\ 0\leq v_k \leq c\Ind_{1\leq k\leq s} + (1-x_{0,k})\Ind_{s+1\leq k\leq d}, & 0\leq v_d \leq c^{-\gamma}
\end{array}	\Big\}. $$ 
Define these two events, for some $C_1>0$ and $C'_1>0$, 
\begin{align*}
	\Lambda^{(2)}=\{\sup_{(\bm{u},\bm{v})\in R_{a,b,\gamma}}\abs{\HH_{2,n}(\bm{u},\bm{v})-\HH_{2,n}(\bu,\bv\Ind_{[s+1:d-1]})}\leq (C_1/\eta)\sqrt{c^{as-\gamma}\log c} \},\\
	\Lambda'^{(2)}=\{\sup_{(\bm{u},\bm{v})\in R_{a,b,\gamma}}\abs{\HH_{1,n}(\bm{u},\bm{v})-\HH_{1,n}(\bu,\bv\Ind_{[s+1:d-1]})}\leq (C'_1/\epsilon)\sqrt{c^{as-\gamma}\log c} \}.
\end{align*}
Since $\HH_{2,n}(\bm{u},\bm{v})\leadsto \sigma_0 H_2(\bm{u},\bm{v})$ 
in $\P_0$-probability in $\LL_{\infty}([\bm{0},c\bm{1}]\times[\bm{0},c\bm{1}])$  and  
$$
\HH_{1,n}(\bm{u},\bm{v})\leadsto \sigma_0 H_1(\bm{u},\bm{v}) \text{ in }\LL_{\infty}([\bm{0},c\bm{1}]\times [\bm{0},c\bm{1}]),
$$
for any $c>0$, it follows from 
Lemma \ref{lemma:c6} that, when $n$ and $c$ are large enough, there exist constants $C_1$ and $C'_1$ depending on $\sigma_0^2,d,a$ only, such that $\P_0(\Pi(\Lambda^{(2)}|\mathbb{D}_n)\geq 1-\eta)\geq 1-\epsilon$ and $\P_0(\Lambda'^{(2)})\geq 1-\epsilon$. 

By Lemma~\ref{lemma:c5}, given any $\eta,\epsilon>0$, there  exists $\rho_{\eta\epsilon}>0$ such that, when $a>1$, $c>1$ and $n$ large enough, we have
\begin{align*}
    &\P_0\times\Pi\big(\min_{\substack{0\leq v_k \leq c\Ind_{1\leq k\leq d}\\ 0\leq v_d \leq c^{-\gamma}}} \max_{\substack{0\leq u_k \leq c^a \Ind_{1\leq k\leq d}\\0\leq u_d\leq c^{-b}}}\\
    &\qquad\qquad\qquad \HH_{2,n}(\bu,\bv\Ind_{[s+1:d-1]})  + \HH_{1,n}(\bu,\bv\Ind_{[s+1:d-1]}) \leq \sqrt{c^{as-b/x_{0,d}}}\rho_{\eta\epsilon} \big)\\
    & \quad \lesssim 
    \P\big(\min_{\substack{0\leq v_k \leq c\Ind_{s+1\leq k\leq d}\\ v_d =0}} \max_{\substack{0\leq u_k \leq c^a \Ind_{1\leq k\leq d}\\0\leq u_d\leq c^{-b}}}   H_{2}(\bm{u},\bm{v})+H_{1}(\bm{u},\bm{v}) \leq \sqrt{c^{as-b/x_{0,d}}}\rho_{\eta\epsilon}\big)\\
    & \quad \leq \P \big(\min_{\substack{0\leq v_k \leq c\Ind_{s+1\leq k\leq d}\\ v_d =0}} \max_{\substack{0\leq u_k \leq  \Ind_{1\leq k\leq d}\\0\leq u_d\leq x_{0,d}}} H_{2}(\bm{u},\bm{v})+H_{1}(\bm{u},\bm{v}) \leq \rho_{\eta\epsilon} \big)\leq \eta\epsilon.
\end{align*}
Hence, $\P_0\times \Pi(\Lambda^{(3)})\geq 1-\eta\epsilon$ for sufficiently large $n$, where $\Lambda^{(3)}$ stands for the event that 
for any  $0\leq v_k \leq c\Ind_{\{1\leq k\leq d\} },  0\leq v_d \leq c^{-\gamma}$, 
there exists  $0\leq u_k \leq c^a \Ind_{\{1\leq k\leq d\} }, 0\leq u_d\leq c^{-b}$  such that $\HH_{2,n}(\bu,\bv\Ind_{[s+1:d-1]}) + \HH_{1,n}(\bu,\bv\Ind_{[s+1:d-1]}) > C_2\sqrt{c^{as-b/x_{0,d}}}$ for  some constant $C_2>0$ Therefore, we have $\P_0(\Pi(\Lambda^{(3)}) \geq 1-\eta)\geq 1-\epsilon $.

Let $u(c)=(c^a\Ind_{\{1\leq k\leq s\}} +x_{0k}\Ind_{\{s+1\leq k\leq d-1\} }+c^{-b}\Ind_{\{k=d\}}: k=1,\ldots,d)$ and $v(c)=(c\Ind_{1\leq k\leq s}+(1-x_{0,k})\Ind_{s+1\leq k\leq d-1}+c^{-\gamma}\Ind_{k=d}: k=1,\ldots,d)$.
By Bernstein's inequality (cf. Lemma 2.2.9 of \cite{van1996weak}), 
\begin{align*}
   \P_0\big(\big|\sum_{[j(-u(c)):j(v(c))]}N_{\bj}-\E[\sum_{[j(-u(c)):j(v(c))]}N_{\bj}]\big|\geq n\sigma^2_c \big)\leq C_3\exp\{-C_3^{-1}n\sigma^2_c\},
\end{align*}
where $\sigma^2_c=\text{Var}(\Ind_{\{\bX\in I_{[j(-u(c)):j(v(c))]\}}})\leq\E(\Ind_{\{\bX\in I_{[j(-u(c)):j(v(c))]\}}})$. 
Then for $$
\Lambda^{(4)}=\big\{ \sum_{[j(-u(c)):j(v(c))]}N_j \leq 2\E[\sum_{[j(-u(c)):j(v(c))]}N_j]\big\},
$$
it follows that $\P_0(\Lambda^{(4)})\to 1$.

On $\Lambda^{(0)}\cap \Lambda^{(1)}\cap\Lambda'^{(1)}\cap\Lambda^{(2)}\cap\Lambda^{(3)}\cap\Lambda^{(4)}$, it
holds that 
\begin{eqnarray*}
    \lefteqn{\max_{\substack{0\leq u_k \leq c^a \Ind_{1\leq k\leq d}\\0\leq u_d\leq c^{-b}}} A_n(\bm{u},\bm{v}^*)+B_{1,n}(\bu,\bv^{\ast})}\\ &&=\max_{\substack{0\leq u_k \leq c^a \Ind_{1\leq k\leq d}\\0\leq u_d\leq c^{-b}}}\frac{\HH_{2,n}(\bm{u},\bm{v}^*)+\HH_{1,n}(\bm{u},\bm{v}^*)}{\omega^2_n \sum_{[j(-\bm{u}):j(\bm{v}^*)]}N_j}\\
    &&\geq \frac{C_2\sqrt{c^{as-b}}-C_1\sqrt{c^{as-\gamma}\log c} /\eta}{(c^a+c)^s(c^{-b}+c^{-\gamma})},
\end{eqnarray*}
which is greater than or equal to $C_3 c^{(b-as)/2}$ for some positive constant $C_3$.

On the other hand, for some positive constant $C_4$, we have
\begin{eqnarray*}
    \lefteqn{\min_{\substack{0\leq u_k \leq c^a \Ind_{1\leq k\leq d}\\0\leq u_d\leq c^{-b}}}
    B_{2n}(\bu,\bv^{\ast})}\\
    && = \omega_n^{-1} \min_{\substack{0\leq u_k \leq c^a \Ind_{1\leq k\leq d}\\0\leq u_d\leq c^{-b}}}
    \overline{f_0(\bX)}|_{I_{[j(-\bm{u}):j(\bm{v}^*)]}}-f_0(\bx_0)\\
    &&\geq\omega_n^{-1}( f_0(\bx_0-(c^a+o(1)) \bm{r}_n)-f_0(\bx_0) ),
\end{eqnarray*}
which is greater than or equal to $ -C_4 c^{a\max_{k\leq s}\beta_k}$. 
In view of (6.9), we conclude that, on $\Lambda^{(0)}\cap \Lambda^{(1)}\cap\Lambda'^{(1)}\cap\Lambda^{(2)}\cap\Lambda^{(3)}\cap\Lambda^{(4)}$,
 $f^*(\bx_0)-f(\bx_0)$ is bounded below by 
\begin{eqnarray*}
     \lefteqn{\omega_n \max_{\substack{0\leq u_k \leq c^a \Ind_{1\leq k\leq d}\\0\leq u_d\leq c^{-b}}} \{A_n(\bm{u},\bm{v}^*) + B_{1n}(\bm{u},\bm{v}^*)\} }\\
     &&  + \min_{\substack{0\leq u_k \leq c^a \Ind_{1\leq k\leq d}\\0\leq u_d\leq c^{-b}}} B_{2n}(\bm{u},\bm{v}^*) - \max_{\substack{0\leq u_k \leq c^a \Ind_{1\leq k\leq d}\\0\leq u_d\leq c^{-b}}} |A'_{n}(\bm{u},\bm{v}^*)| \\
    &&\geq \omega_n c^{a\max\{\beta_k:k\leq s\} }(C_3c^{(b-as)/2-a\max_{k\leq s}\beta_k}-C_4+o(1)).
\end{eqnarray*}
Take $a=3$, $b\geq 2(1+a\max\{\beta_k:k\leq s\} )+as$ and $\gamma=b+1$. Hence the intersection of these events can only occur with arbitrarily small posterior probability in $\P_0$-probability in view of Lemma 6.6, for large enough $n$ and $c$. 
As $\Pi(\Lambda^{(0)}\cap\Lambda^{(1)}\cap\Lambda^{(2)}\cap\Lambda^{(3)}|\mathbb{D}_n)\to_{\P_0} 0$ while $\Pi(\Lambda^{(1)}\cap\Lambda^{(2)}\cap\Lambda^{(3)}|\mathbb{D}_n)\to_{\P_0} 1$ when $n,c\to \infty$,
thus we can conclude $\Pi([\Lambda^{(0)}]^c\cap \Lambda^{(1)}|\mathbb{D}_n)\to 1$ in $\P_0$-probability. 
\end{proof}

\def\thesection{\Alph{section}}
\section{Appendix}
\label{sec:appendixB}

\begin{lemma}
	\label{domain truncation}
	Let $\mathbb{D}_n$, $n\ge 1$, be a set of random observations with distribution $\P_0^n$. Let $W_n^*,W_{n,c}$, $c>0$, $n=1,2,\ldots$, be random variables  and let $\cL^{\ast}_n = \mathcal{L}(W_n^{\ast}|\mathbb{D}_n)$, $\cL^{\ast}_{n,c}= \mathcal{L}(W_n^{\ast}|\mathbb{D}_n)$ stand for their conditional distributions given $\mathbb{D}_n$ respectively, viewed as random measures on $\RR$. 
	Let  $W,W_c$, $c>0$, be random variables and $H$ be a random process, and $\cL_c = \mathcal{L}(W_c|H)$, $\cL = \mathcal{L}(W|H)$. Assume that 
	\begin{enumerate}
		\item [{\rm (i)}] for every $c>0$, $\cL^{\ast}_{n,c} \leadsto \cL_c$;
		\item [{\rm (ii)}]$\lim_{c\to\infty}\limsup_{n\to\infty} \P(W^{\ast}_{n,c}\neq W^{\ast}_n|\mathbb{D}_n)=0$ in $\P_0^n$-probability;
		\item [{\rm (iii)}] $\P(W_c\neq W)\to 0$ as $c\to\infty$.   
	\end{enumerate}
Then $\cL_n^*\leadsto \cL$.
\end{lemma}

\begin{proof}
	Let $\mathfrak{M}$ denote the collection of random probability measure on $(\RR, \mathscr{B})$, where $\mathscr{B}$ is the Borel $\sigma$-algebra. Fix a uniformly continuous function $f:\mathfrak{M} \to [0,1]$. For a chosen $\epsilon>0$, get $0<\eta<\epsilon$, $k$ and $g_1,\ldots, g_k$ uniformly continuous functions from $\RR$ to $[0,1]$ depending on $f$ only, such that $\sum_{j=1}^k |\int g_j \mathrm{d} Q - \int g_j \mathrm{d} Q'| < 2k\eta$ implies $|f(Q) - f(Q')| < \epsilon$.
	
	For any $n$ and $c$, we have 
\begin{equation*}
    |\E f(\cL_n^{\ast}) - \E f(\cL)|	\leq  \E |f(\cL_n^{\ast}) - f(\cL_{n,c}^{\ast})| + |\E f(\cL_{n,c}^{\ast}) - \E f(\cL_{c})| + \E |f(\cL_{c}) - f(\cL)|,
\end{equation*}
		so it suffices to bound each term for all sufficiently large $n$ and a suitable $c>0$. 
		
		Using (iii), get $c'>0$ such that $\P(W_c\neq W|H)<\epsilon$ for all $c>c'$ on a set $E$ with $\P (E^c)<\epsilon$. 
		
		From (ii), get $c^*\ge c'$ and $N^*\ge 1$ such that $\P (W_{n,c^*}\ne W_n|\mathbb{D}_n)<\eta$ on a set $E_n$ with $\P(E_n^c)<\eta$ for all $n\ge N^*$. 
		
		From (i), get $N\ge N^*$ such that $|\E f(\cL_{n,c^*}^{\ast}) - \E f(\cL_{c^*})| <\epsilon$ for all $n\ge N$. 
	
	Since $|f|\le 1$,  
	$$
	\E |f(\cL_n^{\ast}) - f(\cL_{n,c^*}^{\ast})|\le \epsilon+\P (|f(\cL_n^{\ast}) - f(\cL_{n,c}^{\ast})|>\epsilon),
	$$
	so it suffices to control 
	$$
	\P(\sum_{j=1}^k |\int g_j d\cL_{n,c^*}-\int g_j d\cL_{n}|\ge k\eta).
	$$
	The $j$th term 
	$$ \E [|g(W_{n,c^*}^*)-g(W_n^*)||\mathbb{D}_n]\le \eta+\P (|W_{n,c^*}^*-W_n^*|>\eta|\mathbb{D}_n)<2\eta,~ j=1,\ldots,k,
	$$  on the event $E_n$ for all $n\ge N$. Hence $\P (|f(\cL_n^{\ast}) - f(\cL_{n,c}^{\ast})|>\epsilon)<\epsilon$ on $E_n$ for all $n\ge N$. 
	
	By the same argument,  
	$\E |f(\cL_c^*) - f(\cL)|\le \epsilon+\P (|f(\cL_c) - f(\cL)|>\epsilon)$, and 
	$$
	\E [|g(W_{c^*})-g(W)|]\le \eta+\P (|W_{c^*}-W|>\eta)<2\eta,
	$$
	assuring that  $\P (|f(\cL_c) - f(\cL)|>\epsilon)<\epsilon$ on $E$. 
	
	Piecing these together, using the value $c=c^*$, for all $n\ge N$, we obtain that 
	$$
	|\E f(\cL_n^{\ast}) - \E f(\cL)| 	\leq  3 \epsilon+\P (E_n^c)+\P(E^c)<5\epsilon.
	$$
	Since $\epsilon>0$ is arbitrary, this completes the proof.
\end{proof}

\begin{lemma}\label{lemma:njconv}
	For $(\bm{u},\bm{v})\in [\bm{0}, c\bm{1}]\times[\bm{0}, c\bm{1}]$ such that  $u_k\leq x_{0,k}$, $v_k\leq 1-x_{0,k}$ for all $s+1\leq k\leq d$, we have that 
    $\omega_n^2\sum_{\bj\in [j(-\bm{u}):j(\bm{v})]} N_{\bj} \to_{\P_0} \prod_{k=1}^{s}(u_k+v_k)D_s(\bm{u},\bm{v})$.
\end{lemma}

\begin{proof}
Let $\bx_s=(x_{0,1},\ldots,x_{0,s}, x_{s+1}, \ldots x_d)$. For $s<d$, let 
$$   D_s^J(\bm{u}, \bm{v})=\int_{ \substack{(\ceil{(x_{0,k} - u_k)J_k} - 1)/J_k<x_k\leq\ceil{(x_{0,k} + v_k)J_k}/J_k\\s+1\leq k \leq d}} g(\bx_s) \mathrm{d}x_{s+1}\cdots \mathrm{d}x_d,$$ 
 and  $D_d^J(\bm{u}, \bm{v})=D_d(\bm{u}, \bm{v})=g(\bx_0)$.
As $0<a_1 \leq g(\bx)\leq a_2 \leq \infty$ and $J_k\to\infty$ for every $k=s+1,\ldots, d$, we have 
$    \abs{D_s^J(\bm{u}, \bm{v}) - D_s(\bm{u}, \bm{v})} \lesssim \max\{ J_k^{-1}: {s+1\leq k \leq d}\} \to 0$ as $n\to \infty$.

By the continuity of $g(\bx)$ at $\bx_0$, and using the facts that $u_k,v_k\leq c$, $r_{n,k}\to 0$ for $1\leq k\leq s$, and $J_k\gg r_{n,k}^{-1}$, it follows that  
\begin{eqnarray*}
 \lefteqn{\E\big[\omega_n^{2}\sum_{\bj\in [j(-\bm{u}):j(\bm{v})]} N_{\bj}\big] = n \omega_n^2\int_{I_{[j(-\bm{u}):j(\bm{v})]}} g(\bx)\d \bx}\\ 
\\
&&= n \omega_n^2 \prod_{k=1}^{s}(u_k r_{n,k}+v_k r_{n,k} + O(J_k^{-1}))(D_s^J(\bm{u}, \bm{v}) + o(1))\\
&&= n \omega_n^{2+\sum_{k=1}^s \beta_k^{-1}}\prod_{k=1}^{s}(u_k+v_k + O(( r_{n,k}J_k)^{-1}))(D_s^J(\bm{u},\bm{v})+o(1))\\
&&\to \prod_{k=1}^{s}(u_k+v_k)D_s(\bm{u},\bm{v}).
\end{eqnarray*}
Further, as $\sum_{\bj\in [j(-\bm{u}):j(\bm{v})]} N_{\bj}$ is binomially distributed, 
\begin{align*}
    \text{Var}\big(\omega_n^{2}\sum_{\bj\in [j(-\bm{u}):j(\bm{v})]} N_{\bj}\big)
\leq \omega_n^2 \big(\prod_{k=1}^s\left(u_k+v_k\right)D_s(\bm{u},\bm{v}) + o(1)\big)\to 0.
\end{align*}
Thus conclusion now follows from Chebyshev's inequality.
\end{proof}

\begin{lemma}[Theorem 2.11.9 of \cite{van1996weak}]
	\label{lemma:clt_partial_sum_process}
	For each $n$, let $Z_{n1},\ldots,Z_{n m_n}$ be independent stochastic processes indexed by a totally bounded semi-metric space $(\mathcal{F},\rho)$. $\|Z_{ni}\|_{\cF}=\sup\{ |Z_{ni}(f)|: f\in \cF\}$, $i=1,\ldots,m_n$, $n\in \NN$. Suppose that
	\begin{enumerate}
		\item [{\rm (i)}] $\sum_{i=1}^{m_n} \E \pnorm{Z_{ni}}{\mathcal{F}}^2 \Ind_{ \{\pnorm{Z_{ni}}{\mathcal{F}}>\eta\} }\to 0$ for every $\eta>0$.
		\item [{\rm (ii)}] $\sup\{\sum_{i=1}^{m_n} \E\big(Z_{ni}(f)-Z_{ni}(g)\big)^2: \rho(f,g)<\delta_n \} \to 0$ for every $\delta_n\to 0$.
		\item [{\rm (iii)}] $\int_0^{\delta_n} \sqrt{\log \mathcal{N}_{[\,]}(\epsilon, \mathcal{F}, \|\cdot\|_n)}\ \d{\epsilon}\to 0$ for every $\delta_n\to 0$, where 
			\end{enumerate}
		\begin{align*}
\mathcal{N}_{[\,]}(\epsilon, \mathcal{F}, \|\cdot\|_n)= \min \{ N: \cup_{j=1}^N \mathcal{F}_{\epsilon j}^n	\supset \mathcal{F}, \; \E \sup_{f,g \in \mathcal{F}_{\epsilon j}^n} \abs{Z_{ni}(f)-Z_{ni}(g)}^2\leq \epsilon^2\}.
		\end{align*}
	Then the sequence $\sum_{i=1}^{m_n} (Z_{ni}- \E Z_{ni})$ is asymptotically tight in $\LL_{\infty}(\mathcal{F})$.
\end{lemma}

\begin{lemma}
	\label{lemma:Nj}
	Let $E_n=\big\{a_1 {n}/({2\prod_{k=1}^dJ_k})\leq  N_{\bj} \leq 2a_2{n}/({\prod_{k=1}^d J_k})$ for all $\bj\big\}$, where $a_1$ and $a_2$ are respectively lower and upper bounds of the density $g$. If $n/(\prod_{k=1}^d J_k)\gg \log(n)$,
	then $\P_0(E_n)\to 1$.
\end{lemma} 

\begin{proof}
	This can be shown with the same lines of argument of the proof of Lemma A.2 of \cite{chakraborty2020coverage} by replacing $J$ there with $\prod_{k=1}^d J_k$ and noting that $n/\prod_{k=1}^d
	J_k\gg \log(n) \gtrsim \log(\prod_{k=1}^d J_k)$.
\end{proof}

\begin{lemma}
	\label{lemma:sigmasq}
	Under the condition of Theorem 4.1, $\hat{\sigma}_n^2$ converges to $\sigma_0^2$ in probability at the rate $\max\{n^{-1/2}, n^{-1}\prod_{k=1}^d J_k, \max( J_k^{-1}: 1\le k \le d)\}$.
\end{lemma}

\begin{proof}
	$|\hat{\sigma}^2_n - \sigma^2_0|$ is bounded by up to a constant multiple of
	\begin{equation}\label{formula:sigmasqdecomp}
	\begin{aligned}
	    &\big|\frac{1}{n}\sum_{i=1}^n \varepsilon^2_i - \sigma^2_0\big|
	    + \frac{1}{n}\sum_{i=1}^n(f_0(\bX_i)-\overline{f_0(\bX)}|_{I_{\bj}})^2\\
	    &+\frac{1}{n}\sum_{\bj\in[\bm{1}:\bm{J}]}\frac{\lambda_{\bj}^{-2}N_{\bj}}{N_{\bj}+\lambda_{\bj}^{-2}}(\overline{f_0(\bX)}|_{I_{\bj}} - \zeta_{\bj})^2
	    \\
	    &+\big|\frac{1}{n} \sum_{\bj\in[\bm{1}:\bm{J}]}\frac{N_{\bj}\bar{\varepsilon}|_{I_{\bj}}}{N_{\bj}+\lambda_{\bj}^{-2}}(\overline{f_0(\bX)}|_{I_{\bj}} - \zeta_{\bj})\big|
	     + \frac{1}{n} \sum_{\bj\in[\bm{1}:\bm{J}]}{N_{\bj}(\bar{\varepsilon}|_{I_{\bj}})^2}.
	\end{aligned}
	\end{equation}
	The first term of \eqref{formula:sigmasqdecomp} is $O_{\P_0}(n^{-1/2})$. By the monotonicity of $f_0$, the second term is bounded by $n^{-1}\sum_{\bj\in[\bm{1}:\bJ]}N_{\bj}(f_0(\bj/\bJ) - f_0((\bj-1)/\bJ))^2$. 
Because  $|f_0(\bx)|\le \max\{f_0(\bm{0}), f_0(\bm{1})\}$ for every $\bx\in[0,1]^d$, 
	on the event $E_n$ defined in Lemma \ref{lemma:Nj} with probability tending to $1$,  the last display is further bounded by a multiple of $\prod_{k=1}^d J_{k}^{-1}\cdot \sum_{\bj\in[\bm{1}:\bJ]} f_0(\bj/\bJ) - f_0((\bj-1)/\bJ)$.
	Note that $[\bm{1}:\bm{J}]$ can be partitioned into no more than $\sum_{k=1}^d\prod_{p\neq k} J_p$ subsets $\{A_q\}_q$, where each $A_q$ is the largest possible set such that for every pair of $\bj_1,\bj_2\in A_q$, there exists an integer $a$ such that $\bj_1 = \bj_2 + a\bm{1}$. Thus the expression is bounded by
	$$
	    \prod_{k=1}^d J_{k}^{-1} \sum_{k=1}^d\prod_{p\neq k} J_p |f_0(\bm{1})-f_0(\bm{0})|\lesssim \max\{ J_k^{-1}: 1\leq k \leq d\},
$$
	since $J_k\to \infty$ for every $k$. Hence, the second term in \eqref{formula:sigmasqdecomp} is $O_{\P_0}(\max_k J_k^{-1})$.
	On the event $E_n$, the third term is bounded by a constant multiple of $\prod_{k=1}^d J_k/n$
	since the hyperparameters, $\zeta_{\bj}$ and $\lambda_{\bj}$, and the regression function $f_0$ are bounded.
	Noting that Var$(\bar{\varepsilon}|_{I_{\bj}}|\bX)=\sigma_0^2/N_{\bj}$ and using Lemma \ref{lemma:Nj}, the fourth term is $O_{\P_0}(\prod_{k=1}^d J_k / n)$. 
	The expectation of the last term is $O(\prod_{k=1}^d J_k / n)$.
It follows that the last term is $O_{\P_0}(\prod_{k=1}^dJ_k/n)$. 
\end{proof}

\begin{lemma}
	\label{lemma:submartingale}
	Let $\cJ=\{\cJ_1\times\cdots\times \cJ_m \}\subseteq \NN^m$ and $\{S_{\bj}, \bj\in \cJ\}$ be a collection of random variables. Assume that for every $1\leq k\leq m$ and every $(j_1,\ldots,j_{k-1},j_{k+1},\ldots,j_m)$,  $\{S_{(j_1,\ldots,j_{k-1},j,j_{k+1},\ldots,j_m)} , \mathscr{F}_j^{(k)}, j\in\cJ_k\}$ is a martingale. Then for $p > 1$, we have that 
	\begin{align*}
	\P(\max_{\bj\in\cJ} |{S_{\bj}}| > \epsilon) &\leq \frac{(p/(p-1))^{p(m-1)}}{\epsilon^p}  \E|{S_{(n_1,\ldots,n_m)}}|^{p};\\
	\E(\max_{\bj\in\cJ} \abs{S_{\bj}}^p ) &\leq (p/(p-1))^{pm} \E|{S_{(n_1,\ldots,n_m)}}|^{p}.
	\end{align*}
\end{lemma}

\begin{proof} 
	The proof is adapted from Lemma 1 of \cite{smythe1974sums}.
	Without loss of generality, we assume that $\cJ_k= \{1, 2, \cdots, n_k\}$ for every $k\leq m$. Define
	\begin{gather*}
	\tau_1 = \inf\{t_1: \max \{\abs{S_{(t_1, j_2, \ldots, j_m)}}: j_k \in \cJ_k\text{ for }k\geq 2\} >\epsilon \},\\
	\tau_2 = \inf\{t_2: \max \{\abs{S_{(\tau_1, t_2, \ldots, j_m)}}: j_k \in \cJ_k\text{ for }k\geq 3\} >\epsilon \},\\
	\vdots \\
	\tau_m = \inf\{t_m: \max \{\abs{S_{(\tau_1, \ldots, \tau_{m-1}, j_m)}}: j_m \in \cJ_m\} >\epsilon \},
	\end{gather*}
	where $\inf \varnothing =\infty$ by convention. Then $\{ \max\{ \abs{S_{\bj}}: \bj\in\cJ\} > \epsilon \} = \cup_{\bj\in\cJ} \{\bm{\tau} = \bj\}$, where $\bm{\tau} = (\tau_1,\ldots,\tau_m)\trans$. 
	
	As $\{\abs{S_{(j,j_2,\ldots,j_m)}}^p , \mathscr{F}_j^{(1)}, j\in\cJ_1\}$ is a nonnegative submartingale, $\E(|{S_{\bj}}|^p \Ind_{\{\bm{\tau}=\bj\}})$ can be bounded by 
	\begin{align*}
	 \E( \E(\abs{S_{(n_1,j_2,\ldots,j_m)}}^p | \mathscr{F}_{j_1}^{(1)})\Ind_{\{\bm{\tau}=\bj\}})=\E(\abs{S_{(n_1,j_2,\ldots,j_m)}}^p\Ind_{\{\bm{\tau}=\bj\}}),
	\end{align*}
for every $\bj=(j_1,\ldots, j_m)\in\cJ$.
Hence it follows that 
	\begin{multline*}
	\P(\max_{\bj\in\cJ} \abs{S_{\bj}} > \epsilon) \leq \epsilon^{-p}\sum_{j\in\cJ}\E(\abs{S_{\bj}}^p \Ind_{\{\bm{\tau}=\bj\}}) \\ \leq \epsilon^{-p}\sum_{j\in\cJ}\E(\abs{S_{(n_1,\tau_2,\ldots,\tau_m)}}^p \Ind_{\{\bm{\tau}=\bj\}})
	\leq \epsilon^{-p} \E(\abs{S_{(n_1,\tau_2,\ldots,\tau_m)}}^p).
	\end{multline*}
	$\{\max\{\abs{S_{(n_1,j,j_3,\ldots,j_m)}}^p: j_k\in\cJ_k, k\geq 3\} , \mathscr{F}_j^{(2)}, j\in\cJ_2\}$ is a nonnegative submartingale since $\{\abs{S_{(n_1,j,j_3,\ldots,j_m)}}^p , \mathscr{F}_j^{(2)}, j\in\cJ_2\}$ is a nonnegative submartingale. Hence by Doob's inequality, 
	\begin{align*}
	\E(\abs{S_{(n_1,\tau_2,\ldots,\tau_m)}}^p) 
	&\leq \E(\max_{j\in \cJ_2}\{\max\{\abs{S_{(n_1,j,j_3,\ldots,j_m)}}^p: j_k\in\cJ_k, k\geq 3\}\}) \\
	&\leq \big(\frac{p}{p-1}\big)^p \E(\max\{|{S_{(n_1,n_2,j_3,\ldots,j_m)}}|^p: j_k\in\cJ_k, k\geq 3\}).
	\end{align*}
	Using Doob's inequality repeatedly in the subsequent coordinates, we have 
	\begin{align*}
	\E(\abs{S_{(n_1,\tau_2,\ldots,\tau_m)}}^p)
	\leq \big(\frac{p}{p-1}\big)^{p(m-1)} \E(|{S_{(n_1,\ldots,n_m)}}|^p),
	\end{align*}
	which gives the first inequality. The second one follows from the similar argument above. 
\end{proof}

\begin{lemma}[Lemma C.7 of \cite{han2020}]
	\label{lemma:delta}
	Let $g:[0,c]^d\times[0,c]^d \to \RR,\, g(\bu,\bv)=\prod_{k=1}^d(u_k+v_k)$. Then
	\begin{align*}
	\abs{g(\bu,\bv)-g(\bu',\bv')}\leq (2c)^d\sqrt{d}(\|\bu-\bu'\|+\|\bv-\bv'\|).
	\end{align*}
\end{lemma}

\begin{lemma}[Lemma C.5 of \cite{han2020}]
	\label{lemma:c5}
	Let $\mathbb{G}_n$ be the empirical measure with respect to $G$.
	Under Assumption 1 and Assumption 2, 
	for some $c_0\geq 1$ and $\bm{l}\in L^{\ast}$, we can find a sequence $u_n\downarrow 0$ such that with probability at least $1-O(n^{-2})$, 
	\begin{multline*}
	    \sup_{\substack{c_0^{-1}\bm{1}\preceq \bu \preceq c_0\bm{1} \\ \bv \succeq \bm{0}}}
	    \left(\omega_n \prod_k(u_k^{l_k}\vee v_k^{l_k}) \right)^{-1} \\ \cdot \Big|\frac{\mathbb{G}_n(\bX-\bx_0)^{\bm{l}}\Ind_{[\bx_0 - \bu\circ \bm{r}_n, \bx_0 + \bv\circ \bm{r}_n]} }{\mathbb{G}_n \Ind_{[\bx_0 - \bu\circ \bm{r}_n, \bx_0 + \bv\circ \bm{r}_n]} }  - \frac{{G}_n(\bX-\bx_0)^{\bm{l}}\Ind_{[\bx_0 - \bu\circ \bm{r}_n, \bx_0 + \bv\circ \bm{r}_n]} }{{G}_n \Ind_{[\bx_0 - \bu\circ \bm{r}_n, \bx_0 + \bv\circ \bm{r}_n]} } \Big|,
	\end{multline*}
	is bounded from above by $u_n$.
\end{lemma}

\begin{lemma}[Lemma C.6 of \cite{han2020}]\label{lemma:c6}
	Let $(a, b, \gamma)$ be such that $a>1$, $0<b<\gamma<b+(a+1)$. Let $R_{a,b,\gamma}(c)$ be defined as in Lemma 6.7 and $H_l(\bu,\bv)$ as in Theorem 4.1, for $l=1,2$. Then there exists some positive constant $C$ depending on $d,a$ and $\sigma_0$ such that for any $c>1$ and $l=1,2$,
	\begin{align*}
	    \E[\sup_{(\bm{u},\bm{v})\in R_{a,b,\gamma}}\abs{H_l(\bm{u},\bm{v})-H_l(\bu,\bv\Ind_{[s+1:d-1]})}] \leq C\sqrt{c^{as-\gamma-a\Ind\{s=d\}}\log c}.
	\end{align*}
\end{lemma}

\begin{lemma}[Lemma C.8 of \cite{han2020}]\label{lemma:c8}
	Let $H_l(\bu,\bv)$  be as in Theorem 4.1, for $l=1,2$. Then for any $\epsilon>0$, there exists $\rho_{\epsilon}>0$ such that
	\begin{align*}
	    \P \big(\min_{\substack{0\leq v_k \leq \Ind_{s+1\leq k\leq d}\\ v_d =0}} \max_{\substack{0\leq u_k \leq  \Ind_{1\leq k\leq d}}} H_{l}(\bm{u},\bm{v})\ge \rho_{\epsilon} \big)\ge 1-\epsilon.
	\end{align*}
\end{lemma}

\end{document}